\documentclass[12pt]{article}

\usepackage{url}
\usepackage{fullpage}
\usepackage{here}
\usepackage{cite}            
\usepackage{graphicx}          
\usepackage{slashbox}                               
\usepackage[T1]{fontenc}
\usepackage{grffile}

\usepackage{mathtools}

\usepackage{amssymb}

\usepackage{amsmath}
\usepackage{graphicx}
\usepackage{amsfonts}
\usepackage{multirow}
\usepackage{enumerate}

\usepackage{color}
\usepackage{mathrsfs}
\usepackage{dsfont}

\providecommand{\blue}[1]{\color{black}{#1}\color{black}\hspace{0pt}}

\everymath{\displaystyle}

\newtheorem{theorem}{Theorem}[section]
\newtheorem{corollary}[theorem]{Corollary}
\newtheorem{proposition}[theorem]{Proposition}
\newtheorem{lemma}[theorem]{Lemma}
\newtheorem{define}[theorem]{Definition}

\newtheorem{remark}[theorem]{Remark}
\newtheorem{example}[theorem]{Example}

\newcommand{\mendth}{\hfill \ensuremath{\vartriangle}}

\DeclareMathOperator*{\col}{col}
\DeclareMathOperator*{\diag}{diag}
\DeclareMathOperator*{\T}{\intercal}

\DeclareMathOperator{\eps}{\varepsilon}

\def\ds{\textnormal{d}s}
\def\dt{\textnormal{d}t}

\def\d{\textnormal{d}}
\def\Tmin{T_\textnormal{min}}
\def\Tmax{T_\textnormal{max}}
\def\T{\textnormal{T}}

\def\tk{\ensuremath \{t_k\}_{k\in\mathbb{Z}_{\ge0}}}

\newenvironment{proof}{{\it Proof :~}}{\hfill$\diamondsuit$\\}

\title{Hybrid $L_\infty\times\ell_\infty$-Performance Analysis and Control of Linear Time-Varying Impulsive and Switched Positive Systems}

\author{Corentin Briat\thanks{email: {\tt  corentin@briat.info}; url:~{\tt http://www.briat.info}}}

\date{}

\begin{document}

\maketitle
\vspace{-1cm}
\begin{abstract}
\noindent Recent works have shown that the $L_1$ and $L_\infty$-gains are natural performance criteria for linear positive systems as they can be characterized using linear programs. Those performance measures have also been extended to linear positive impulsive and switched systems through the concept of hybrid $L_1\times\ell_1$-gain. For LTI positive systems, the $L_\infty$-gain is known to coincide with the $L_1$-gain of the transposed system and, as a consequence, one can use linear copositive Lyapunov functions for characterizing the $L_\infty$-gain of LTI positive systems. Unfortunately, this does not hold in the time-varying setting and one cannot characterize the hybrid $L_\infty\times\ell_\infty$-gain of a linear positive impulsive system in terms of the hybrid $L_1\times\ell_1$-gain of the transposed system. An alternative approach based on the use of linear copositive max-separable Lyapunov functions is proposed. We first prove very general necessary and sufficient conditions characterizing the exponential stability and the $L_\infty\times\ell_\infty$- and $L_1\times\ell_1$-gains using linear max-separable copositive and linear sum-separable copositive Lyapunov functions. These two results are then connected together using operator theoretic results  and the notion of adjoint system. Results characterizing the stability and the hybrid $L_\infty\times\ell_\infty$-gain of linear positive impulsive systems under arbitrary, constant, minimum, and range dwell-time constraints are then derived from the previously obtained general results. These conditions are then exploited to yield constructive convex stabilization conditions via state-feedback. By reformulating linear  positive switched systems as impulsive systems with multiple jump maps, stability and stabilization conditions are also obtained for linear positive switched systems. It is notably proven that the obtained conditions generalize existing ones of the literature. As all the results are stated as infinite-dimensional linear programs, sum of squares programming is used to turn those optimization problems into sufficient tractable finite-dimensional semidefinite programs. Interestingly, the relaxation becomes necessary if we allow the degrees of the polynomials to be arbitrarily large. Several particular cases of the approach such as LTV positive systems and periodic positive systems are also discussed for completeness. Examples are given for illustration.\\

\end{abstract}
\section{Introduction}

Linear positive systems \cite{Farina:00} have recently attracted a lot of attention because of their surprising properties. Several problems known to be NP-Hard, in general, turn out to be, in fact, deceptively simple in the context of linear positive systems. Such problems include the design of structured controllers with and without prescribed bounds \cite{Briat:11h}, the exact computation of gains such as the $L_1$-, $\ell_1$, $L_\infty$- and $\ell_\infty$-gains \cite{Briat:11h,Briat:11g}, the exact stability analysis of uncertain systems \cite{Ebihara:11,Briat:11h,Colombino:15} and systems with delays \cite{Kaczorek:09,Shen:15,Shen:15b,Briat:16b,Zhu:17b}, etc. Positive systems can also be used to model a wide variety of real-world systems and have found application in systems biology \cite{Briat:12c,Briat:13h,Briat:13i,Blanchini:14,Briat:15e,Briat:16a}, biochemistry \cite{Hynn:01}, physiology \cite{Linninger:09}, epidemiology and disease dynamics \cite{Murray:02,Jonsson:14,Jonsson:16}, etc. They are sometimes used as comparison systems in order to establish simple and potentially scalable stability results  (see e.g. \cite{Ai:17,Sootla:17cdc} and the references therein)  and they have also been shown to play an instrumental role in the design of so-called interval observers; see e.g. \cite{Gouze:00,Briat:15g,Efimov:16c,Chambon:16,Briat:18:IntImp,Briat:18:IntMarkov,Briat:19:IJC} and the references therein. Linear positive impulsive systems have been recently studied in \cite{Zhang:14b,Briat:16c,Briat:18:IntImp,Briat:19:IJC}; see also the references therein. In particular, the papers \cite{Briat:18:IntImp,Briat:19:IJC} deal with the stability and hybrid $L_1\times\ell_1$-performance analysis of such systems under minimum and range dwell-time constraints. Linear positive switched systems have also been considered in \cite{Blanchini:15,Briat:18:IntImp}.

The objective of this paper is to develop a theory for the stability and hybrid $L_\infty\times\ell_\infty$ performance analysis, and the state-feedback control of linear positive impulsive systems. In \cite{Briat:11g,Briat:11h}, it was demonstrated that, in the case of LTI positive systems, the $L_\infty$-gain ($\ell_\infty$-gain) of an LTI continuous-time (discrete-time) positive system is equal to the $L_1$-gain ($\ell_1$-gain) of the so-called \emph{transposed system}. This is an immediate consequence of the explicit formulas for such gains; see e.g. \cite{Desoer:75a}. Interestingly, it was shown in \cite{Briat:11g,Briat:11h} that computing the $L_1$-gain or the $\ell_1$-gain amounts to solving a simple linear program. Unfortunately, this is only possible in the time-invariant setting and this cannot be directly generalized to more general systems such as linear time-varying systems, switched systems or impulsive systems. In this regard, we cannot use the approach developed in \cite{Briat:18:IntImp,Briat:19:IJC} to our current problem and novel one is required.

The proposed approach consists of first proving very general necessary and sufficient conditions, inspired from the ideas in \cite{Blanchini:15}, characterizing the exponential stability as well as the hybrid  $L_\infty\times\ell_\infty$ or the $L_1\times\ell_1$ performance of an impulsive system for a given and fixed sequence of impulse times. Following the terminology introduced in \cite{Dashkhovskiy:11,Ito:12} in the context of monotone systems, the former result is based on the use of a max-separable linear copositive Lyapunov function whereas the latter is based on the use of a linear sum-separable copositive Lyapunov function. There is an interesting connection between those results, the class of Lyapunov function and the type of associated performance: on one side we have a sum- separable Lyapunov function for the hybrid $L_1\times\ell_1$ performance (which involves the sum of inputs and outputs) while we have a max-separable Lyapunov function for the hybrid $L_\infty\times\ell_\infty$ performance (which involves the maximum of inputs and outputs). This is due to the fact that the dual norm of the vector 1-norm, directly associated with the linear sum-separable copositive Lyapunov function, is dual to the vector $\infty$-norm, which is associated with the max-separable linear copositive Lyapunov function. Similarly, the dual of the  $L_1\times\ell_1$-norm is the $L_\infty\times\ell_\infty$-norm. We use these facts to connect the two previously obtained results all together through the concept of adjoint system and the use of standard operator theoretic results. The resulting result states that a linear positive impulsive system is exponentially stable and has an $L_\infty\times\ell_\infty$-gain equal to $\gamma$ if and only if the adjoint system is backward dissipative with a linear sum-separable copositive storage function with the supply-rate characterizing the hybrid $L_1\times\ell_1$ performance. This result is the time-varying counterpart of the result based on the transposed system only valid in the LTI case. In particular, when the system is made time-invariant, the conditions reduce to those based on the transposed system.

Based on those very general results, stability and performance analysis conditions can be easily obtained. Arbitrary, constant, range and minimum dwell-times constraints are considered. The results are notably specialized to timer-dependent impulsive systems, a class of systems often considered in the literature as it arises naturally from certain control and observation problems; see e.g. \cite{Allerhand:11,Briat:13d,Briat:14f,Briat:15i}. An important advantage of this class of systems is that the stability conditions, expressed as infinite-dimensional linear programs, are tractable as the timer variable takes values within a compact set. This makes the verification of those conditions a much simpler task as they may be tackled using polynomial optimization techniques such as sum of squares programming \cite{Parrilo:00,sostools3}. Even though the sum of square relaxation is sufficient only, we prove that it is also necessary provided that the degrees of the polynomials can be chosen arbitrarily large.

Stabilization conditions are then derived. It turns out that even when the matrices of the system are time-invariant, timer-dependent and dwell-time-dependent controllers need to be considered in order to obtain convex design conditions, a crucial point which justifies the consideration of timer-dependent systems. Using a linear analogous of Finsler's Lemma, we also provide conditions where part of the designed controllers is time-invariant, possibly at the expense of some additional conservatism.

To illustrate the generality of the approach, the results are then specialized to linear positive switched systems albeit only the minimum dwell-time constraint is considered. We also provide conditions for subclasses of such systems such as linear time-varying, periodic and linear time-invariant systems in both the continuous- and discrete-time.\\

\noindent\textbf{Outline.} The general stability and performance results derived in Section \ref{sec:fund} serve as the basis of the specialized results in Section \ref{sec:Linf_imp_stab}. Stabilization results are obtained in the following section, Section \ref{sec:Linf_imp_stabz}, which are then specialized to positive switched systems in Section \ref{sec:Linf_sw}. Finally, interesting particular cases of the proposed approach are discussed in Section \ref{sec:part}. Illustrative examples are provided in the relevant sections.\\

\noindent\textbf{Notations.} The set of integers greater or equal to $n\in\mathbb{Z}$ is denoted by $\mathbb{Z}_{\ge n}$. The cones of positive and nonnegative vectors of dimension $n$ are denoted by $\mathbb{R}_{>0}^n$ and $\mathbb{R}_{\ge0}^n$, respectively.  For any matrix $M$, the inequalities $M\ge0$ and $M>0$ are always interpreted componentwise. The set of diagonal matrices of dimension $n$ is denoted by $\mathbb{D}^n$ and the subset of those being positive definite is denoted by $\mathbb{D}_{\succ0}^n$. The $n$-dimensional vector of ones is denoted by $\mathds{1}_n$. The dimension will be often omitted when its value is obvious from the context. For some elements, $\{x_1,\ldots,x_n\}$, the operator $\textstyle \diag_{i=1}^n(x_i)$ builds a matrix with diagonal entries given by $x_1,\ldots,x_n$ whereas $\textstyle\col_{i=1}^n(x_i)$ creates a vector by vertically stacking them with $x_1$ on the top. The limit from the right of a function $f$ at some point $t$ is denoted by $\textstyle f(t^+):=\lim_{s\downarrow t}f(s)$. The natural basis for $\mathbb{R}^n$ is denoted by $\{e_1,\ldots,e_n\}$. We also use the shorthand $\partial_s$ for the partial derivative operator $\frac{\partial}{\partial s}$.

\section{General stability and hybrid performance results for LTV positive impulsive systems}\label{sec:fund}

General stability results for LTV impulsive systems are obtained in this section using max-separable linear copositive Lyapunov functions or sum-separable linear copositive Lyapunov functions. These results are extended to characterize the hybrid $L_\infty\times\ell_\infty$-gain and the hybrid $L_1\times\ell_1$-gain, respectively. A link connecting those results using operator theoretic ideas is also provided.

\subsection{Preliminaries}

We consider in this paper linear time-varying impulsive systems of the form
\begin{equation}\label{eq:mainsyst}
\begin{array}{rcl}
  \dot{x}(t)&=&\tilde{A}(t)x(t)+\tilde{B}_c(t)u_c(t)+\tilde{E}_c(t)w_c(t),t\in(t_k,t_{k+1}],k\in\mathbb{Z}_{\ge0}\\
  x(t_k^+)&=&\tilde{J}(k)x(t_k)+\tilde{B}_d(k)u_d(k)+\tilde{E}_d(k)w_d(k),k\in\mathbb{Z}_{\ge1}\\
  z_c(t)&=&\tilde{C}_c(t)x(t)+\tilde{D}_c(t)u_c(t)+\tilde{F}_c(t)w_c(t),t\ge0\\
  z_d(k)&=&\tilde{C}_d(k)x(t_k)+\tilde{D}_d(k)u_d(k)+\tilde{F}_d(k)w_d(k),k\in\mathbb{Z}_{\ge0}\\
  x(t_0)&=&x(t_0^+)=x_0
\end{array}
\end{equation}
where $x(t),x_0\in\mathbb{R}^n$, $u_c(t)\in\mathbb{R}^{m_c}$, $w_c(t)\in\mathbb{R}^{p_c}$, $u_d(k)\in\mathbb{R}^{m_d}$, $w_d(k)\in\mathbb{R}^{p_d}$, $z_c(t)\in\mathbb{R}^{q_c}$ and $z_d(k)\in\mathbb{R}^{q_d}$, $t\in\mathbb{R}_{\ge0}$, $k\in\mathbb{Z}_{\ge0}$ are the state of the system, the initial condition, the continuous-time exogenous input, the continuous-time control input, the discrete-time exogenous input, the discrete-time control input, the continuous-time performance output and the discrete-time performance output, respectively. The matrix-valued functions $\tilde{A}(t)\in\mathbb{R}^{n\times n}$, $\tilde{E}_c(t)\in\mathbb{R}^{n\times p_c}$, $\tilde{C}_c(t)\in\mathbb{R}^{q_c\times n}$ and $\tilde{F}_c(t)\in\mathbb{R}^{q_c\times p_c}$ are assumed to be continuous and bounded for all $t\in\mathbb{R}_{\ge0}$. The matrices $\tilde{J}(k)\in\mathbb{R}^{n\times n}$, $\tilde{E}_d(k)\in\mathbb{R}^{n\times p_d}$, $\tilde{C}_d(k)\in\mathbb{R}^{q_d\times n}$ and $\tilde{F}_d(k)\in\mathbb{R}^{q_c\times p_c}$ are bounded matrices for all $k\in\mathbb{Z}_{\ge0}$. The sequence of impulse times $\{t_k\}_{k\in\mathbb{Z}_{\ge1}}$ is assumed to verify the properties: (a) $T_k:= t_{k+1}-t_k>0$ for all $k\in\mathbb{Z}_{\ge1}$, and (b) $\textstyle t_k\to\infty$ as $k\to\infty$. \blue{In the following, we will denote by $\frak{T}$ a sequence $\{t_1,t_2,\ldots\}$ of impulse times that belongs to some family $\bar{T}$. We, moreover, define the sequence of impulse times relative to $t_0$ as $\frak{T}_{t_0}:=\{t_0,t_1,\ldots\}$ for all $\frak{T}\in\mathbb{T}$ where $0\le t_0< t_1$. We denote by $\mathbb{T}_0$ the set of all possible $(t_0,\frak{T})$ where $\frak{T}\in\mathbb{T}$. It is given by $\mathbb{T}_0=\{(\tau,\frak{T}):[0,\min\{\frak{T}\}),\frak{T}\in\mathbb{T}\}$. It it possible to extend this definition to the case where $t_0\ge t_1$ at the expense of a dramatic increase of notational complexity. This is the reason why we will only consider the case $0\le t_0< t_1$ in the current paper.}

We define hybrid the $L_1\times\ell_1$- and $L_\infty\times\ell_\infty$-norm as follows:
\begin{define}
  Let us consider a hybrid signal $w:=(w_c,w_d)$ where $w_c:\mathbb{R}_{\ge0}\mapsto\mathbb{R}^{p_c}$ and $w_d:\mathbb{Z}_{\ge0}\mapsto\mathbb{R}^{p_d}$. Then,
  \begin{enumerate}[(a)]
    \item the hybrid $L_\infty\times\ell_\infty$-norm is defined as
  \begin{equation}
    \left|\left|\begin{bmatrix}
      w_c\\
      w_d
    \end{bmatrix}\right|\right|_{L_\infty\times\ell_\infty}:=\left|\left|\begin{bmatrix}
      ||w_c||_{L_\infty}\\
      ||w_d||_{\ell_\infty}
    \end{bmatrix}\right|\right|_\infty=\max\{||w_c||_{L_\infty}, ||w_d||_{\ell_\infty}\},
  \end{equation}
  \item the hybrid $L_1\times\ell_1$-norm is defined as
  \begin{equation}
    \left|\left|\begin{bmatrix}
      w_c\\
      w_d
    \end{bmatrix}\right|\right|_{L_1\times\ell_1}:=\left|\left|\begin{bmatrix}
      ||w_c||_{L_1}\\
      ||w_d||_{\ell_1}
    \end{bmatrix}\right|\right|_1= ||w_c||_{L_1}+||w_d||_{\ell_1}.
  \end{equation}
  \end{enumerate}
\end{define}

Based on the above definition, we are now able of clearly defining the concept of hybrid gain for a hybrid operator:
\begin{define}
  Let us consider a time-varying hybrid system represented by the bounded operator $\Sigma_{t_0,\frak{T}}:L_p\times\ell_p\mapsto L_p\times\ell_p$ where $\frak{T}$ is the sequence of impulse times and $t_0\ge0$ is the initial time such that $(t_0,\frak{T})\in\mathbb{T}_0$. Then, the $L_p\times\ell_p$-gain of $\Sigma_{t_0,\frak{T}}$, $p\in\mathbb{Z}_{\ge1}$, $(t_0,\frak{T})\in\mathbb{T}$, denoted by $||\Sigma_{t_0,\frak{T}}||_{\mathbb{T},L_p\times\ell_p}$, is defined as
  \begin{equation}
    ||\Sigma_{t_0,\frak{T}}||_{\mathbb{T}_0,L_p\times\ell_p}:=
\sup_{{\scriptsize\begin{array}{c}
||(w_c,w_d)||_{L_p\times\ell_p}=1\\
  x_0=0, (t_0,\mathfrak{T})\in\mathbb{T}_0
\end{array}}} \left|\left|\Sigma_{\frak{T}_{t_0}}\begin{pmatrix}
      w_c\\w_d
    \end{pmatrix}\right|\right|_{L_p\times\ell_p}.
  \end{equation}
  It can be alternatively defined as
  \begin{equation}
    ||\Sigma_{t_0,\frak{T}}||_{\mathbb{T}_0,L_p\times\ell_p}:=\inf_{\gamma>0}\gamma \textnormal{ such that }  \left|\left|\Sigma_{t_0,\frak{T}}\begin{pmatrix}
      w_c\\w_d
    \end{pmatrix}\right|\right|_{L_p\times\ell_p}<\gamma \left|\left|\begin{bmatrix}
      w_c\\
      w_d
    \end{bmatrix}\right|\right|_{L_p\times\ell_p}+v(||x_0||_p)
  \end{equation}
  holds for all  $(w_c,w_d)\in L_p\times \ell_p$, all $(t_0,\frak{T})\in\mathbb{T}_0$, and for some $v(\cdot)$ such that $v(0)=0$, $v(s)>0$ for all $s>0$, and $v(s)\to\infty$ as $s\to\infty$.
\end{define}

We have the following known result regarding the internal positivity of the impulsive system \blue{\eqref{eq:mainsyst}}:
\begin{proposition}[\cite{Briat:18:IntImp,Briat:19:IJC}]\label{prop:positive}
The following statements are equivalent:
\begin{enumerate}[(a)]
  \item The system \eqref{eq:mainsyst} with $u_c,u_d\equiv0$ is internally positive, i.e. for any sequence $\{t_k\}_{k\in\mathbb{Z}_{\ge0}}$, any initial time $t_0\ge0$, any initial condition $x_0\ge0$,  and any nonnegative inputs $w_c(t)\ge0$ and $w_d(k)\ge0$, we have that $x(t)\ge0$, $z_c(t)\ge0$ and $z_d(k)\ge0$ for all $t\in\mathbb{R}_{\ge t_0}$ and all $k\in\mathbb{Z}_{\ge0}$.
  \item For all $t\in\mathbb{R}_{\ge0}$, the matrix-valued function $\tilde A(t)$ is Metzler\footnote{A square matrix is Metzler if all its off-diagonal entries are nonnegative.} and the matrix-valued functions $\tilde{E}_c(t),\tilde{C}_c(t)$ and $\tilde{F}_c(t)$ are nonnegative\footnote{A matrix is nonnegative if all its entries are nonnegative.}, and for all $k\in\mathbb{Z}_{\ge0}$ the matrices $\tilde{J}(k),\tilde{E}_d(k),\tilde{C}_d(k)$ and $\tilde{F}_d(k)$ are nonnegative.
\end{enumerate}
\end{proposition}
\begin{proof}
The proof follows from a mixture of the arguments used in the positivity analysis of continuous- and discrete systems. It is therefore omitted.
\end{proof}

\blue{It is also convenient to define the state-transition matrix associated with the system \eqref{eq:mainsyst}:
\begin{define}[State-transition matrix]
Let us consider a sequence of impulse times $\mathfrak{T}$ and an initial condition $t_0\ge0$ which gives $\frak{T}_{t_0}$ as the sequence of impulse times relative to $t_0$. The state-transition matrix $\Phi(t,s)$ associated with the system \eqref{eq:mainsyst} defined as
\begin{equation}
\begin{array}{rcl}
  \partial_t\Phi(t,s)&=&\tilde{A}(t)\Phi(t,s),\ t_k<s\le t\le t_{k+1}, k\in\mathbb{Z}_{\ge0}\\
  \Phi(t_k^+,s)&=&\tilde{J}(k)\Phi(t_k,s),\ s\le t_k,k\in\mathbb{Z}_{\ge1}\\
  \Phi(s,s)&=&I,\ s\ge0.
\end{array}
\end{equation}
verifies the following properties:
\begin{enumerate}[(a)]
  \item it is invertible on every interval $(t_k,t_{k+1}]$, $k\in\mathbb{Z}_{\ge0}$ and it is invertible on $(t_{k-1},t_{k+1}]$ if and only if $J(k)$ is invertible, and
  \item it verifies
  \begin{equation}
    \partial_s\Phi(t,s)=-\Phi(t,s)\tilde{A}(s)
  \end{equation}
  on every interval where it is invertible; i.e. $t_i<s\le t\le t_{i+1}$, $i\ge0$.
  \end{enumerate}
\end{define}}

\blue{This allows us to explicitly express the output of the system in terms of hybrid convolution operators:
\begin{proposition}[Input/output Operators]
Let us consider a sequence of impulse times $\mathfrak{T}$ and an initial condition $t_0\ge0$ which gives $\frak{T}_{t_0}$ as the sequence of impulse times relative to $t_0$. Assuming zero initial conditions, the outputs of the system \eqref{eq:mainsyst} can be written in terms of the operators $G_{cc}^{t_0,\frak{T}}$, $G_{cd}^{t_0,\frak{T}}$, $G_{cd}^{t_0,\frak{T}}$ and  $G_{dd}^{t_0,\frak{T}}$ as
  \begin{equation}
    \begin{bmatrix}
      z_c\\
      z_d
    \end{bmatrix}=\begin{bmatrix}
      G_{cc}^{t_0,\frak{T}} & G_{cd}^{t_0,\frak{T}}\\
      G_{dc}^{t_0,\frak{T}} & G_{dd}^{t_0,\frak{T}}
    \end{bmatrix}\begin{bmatrix}
      w_c\\
      w_d
    \end{bmatrix}
  \end{equation}
  where these operators depend on the dwell-times sequence; i.e. for two different sequences, these operators are different. They are given by sequence dependent
  \begin{equation}
    \begin{array}{rcl}
      (G_{cc}^{t_0,\frak{T}}w_c)(t)&=&\tilde C_c(t)\int_{t_0}^t\tilde{\Phi}(t,s)\tilde E_c(s)w_c(s)ds+\tilde F_c(t)w_c(t),\ t_k<s\le t\le t_{k+1}, k\in\mathbb{Z}_{\ge0}\\[1em]
      (G_{cd}^{t_0,\frak{T}}w_d)(t)&=&\tilde C_c(t)\sum_{i=0}^k\tilde{\Phi}(t,t_i^+)\tilde{E}_d(i)w_d(i),\ t_k<s\le t\le t_{k+1}, k\in\mathbb{Z}_{\ge0}\\[1em]
      (G_{dc}^{t_0,\frak{T}}w_c)(k)&=&\tilde C_d(k)\int_{t_0}^{t_k}\tilde{\Phi}(t_k,s)\tilde E_c(s)w_c(s)ds,\ k\ge0\\[1em]
      (G_{dd}^{t_0,\frak{T}}w_d)(k)&=&\tilde C_d(k)\sum_{i=0}^{k-1}\tilde{\Phi}(t_k,t_i^+)\tilde{E}_d(i)w_d(i)+\tilde{F}_d(k)w_d(k),\ k\ge0
    \end{array}
  \end{equation}
  where we have set $w_d(0)=0$, by definition.
\end{proposition}}

\blue{We also consider the following concept of hybrid uniform exponential stability:
\begin{define}\label{def:expstab}
  Assume that the set of initial times and impulse times sequence $\mathbb{T}_0$ is given. Then, the system said to be uniformly exponentially stable with hybrid rate $(\alpha,\rho)$ if there exist an $M>0$, a $\rho\in(0,1)$ and an $\alpha>0$ such that we have
  \begin{equation}
    ||x(t)||_p\le M\rho^{\kappa(t,t_0)}e^{-\alpha(t-t_0)}||x_0||_p
  \end{equation}
  for some $p\in\mathbb{Z}_{>0}$ and for all $(t_0,\frak{T})\in\mathbb{T}_0$, $t\ge t_0$, where $\kappa(t,s)$ denotes the number of jumps between $s\le t$. Equivalently,
  \begin{equation}
    ||\tilde{\Phi}(t,s)||_p\le M\rho^{\kappa(t,s)}e^{-\alpha(t-s)}
  \end{equation}
  for all $(t_0,\frak{T})\in\mathbb{T}_0$, $t\ge s\ge t_0$.
\end{define}
This definition combines both a continuous exponential decay-rate $\alpha$ together with a discrete-time decay rate $\rho$ to capture both the continuous-time behavior and the discrete behavior at jumps. Note also that the choice of the norm does not matter here as all the norms are equivalent in finite-dimensional spaces.}

\subsection{General stability and hybrid $L_\infty\times\ell_\infty$ performance results for linear positive impulsive systems}

The preliminary following result states a necessary and sufficient condition for an LTV impulsive system to be exponentially stable for a given sequence of impulse time instants using linear max-separable copositive Lyapunov functions:
\begin{lemma}\label{lem:general_max}
  Assume that the sequence of impulse times $\frak{T}=\{t_k\}_{k\in\mathbb{Z}_{\ge1}}$ is given. Then, the following statements are equivalent:
  \begin{enumerate}[(a)]
    \item The time-varying impulsive system
    \begin{equation}\label{eq:stabTV}
      \begin{array}{rcl}
        \dot{x}(t)&=&\tilde{A}(t)x(t),\ t\in(t_k,t_{k+1}],k\in\mathbb{Z}_{\ge0}\\
        x(t_k^+)&=&\tilde J(k)x(t_k),\ k\in\mathbb{Z}_{\ge1}\\
        x(t_0^+)&=&x(t_0)=x_0
      \end{array}
    \end{equation}
    is uniformly exponentially stable \blue{in the sense of Definition \ref{def:expstab}.}
    \item There exist positive vectors $\bar\xi_1,\bar\xi_2\in\mathbb{R}^n_{>0}$, $0<\bar\xi_1\le \bar\xi_2$,  and a continuously differentiable vector-valued function $\xi:[t_0,\infty)\mapsto\mathbb{R}^n$, verifying $\bar\xi_1<\xi(t)<\bar\xi_2$, $t\ge t_0$, such that the function
    \begin{equation}
    V(t,x(t))=\max_{i\in\{1,\ldots,n\}}\dfrac{x_i(t)}{\xi_i(t)}
    \end{equation}
    is a uniform Lyapunov function for the system \eqref{eq:stabTV}.
    \item For any $b,d\in\mathbb{R}^n_{>0}$, the differential-difference inequality
    \begin{equation}\label{eq:jdlsjdls}
      \begin{array}{rcl}
        -\dot{\xi}(t)+\tilde{A}(t)\xi(t)+b&<&0,\ t\in(t_k,t_{k+1}],k\in\mathbb{Z}_{\ge0}\\
        \tilde J(k)\xi(t_k)-\xi(t_k^+)+d&<&0,\ k\in\mathbb{Z}_{\ge1}
      \end{array}
    \end{equation}
    has a continuously differentiable positive solution $\xi(t)$ verifying $\bar\xi_1<\xi(t)<\bar\xi_2$ for some positive vectors $\bar\xi_1,\bar\xi_2\in\mathbb{R}^n_{>0}$, $0<\bar\xi_1\le \bar\xi_2$, for all $t\ge t_0$, and for all $0\le t_0<t_1$.
  \end{enumerate}
\end{lemma}
\begin{proof}
The proof of this result is inspired by \cite{Blanchini:15} for switched systems and extends it to the case of the impulsive systems. The proof that (b) implies (a) follows from Lyapunov theory.\\

\noindent\textbf{Proof that (c) implies (a) and (b).} Let us define the Lyapunov function $V(x,t)=\max_k\{x_k(t)/\xi_k(t)\}$. where $\xi(t)$ verifies \eqref{eq:jdlsjdls} and note that this function is absolutely continuous, hence differentiable almost everywhere. We also observe that $\textstyle\frac{||x||_\infty}{\max\{\bar{\xi}_2\}}\le V(x,t)\le\frac{||x||_\infty}{\min\{\bar{\xi}_1\}}$. Assume that $$V(x,t)=\dfrac{x_i(t)}{\xi_i(t)} \textnormal{ and } V(x,t^+)=\dfrac{x_j(t)}{\xi_j(t)},$$ then we have that
  \begin{equation}
      \dot{V}(x,t^+)    = \sum_{k=1}^n\tilde{A}_{jk}(t)\dfrac{x_k(t)}{\xi_j(t)}-\dfrac{x_j(t)}{\xi_j(t)^2}\dot{\xi}_j(t).
  \end{equation}
  Using \eqref{eq:jdlsjdls} yields
  \begin{equation}
    \begin{array}{rcl}
      \dot{V}(x,t^+)    &=&   \displaystyle\sum_{k=1}^n\tilde{A}_{jk}(t)\dfrac{x_k(t)}{\xi_j(t)}-\dfrac{x_j(t)}{\xi_j(t)^2}\dot{\xi}_j(t)\\
                                    &<&   \displaystyle\sum_{k=1}^n\tilde{A}_{jk}(t)\dfrac{x_k(t)}{\xi_k(t)}\dfrac{\xi_k(t)}{\xi_j(t)}-\left(\sum_{k=1}^n\tilde{A}_{jk}(t)\xi_k(t)+b_j\right)\dfrac{x_j(t)}{\xi_j(t)^2}\\
                                    &=&   \displaystyle\dfrac{1}{\xi_j(t)}\left[\sum^n_{k=1,\blue{k\ne j}}\tilde{A}_{jk}(t)\xi_k(t)\left(\dfrac{x_k(t)}{\xi_k(t)}-\dfrac{x_j(t)}{\xi_j(t)}\right)-\dfrac{x_j(t)}{\xi_j(t)}b_j\right]\\
                                    &<&   \displaystyle\dfrac{1}{\xi_j(t)}\left[\sum^n_{k=1,\blue{k\ne j}}\tilde{A}_{jk}(t)\xi_k(t)\left(\dfrac{x_i(t)}{\xi_i(t)}-\dfrac{x_j(t)}{\xi_j(t)}\right)-\dfrac{x_j(t)}{\xi_j(t)}b_j\right]\\
                                    &=&\displaystyle-\dfrac{x_j(t)}{\xi_j(t)^2}b_j\\
                                    &\le&-\dfrac{b_j}{e_j^{\T}\xi_{2,j}}V(x,t^+)\\
                                    &\le&-\dfrac{\min\{b\}}{\max\{\bar\xi_{2}\}}V(x,t^+)
    \end{array}
  \end{equation}
  \blue{where we have used the facts that the off-diagonal elements of $\tilde{A}(t)$ are nonnegative and that we have we have $x_i(t)/\xi_i(t)=x_j(t)/\xi_j(t)$ at time $t$.} Similarly, we have that
  \begin{equation}
    \begin{array}{rclcrcl}
      V(x,t_k)&=&\dfrac{x_i(t_k)}{\xi_i(t_k)}&\textnormal{and}&   V(x,t_k^+)&=&\dfrac{x_j(t_k^+)}{\xi_j(t_k^+)}.
    \end{array}
  \end{equation}
  Hence,
  \begin{equation}
  \begin{array}{rcl}
    V(x,t_k^+)-V(x,t_k) &=& \displaystyle\dfrac{x_j(t_k^+)}{\xi_j(t_k^+)}-\dfrac{x_i(t_k)}{\xi_i(t_k)}\\[1em]
                                        &<& \displaystyle\dfrac{\sum_{s=1}^n\tilde J_{js}(k)x_s(t_k)}{\sum_{s=1}^n\tilde J_{js}(k)\xi_s(t_k)+d_j}-\dfrac{x_i(t_k)}{\xi_i(t_k)}\\[1em]
                                        &=& \displaystyle\dfrac{\sum_{s=1}^n\tilde J_{js}(k)\xi_s(t_k)\frac{x_s(t_k)}{\xi_s(t_k)}}{\sum_{s=1}^n\tilde J_{js}(k)\xi_s(t_k)+d_j}-\dfrac{x_i(t_k)}{\xi_i(t_k)}\\[1em]
                                        &<& \displaystyle\dfrac{\sum_{s=1}^n\tilde J_{js}(k)\xi_s(t_k)+d_j-d_j}{\sum_{s=1}^n\tilde J_{js}(k)\xi_s(t_k)+d_j}\frac{x_i(t_k)}{\xi_i(t_k)}-\dfrac{x_i(t_k)}{\xi_i(t_k)}\\[1em]
                                        &=& \displaystyle\left(1-\dfrac{d_j}{\sum_{s=1}^n\tilde J_{js}(k)\xi_s(t_k)+d_j}\right)\frac{x_i(t_k)}{\xi_i(t_k)}-\dfrac{x_i(t_k)}{\xi_i(t_k)}\\[1em]
                                        &=& -\dfrac{d_j}{\sum_{s=1}^n\tilde J_{js}(k)\xi_s(t_k)+d_j}V(x,t_k)\\[1em]
                                        &<&-\dfrac{d_j}{\xi_j(t_k^+)}V(x,t_k)\\[1em]
                                        &<&-\dfrac{\min\{d\}}{\max\{\bar\xi_{2}\}}V(x,t_k)
  \end{array}
  \end{equation}
  where we have used the fact that $\textstyle\sum_{s=1}^n\tilde J_{js}(k)\xi_s(t_k)+d_j<\xi_j(t_k^+)$. Note also that $\min\{d\}/\max\{\bar \xi_2\}<1$ since $\bar\xi_2\ge \xi(t_k^+)>d$. \blue{Letting $\textstyle\alpha=\frac{\min\{b\}}{\max\{\bar\xi_{2}\}}$ and $\textstyle\rho=1-\frac{\min\{d\}}{\max\{\bar\xi_{2}\}}$, we get that
  \begin{equation}
    \begin{array}{rcl}
      \dot{V}(x,t)&\le&-\alpha V(x,t), t\in(t_k,t_{k+1}],\ k\in\mathbb{Z}_{\ge0}\\
      V(x,t_k^+)&\le&\rho V(x,t_k),\ k\in\mathbb{Z}_{\ge1}.
    \end{array}
  \end{equation}
  This implies that
  \begin{equation}
    V(x,t)\le V(x_0,t_0)\rho^{\kappa(t,t_0)}e^{-\alpha(t-t_0)},\ t\ge t_0
  \end{equation}
  which implies in turn
  \begin{equation}
    ||x(t)||_\infty\le\dfrac{\max\{\bar{\xi}_2\}}{\min\{\bar{\xi}_1\}}\rho^{\kappa(t,t_0)}e^{-\alpha(t-t_0)}||x_0||_\infty,\ t\ge t_0
  \end{equation}
  which proves the uniform exponential stability of the system \eqref{eq:mainsyst} in the sense of Definition \ref{def:expstab}; i.e. (a) and (c) hold.}\\

  \noindent\textbf{Proof that (a) implies (c).} \blue{Assume that for the given sequence of jump times the system is uniformly exponentially stable in the sense of Definition \ref{def:expstab}. Define further the function
  \begin{equation}\label{eq:xiLinfproof}
    \xi(t)=\int_{-\infty}^t\tilde{\Phi}(t,s)(b+\eps\mathds{1})\ds+\smashoperator{\sum_{t_1^+\le t_i^+\le t}}\tilde{\Phi}(t,t_i^+)(d+\eps\mathds{1}).
  \end{equation}
  We need to show first that this expression is well-defined and is uniformly bounded from above by some vector $\bar\xi_2>0$. From the definition of uniform exponential stability, i.e. Definition \ref{def:expstab}, we have that
  \begin{equation}
  \begin{array}{rcl}
    ||\xi(t)||_\infty&\le&\int_{-\infty}^tM\rho^{\kappa(t,s)}e^{-\alpha(t-s)}||b+\eps\mathds{1}||_\infty\ds+\smashoperator{\sum_{t_1^+\le t_i^+\le t}}Me^{-\alpha(t-t_i)}\rho^{\kappa(t,t_i)}||d+\eps\mathds{1}||_\infty\\
                &\le&\int_{-\infty}^tMe^{-\alpha(t-s)}||b+\eps\mathds{1}||_\infty\ds+\smashoperator{\sum_{t_1^+\le t_i^+\le t}}M\rho^{\kappa(t,t_i)}||d+\eps\mathds{1}||_\infty\\
                &\le&\int_{-\infty}^tMe^{-\alpha(t-s)}||b+\eps\mathds{1}||_\infty\ds+\sum_{k\ge0}^\infty M\rho^k||d+\eps\mathds{1}||_\infty\\
            &=&M\left(\dfrac{1}{\alpha}||b+\eps\mathds{1}||_\infty+\dfrac{1}{1-\rho}||d+\eps\mathds{1}||_\infty\right).
  \end{array}
  \end{equation}
  This proves that we have $\xi(t)\le\bar\xi_2$ for all $t_0\ge0$ where $\bar\xi_2=M\left(\dfrac{1}{\alpha}||b+\eps\mathds{1}||_1+\dfrac{1}{1-\rho}||d+\eps\mathds{1}||_1\right)\mathds{1}$.  }

\blue{To show the existence of a positive lower bound $\bar\xi_1>0$, we claim first that there exists a large enough $c>0$ such that
  \begin{equation}\label{eq:cLinf}
   \partial_s\tilde \Phi(t,s)(b+\eps\mathds{1})\le c\tilde \Phi(t,s)(b+\eps\mathds{1})
  \end{equation}
 holds for all $-\infty<s\le t\le t_1$ and $t_k<s\le t\le t_{k+1}, k\in\mathbb{Z}_{\ge1}$. This can be seen from the fact this is equivalent to say that $\tilde \Phi(t,s)(cI_n+\tilde A(s))(b+\eps\mathds{1})\ge0$ where we have used the fact that $\partial_s\tilde \Phi(t,s)=-\tilde \Phi(t,s)\tilde A(s)$ for all admissible values $s\le t$. Since $\tilde \Phi(t,s)$ and $(b+\eps\mathds{1})$ are nonnegative, the inequality can be ensured for all admissible values for $s\le t$ if $cI_n+\tilde A(s)\ge0$ for all the admissible values for $s$. It is worth mentioning here that we have to evaluate here $A(s)$ outside its domain of definition. This is not a problem as it can be extended on $(-\infty,\infty)$ by simply letting $\tilde A(s)=-I_n$ for $s<0$ or any other Hurwitz stable Metzler matrix. In this regard, it is possible to find a large enough $c>0$ such that \eqref{eq:cLinf} is verified by virtue of the Metzler structure of the matrix $\tilde A(s)$ and its boundedness. A suitable choice is $c = \textstyle \sup_{s\ge 0}||(\tilde a_{11}(s),\ldots,\tilde a_{nn}(s))||_\infty+1$ where the $\tilde a_{ii}(s)$'s are the diagonal elements of the matrix $\tilde A(s)$.\\

\noindent Now assume that $t_k<t\le t_{k+1}$ and observe that
\begin{equation}
\begin{array}{rcccl}
   \int_{-\infty}^{t_1}\partial_s\tilde \Phi(t,s)(b+\eps\mathds{1})\ds&=&\tilde{\Phi}(t,t_1)(b+\eps\mathds{1})&\le & c\int_{-\infty}^{t_1}\tilde{\Phi}(t,s)(b+\eps\mathds{1})\ds\\
   \int_{t_i}^{t_{i+1}}\partial_s\tilde \Phi(t,s)(b+\eps\mathds{1})\ds&=&[\tilde{\Phi}(t,t_{i+1})-\tilde{\Phi}(t,t_{i}^+)](b+\eps\mathds{1})&\le & c\int_{t_i}^{t_{i+1}}\tilde{\Phi}(t,s)(b+\eps\mathds{1})\ds\\
   \int_{t_k}^{t}\partial_s\tilde \Phi(t,s)(b+\eps\mathds{1})\ds&=&[I-\tilde{\Phi}(t,t_{k}^+)](b+\eps\mathds{1})&\le & c\int_{t_k}^{t}\tilde{\Phi}(t,s)(b+\eps\mathds{1})\ds.
\end{array}
\end{equation}
Summing all the above terms gives
\begin{equation}
  (b+\eps\mathds{1})+\sum_{i=1}^k[\tilde{\Phi}(t,t_{i})-\tilde{\Phi}(t,t_{i}^+)](b+\eps\mathds{1})\le c\int_{-\infty}^{t}\tilde{\Phi}(t,s)(b+\eps\mathds{1})\ds.
\end{equation}
Adding $c\smashoperator{\sum_{t_1^+\le t_i^+\le t}}\tilde{\Phi}(t,t_i^+)(d+\eps\mathds{1})$ on both sides yields
\begin{equation}
  (b+\eps\mathds{1})+\sum_{i=1}^k[(\tilde{\Phi}(t,t_{i})-\tilde{\Phi}(t,t_{i}^+))(b+\eps\mathds{1})+c\tilde{\Phi}(t,t_i^+)(d+\eps\mathds{1})]\le c\xi(t)
\end{equation}
where we have used the fact that $t_k<t\le t_{k+1}$. Noting that
\begin{equation}
  \sum_{i=1}^k[(\tilde{\Phi}(t,t_{i})-\tilde{\Phi}(t,t_{i}^+))(b+\eps\mathds{1})+c\tilde{\Phi}(t,t_i^+)(d+\eps\mathds{1})]=\sum_{i=1}^k\tilde{\Phi}(t,t_{i}^+)[(\tilde{J}(i)-I)(b+\eps\mathds{1})+c(d+\eps\mathds{1})].
\end{equation}
Choosing $c>0$ such that $[(\tilde{J}(i)-I)(b+\eps\mathds{1})+c(d+\eps\mathds{1})]\ge0$ for all $i\in\mathbb{Z}_{\ge1}$ (note that this does not contradict the previous choice for  $c$) yields
\begin{equation}
  (b+\eps\mathds{1})\le c\xi(t)
\end{equation}
and, as a result, $\xi(t)\ge(b+\eps\mathds{1})/c=:\bar\xi_1$. Hence, we have shown the uniform boundedness and boundedness away from 0 of the function $\xi$ in \eqref{eq:xiLinfproof}.

 We address now the problem of verifying whether $\xi$ in \eqref{eq:xiLinfproof} satisfies the inequalities in \eqref{eq:jdlsjdls}. Differentiating \eqref{eq:xiLinfproof} yields
    \begin{equation}
    \begin{array}{rcl}
      \dot\xi(t)&=&b+\eps\mathds{1}+\tilde{A}(t)\left(\int_{-\infty}^t\tilde{\Phi}(t,s)(b+\eps\mathds{1})\ds+\smashoperator{\sum_{t_1^+\le t_i^+\le t}}\tilde{\Phi}(t,t_i^+)(d+\eps\mathds{1})\right)\\
      &=&\tilde{A}(t)\xi(t)+b+\eps\mathds{1}
    \end{array}
  \end{equation}
  which shows that \eqref{eq:xiLinfproof} verifies the first inequality in \eqref{eq:jdlsjdls}. Similarly, evaluating $\xi(t)$ at $t_k$ and $t_k^+$ yields
  \begin{equation}
  \begin{array}{rcl}
  \xi(t_k)&=&\int_{-\infty}^{t_k}\tilde{\Phi}(t_k,s)(b+\eps\mathds{1})\ds+\smashoperator{\sum_{1\le i<k}}\tilde{\Phi}(t_k,t_i^+)(d+\eps\mathds{1})\\
  \xi(t_k^+)&=&\int_{-\infty}^{t_k}\tilde{\Phi}(t_k^+,s)(b+\eps\mathds{1})\ds+\smashoperator{\sum_{1\le i\le k}}\tilde{\Phi}(t_k^+,t_i^+)(d+\eps\mathds{1})\\
                    &=&\int_{-\infty}^{t_k}\tilde{\Phi}(t_k^+,s)(b+\eps\mathds{1})\ds+(d+\eps\mathds{1})+\smashoperator{\sum_{1\le i<k}}\tilde{\Phi}(t_k^+,t_i^+)(d+\eps\mathds{1})\\
                    &=&\tilde{J}(k)\int_{-\infty}^{t_k}\tilde{\Phi}(t_k,s)(b+\eps\mathds{1})\ds+(d+\eps\mathds{1})+\tilde{J}(k)\smashoperator{\sum_{1\le i<k}}\tilde{\Phi}(t_k,t_i^+)(d+\eps\mathds{1})\\
                    &=&\tilde{J}(k)\xi(t_k)+d+\eps\mathds{1}
  \end{array}
  \end{equation}
  which proves that  \eqref{eq:xiLinfproof} verifies the second inequality in \eqref{eq:jdlsjdls} and, as a result, proves the desired implication.}
\end{proof}

\blue{\begin{remark}\label{rem:relax}
  It is interesting to note that the conditions in the above result can be relaxed. The first one is by allowing hybrid rates $(0,\rho)$ and $(\alpha,1)$ where $\alpha>0$ and $\rho\in(0,1)$. The first situation corresponds to the case where the Lyapunov function is nonincreasing along the flow of the system whereas the second one to the case the Lyapunov is nonincreasing jumps. In each of this case, the corresponding condition can be made a closed inequality and the constant terms $b+\eps\mathds{1}$ and $d+\eps\mathds{1}$ can be set to 0. These cases correspond to the persistent jumping and persistent flowing conditions of \cite{Goebel:12}. The proof of Lemma \ref{lem:general_max} can be adapted to cope with these relaxations by either dropping the integral term in \eqref{eq:xiLinfproof} in the persistent jumping case and the sum term in \eqref{eq:xiLinfproof}  in the persistent flowing case.
\end{remark}}

On the strength of the above result, we are now in position of stating a necessary and sufficient condition for the system \eqref{eq:mainsyst} to be exponentially stable and to have an hybrid $L_\infty\times\ell_\infty$-gain of at most $\gamma$.
\begin{theorem}\label{th:generalLinf}
    Assume that the sequence of impulse times $\frak{T}=\{t_k\}_{k\in\mathbb{Z}_{\ge1}}$ is given. Then, the following statements are equivalent:
  \begin{enumerate}[(a)]
    \item The system \eqref{eq:mainsyst} with $u_c,u_d,w_c,w_d\equiv0$ is uniformly exponentially stable and the hybrid $L_\infty\times\ell_\infty$-gain of the transfer $(w_c,w_d)\mapsto(z_c,z_d)$ is at most $\gamma$.
    \item There exist positive vectors $\bar\xi_1,\bar\xi_2\in\mathbb{R}^n_{>0}$, $0<\bar\xi_1\le\bar\xi_2$, a continuously differentiable vector-valued function $\xi:[t_0,\infty)\mapsto\mathbb{R}_{>0}^n$, $\bar\xi_1\le\xi(t)\le\bar\xi_2$, $t\ge t_0$, and a scalar $\eps>0$ such that
        \begin{equation}\label{eq:condLinfperf}
          \begin{array}{rcl}
            -\dot{\xi}(t)+\tilde{A}(t)\xi(t)+\tilde{E}_c(t)\mathds{1}&<&0, t\in t\in(t_k,t_{k+1}],k\in\mathbb{Z}_{\ge0}\\
            \tilde J(k)\xi(t_k)-\xi(t_k^+)+\tilde E_d(k)\mathds{1}&<&0,k\in\mathbb{Z}_{\ge1}\\
            \tilde{C}_c(t)\xi(t)+\tilde{F}_c(t)\mathds{1}-\gamma\mathds{1}&<&0,t\ge t_0\\
            \tilde{C}_d(k)\xi(t_k)+\tilde{F}_d(k)\mathds{1}-\gamma\mathds{1}&<&0, k\in\mathbb{Z}_{\ge1}
          \end{array}
        \end{equation}
        for any $t_0\in[0,t_1)$.
  \end{enumerate}
\end{theorem}
\begin{proof}
The proof of this result is inspired by \cite{Blanchini:15} for switched systems and extends it to the case of the impulsive systems.\\

  \textbf{Proof that (b) implies (a).} The stability of the system follows from the two first inequalities in \eqref{eq:condLinfperf} and Lemma \ref{lem:general_max}. So let us address the performance. Assume that the nonnegative exogenous disturbances $w_c$ and $w_d$ are such that $||w_c||_\infty\le1$ and $||w_d||_\infty\le1$.

  Note first that from \eqref{eq:mainsyst} we have that
  \begin{equation}
  \begin{array}{rcl}
     x(t)&=&\tilde \Phi(t,t_k^+)x(t_k^+)+\int_{t_k}^{t}\tilde \Phi(t,s)\tilde E_c(s)w_c(t_k+s)\ds,\ t\in(t_k,t_{k+1}], k\in\mathbb{Z}_{\ge0}\\
     x(t_k^+)&=&\tilde J(k)x(t_k)+\tilde E_d(k)w_d(k),\ k\in\mathbb{Z}_{\ge1}
  \end{array}
  \end{equation}
  and from \eqref{eq:condLinfperf} we obtain that
  \begin{equation}
  \begin{array}{rcl}
    \xi(t)&>&\xi(t_k^+)+\tilde \Phi(t,t_k^+)\xi(t_k^+)+\int_{t_k}^{t}\tilde \Phi(t,s)\tilde E_c(s)\mathds{1}\ds,\ t\in(t_k,t_{k+1}], k\in\mathbb{Z}_{\ge0}\\
    \xi(t_k^+)&>&\tilde J(k)x(t_k)+\tilde E_d(k)w_d(k),\ k\in\mathbb{Z}_{\ge1}.
    \end{array}
  \end{equation}
  Letting $e:=\xi-x$ yields
  \begin{equation}
  \begin{array}{rcl}
        e(t)&>&\tilde \Phi(t,t_k^+)e(t_k^+)+\int_{t_k}^{t}\tilde \Phi(t,s)\tilde E_c(s)(\mathds{1}-w_c(t_k+s))\ds,\ t\in(t_k,t_{k+1}], k\in\mathbb{Z}_{\ge0}\\
        e(t_k^+)&>&\tilde  J(k)e(t_k)+\tilde E_d(k)(\mathds{1}-w_d(k)),k\in\mathbb{Z}_{\ge1}.
    \end{array}
  \end{equation}
  Assuming $x_0=0$ (i.e. zero initial conditions), we get that $e(0)=\xi(0)>0$, and, as a result $e(t)\ge0$ at all times since the dynamics of the error is described by an internally positive system and because $\mathds{1}-w_c(t)$ and $\mathds{1}-w_d(k)$ are both nonnegative. We then have that
  \begin{equation}
  \begin{array}{rcl}
        z_c(t)&=&\tilde C_c(t)x(t)+\tilde F_c(t)w_c(t)\\
                              &=&\tilde C_c(t)\xi(t)+\tilde F_c(t)\mathds{1}-\left[\tilde C_c(t)e(t)+\tilde F_c(t)(\mathds{1}-w_c(t))\right]\\
                            &\le&\tilde C_c(t)\xi(t)+\tilde F_c(t)\mathds{1}\\
                            &<&\gamma\mathds{1}
  \end{array}
  \end{equation}
  where we have used the fact that $\tilde C_c(\tau)e(t)+\tilde F_c(t)(\mathds{1}-w_c(t))\ge0$. A similar argument yields $z_d(k)<\gamma\mathds{1}$. This proves the implication.\\

  \noindent\textbf{Proof that (a) implies (b).} Assume that the system \eqref{eq:mainsyst} with $u_c,u_d\equiv0$ is uniformly exponentially stable whenever $w_c  ,w_d\equiv0$, and that it satisfies $||(z_c,z_d)||_{L_\infty\times\ell_\infty}<\gamma$ for all $||(w_c,w_d)||_{L_\infty\times\ell_\infty}=1$ whenever $x_0=0$. \blue{Since the system \eqref{eq:mainsyst} with $u_c,u_d\equiv0$ is uniformly exponentially stable whenever $w_c  ,w_d\equiv0$, then, from Lemma \ref{lem:general_max}, there exists a bounded vector-valued function $v(t)$ which is positive, bounded away from zero, and which verifies
      \begin{equation}
      \begin{array}{rcl}
        -\dot{v}(t)+\tilde{A}(t)v(t)+b&<&0,\ t\in(t_k,t_{k+1}], k\in\mathbb{Z}_{\ge0}\\
       \tilde  J(k)v(t_k)-v(t_k^+)+d&<&0,\ k\in\mathbb{Z}_{\ge1}.
      \end{array}
    \end{equation}
    Define now the function
        \begin{equation}
        \xi^*(t):=\epsilon v(t)+\tilde{r}(t)
    \end{equation}
    where $\epsilon>0$ and
    \begin{equation}
      \tilde{r}(t):=\int_{t_0}^t\tilde{\Phi}(t,s)\tilde{E}_c(s)\mathds{1}\ds+\smashoperator{\sum_{1\le i, t_k<t}}\tilde{\Phi}(t,t_i^+)E_d(i)\mathds{1}, t\in(t_k,t_{k+1}],k\in\mathbb{Z}_{\ge0}.
    \end{equation}
    Note that $\xi^*(t)$ is defined at all times since the system is uniformly exponentially stable in the sense of Definition \ref{def:expstab} and since $\tilde E_c(s)$ and $\tilde E_d(k)$ are uniformly bounded. It is immediate to see that this function satisfies
    \begin{equation}
      -\dot{\xi}^*(t)+\tilde{A}(t)\xi^*(t)+\tilde{E}_c(t)\mathds{1}<-\epsilon b<0
    \end{equation}
    for all $t\in t\in(t_k,t_{k+1}]$, $k\in\mathbb{Z}_{\ge0}$, and
    \begin{equation}
      \tilde J(k)\xi^*(t_k)-\xi^*(t_k^+)+\tilde{E}_d(k)\mathds{1}<-\epsilon d<0
    \end{equation}
    for all $k\in\mathbb{Z}_{\ge1}$. This proves that the two first inequalities in \eqref{eq:condLinfperf} hold with $\xi=\xi^*$. To prove the third one note that
 \begin{equation}
 \begin{array}{rcl}
   \tilde{C}_c(t)\xi^*(t)+\tilde{F}_c(t)\mathds{1}&=&\tilde{C}_c(t)\epsilon v(t)+\tilde{C}_c(t)\left(\int_{t_0}^t\tilde{\Phi}(t,s)\tilde{E}_c(s)\mathds{1}\ds+\smashoperator{\sum_{1\le i, t_k<t}}\tilde{\Phi}(t,t_i^+)E_d(i)\mathds{1}\right)+\tilde{F}_c(t)\mathds{1}\\
   &=&\tilde{C}_c(t)\epsilon v(t)+(G_{cc}\mathds{1})(t)+(G_{cd}\mathds{1})(t)\\
   &<&\tilde{C}_c(t)\epsilon v(t)+\gamma\mathds{1}
 \end{array}
 \end{equation}
 where we have used the fact that $||(G_{cc}\mathds{1})+(G_{cd}\mathds{1})||_{L_\infty}<\gamma$, by assumption. Since $\epsilon>0$ is arbitrary, then we must have that $\tilde{C}_c(t)\xi^*(t)+\tilde{F}_c(t)\mathds{1}<\gamma\mathds{1}$, which proves the third inequality in \eqref{eq:condLinfperf}. The fourth one is proved analogously and, for this reason, the proof is omitted. This proves the result.}
\end{proof}


\subsection{General stability and hybrid $L_1\times\ell_1$ performance results for linear positive impulsive systems}

The following result is the analogue of Lemma \ref{lem:general_max} in the context of the use of a linear sum-separable copositive Lyapunov function:
\begin{lemma}\label{lem:general_lin}
    Assume that the sequence of impulse times $\frak{T}=\{t_k\}_{k\in\mathbb{Z}_{\ge1}}$ is given. Then, the following statements are equivalent:
  \begin{enumerate}[(a)]
    \item The time-varying impulsive system
    \begin{equation}
      \begin{array}{rcl}
        \dot{x}(t)&=&\tilde{A}(t)x(t),\ t\in(t_k,t_{k+1}],k\in\mathbb{Z}_{\ge0}\\
        x(t_k^+)&=&\tilde J(k)x(t_k),\ k\in\mathbb{Z}_{\ge1}\\
        x(t_0^+)&=&x(t_0)=x_0
      \end{array}
    \end{equation}
    is uniformly exponentially stable in the sense of Definition \ref{def:expstab}.
    \item There exist positive vectors $\bar\chi_1,\bar\chi_2\in\mathbb{R}^n_{>0}$, $0<\bar\chi_1\le\bar\chi_2$, and an continuously differentiable vector-valued function $\chi:[t_0,\infty)\mapsto\mathbb{R}^b$, verifying $\bar \chi_1<\chi(t)<\bar \chi_2$ , $t\ge t_0$, such that the function
    \begin{equation}
    V(t,x(t))=\chi(t)^{\T}x(t)
    \end{equation}
    is a uniform Lyapunov function for the system \eqref{eq:stabTV}.
    \item The differential-difference inequality
    \begin{equation}\label{ea:kdsolkds;ldk}
      \begin{array}{rcl}
        \dot{\chi}(t)^{\T}+\chi(t)^{\T}\tilde{A}(t)+b^{\T}&<&0,\ t\in(t_k,t_{k+1}],k\in\mathbb{Z}_{\ge0}\\
        \chi(t_k^+)^{\T}\tilde J(k)-\chi(t_k)^{\T}+d^{\T}&<&0,\ k\in\mathbb{Z}_{\ge1}
      \end{array}
    \end{equation}
    has an almost everywhere differentiable positive solution $\chi(t)$ verifying $\bar \chi_1<\chi(t)<\bar \chi_2$  for some positive vectors $\bar\chi_1,\bar\chi_2\in\mathbb{R}^n_{>0}$, $0<\bar\chi_1\le\bar\chi_2$ and for all $t\ge t_0$ and all $t\le t_0< t_1$.
  \end{enumerate}
\end{lemma}
\begin{proof}
The proof of this result is inspired by \cite{Blanchini:15} for switched systems and extends it to the case of the impulsive systems.\\

  \textbf{Proof that (c) implies (a).} Let us define the Lyapunov function $V(x,t)=\chi(t)^{\T}x$. Note that $\min\{\bar \chi_1\}||x||_1\le V(x,t)\le\max\{\bar \chi_2\}||x||_1$. Note that this function is continuous and differentiable almost everywhere. Then, we have that
  \begin{equation}
    \begin{array}{rcl}
      \dot{V}(x,t)    &=&  \dot{\chi}(t)^{\T}x(t)+\chi(t)^{\T}\tilde{A}(t)x(t)\\
                                    &<& -b^{\T}x(t)\\
                                    &<& -\dfrac{\min\{b\}}{\max\{\bar \chi_2\}}V(x,t).
    \end{array}
  \end{equation}
   Similarly, we have that
  \begin{equation}
    \begin{array}{rclcrcl}
      V(x,t_k)&=&\chi(t_k)^{\T}x(t_k)&\textnormal{and}&   V(x,t_k^+)&=&\chi(t_k^+)^{\T}x(t_k^+).
    \end{array}
  \end{equation}
  Hence,
  \begin{equation}
  \begin{array}{rcl}
    V(x,t_k^+)-V(x,t_k) &=& \chi(t_k^+)^{\T}x(t_k^+)-\chi(t_k)^{\T}x(t_k)\\
                                        &<& \chi(t_k^+)^{\T}x(t_{k}^+)-d^{\T}x(t_{k})-\chi(t_k^+)^{\T}Jx(t_k)\\
                                        &<& -d^{\T}x(t_k)\\
                                        &<&-\dfrac{\min\{d\}}{\max\{d\}}V(x,t_k)
  \end{array}
  \end{equation}
  where we have use the fact that $\zeta(t_k)>d$ from the second inequality in \eqref{ea:kdsolkds;ldk}. \blue{Letting $\alpha=\textstyle\frac{\min\{b\}}{\max\{\bar\xi_{2}\}}$ and $\rho=1-\textstyle\frac{\min\{d\}}{\max\{d\}}$, we get that
  \begin{equation}
    \begin{array}{rcl}
      \dot{V}(x,t)&\le&-\alpha V(x,t), t\in(t_k,t_{k+1}], k\in\mathbb{Z}_{\ge0}\\
      V(x,t_k^+)&\le&\rho V(x,t_k),\ k\ge1.
    \end{array}
  \end{equation}
  This implies that
  \begin{equation}
    V(x,t)\le V(x_0,t_0)\rho^{\kappa(t,t_0)}e^{-\alpha(t-t_0)},\ t\ge t_0
  \end{equation}
  which implies in turn
  \begin{equation}
    ||x(t)||_1\le\dfrac{\max\{\bar{\chi}_2\}}{\min\{\bar{\chi}_1\}}\rho^{\kappa(t,t_0)}e^{-\alpha(t-t_0)}||x_0||_1,\ t\ge t_0
  \end{equation}
  which proves the uniform exponential stability of the system \eqref{eq:mainsyst}. This implies that (a) and (c) hold.}\\

  \noindent\textbf{Proof that (a) implies (c).} Assume that for the given sequence of jump times the system is uniformly exponentially stable and let
  \begin{equation}\label{eq:chiL1proof}
    \chi(t)^{\T}=\int_{t}^{\infty}(b+\eps\mathds{1})^{\T}\tilde{\Phi}(s,t)\ds+\sum_{t_i\ge t}(d+\eps\mathds{1})^{\T}\tilde{\Phi}(t_{i},t).
  \end{equation}
  \blue{We need to show first that this expression is well-defined and is uniformly bounded from above by some $\bar\chi_2>0$. From the definition of uniform exponential stability, i.e. Definition \ref{def:expstab}, we have that
  \begin{equation}
  \begin{array}{rcl}
    ||\chi(t)||_\infty&\le&\int_{t}^\infty M\rho^{\kappa(s,t)}e^{-\alpha(s-t)}||b+\eps\mathds{1}||_\infty\ds+\smashoperator{\sum_{t_i\ge t}}Me^{-\alpha(t_i-t)}\rho^{\kappa(t_i,t)}||d+\eps\mathds{1}||_\infty\\
                &\le&\int_{t}^\infty Me^{-\alpha(s-t)}||b+\eps\mathds{1}||_\infty\ds+\smashoperator{\sum_{t_i\ge t}}M\rho^{\kappa(t_i,t)}||d+\eps\mathds{1}||_\infty\\
                &\le&\int_{t}^\infty Me^{-\alpha(s-t)}||b+\eps\mathds{1}||_\infty\ds+\sum_{k\ge0}^\infty M\rho^k||d+\eps\mathds{1}||_\infty\\
            &=&M\left(\dfrac{1}{\alpha}||b+\eps\mathds{1}||_\infty+\dfrac{1}{1-\rho}||d+\eps\mathds{1}||_\infty\right).
  \end{array}
  \end{equation}
  This proves that we have $\chi(t)\le\bar\chi_2$ for all $t_0\ge0$ where $\bar\chi_2=M\left(\dfrac{1}{\alpha}||b+\eps\mathds{1}||_1+\dfrac{1}{1-\rho}||d+\eps\mathds{1}||_1\right)\mathds{1}$. }

  \blue{To show the existence of a positive lower bound $\bar\chi_1>0$, we claim first that there exists a large enough $c>0$ such that
  \begin{equation}\label{eq:cL1}
   (b+\eps\mathds{1})^{\T}\partial_s\tilde \Phi(s,d)\ge -c(b+\eps\mathds{1})^{\T}\tilde \Phi(s,t)
  \end{equation}
 holds for all $t_k< s\le t\le t_{k+1}$, $k\in\mathbb{Z}_{\ge0}$. This can be seen to be true since using $\partial_s\tilde{\Phi}(s,t)=\tilde A(s)\tilde{\Phi}(s,t)$, the above condition can be reformulated as $(b+\eps\mathds{1})^{\T}(\tilde{A}(s)+c I_n)\tilde{\Phi}(s,t)\ge0$ and one can see that if $c$ is large enough then $\tilde{A}(s)+c I_n$ is nonnegative for all $s\ge0$. Indeed, we can simply take any $c\ge\textstyle\sup_{s\ge 0}||(\tilde a_{11}(s),\ldots,\tilde a_{nn}(s))||_\infty$ where the $\tilde a_{ii}(s)$'s are the diagonal elements of the matrix $\tilde A(s)$.\\

\noindent Now assume that $t_{k-1}<t\le t_{k}$ and observe that
\begin{equation}
\begin{array}{rcccl}
   \int_{t}^{t_k}(b+\eps\mathds{1})^{\T}\partial_s\tilde \Phi(s,t)\ds&=&(b+\eps\mathds{1})^{\T}[\tilde{\Phi}(t_k,t)-I_n]&\ge & -c\int_{t}^{t_k}(b+\eps\mathds{1})\tilde{\Phi}(t,s)\ds\\
   \int_{t_i}^{t_{i+1}}(b+\eps\mathds{1})^{\T}\partial_s\tilde \Phi(s,t)\ds&=&(b+\eps\mathds{1})^{\T}[\tilde{\Phi}(t_{i+1},t)-\tilde{\Phi}(t_{i}^+,t)]&\ge & -c\int_{t_i}^{t_{i+1}}\tilde{\Phi}(t,s)(b+\eps\mathds{1})\ds.
\end{array}
\end{equation}
Summing all the above terms gives
\begin{equation}
  -(b+\eps\mathds{1})+\sum_{i=k}^\infty(b+\eps\mathds{1})^{\T}[\tilde{\Phi}(t_{i},t)-\tilde{\Phi}(t_{i}^+,t)]\ge -c\int_{t}^{\infty}(b+\eps\mathds{1})^{\T}\tilde{\Phi}(s,t)\ds.
\end{equation}
Adding $-c\sum_{t_i\ge t}(d+\eps\mathds{1})^{\T}\tilde{\Phi}(t_{i},t)=-c\sum_{i=k}^\infty(d+\eps\mathds{1})^{\T}\tilde{\Phi}(t_{i},t)$ on both sides yields
\begin{equation}
  -(b+\eps\mathds{1})^{\T}+\sum_{i=k}^\infty[(b+\eps\mathds{1})^{\T}(\tilde{\Phi}(t_{i},t)-\tilde{\Phi}(t_{i}^+,t))-c(d+\eps\mathds{1})^{\T}\tilde{\Phi}(t_{i},t)]\ge -c\chi(t)
\end{equation}
where we have used the fact that $t_{k-1}<t\le t_{k}$. Noting that
\begin{equation}
  \sum_{i=k}^\infty[(b+\eps\mathds{1})^{\T}(\tilde{\Phi}(t_{i},t)-\tilde{\Phi}(t_{i}^+,t))-c(d+\eps\mathds{1})^{\T}\tilde{\Phi}(t_{i},t)]=\sum_{i=k}^\infty[(b+\eps\mathds{1})^{\T}(I_n-\tilde{J}(i))-c(d+\eps\mathds{1})^{\T}]\tilde{\Phi}(t_{i},t)
\end{equation}
Choosing $c>0$ such that $(b+\eps\mathds{1})^{\T}(I_n-\tilde{J}(i))-c(d+\eps\mathds{1})^{\T}\le0$ for all $i\in\mathbb{Z}_{\ge1}$ (note that this does not contradict the previous choice for  $c$) implies that
\begin{equation}
  -(b+\eps\mathds{1})^{\T}\ge -c\chi(t)
\end{equation}
and, as a result, $\chi(t)\ge(b+\eps\mathds{1})/c=:\bar\chi_1$. Hence, we have shown the uniform boundedness and boundedness away from 0 of the function $\chi$ in \eqref{eq:chiL1proof}.

We now prove that $\chi$ in \eqref{eq:chiL1proof} verifies the inequalities in \eqref{ea:kdsolkds;ldk}. Computing the derivative of $\chi(t)$ with respect to time yields
    \begin{equation}
    \begin{array}{rcl}
      \dot{\chi}(t)^{\T}&=&\sum_{t_i\ge t}(d+\eps\mathds{1})^{\T}\tilde{\Phi}(t_{i}^+,t)\tilde{A}(t)+\int_{t}^{\infty}(b+\eps\mathds{1})^{\T}\tilde{\Phi}(s,t)\tilde{A}(t)\ds-(b+\eps\mathds{1})^{\T}\\
       &=& \chi(t)^{\T}\tilde{A}(t)+b^{\T}+\eps\mathds{1}^{\T}
    \end{array}
  \end{equation}
  which shows that it verifies the first inequality in \eqref{ea:kdsolkds;ldk}. Similarly, evaluating $\chi(t)$ at $t_k$ and $t_k^+$ yields
  \begin{equation}
    \begin{array}{rcl}
    \chi(t_k^+)^{\T}&=&\sum_{i=k+1}^\infty(d+\eps\mathds{1})^{\T}\tilde{\Phi}(t_{i},t_k^+)+\int_{t}^{\infty}(b+\eps\mathds{1})^{\T}\tilde{\Phi}(s,t_k^+)\ds\\
      \chi(t_k)^{\T}&=&\sum_{i=k}^\infty(d+\eps\mathds{1})^{\T}\tilde{\Phi}(t_{i},t_k)+\int_{t}^{\infty}(b+\eps\mathds{1})^{\T}\tilde{\Phi}(s,t_k)\ds\\
                      &=&(d+\eps\mathds{1})^{\T}+\sum_{i=k+1}^\infty(d+\eps\mathds{1})^{\T}\tilde{\Phi}(t_{i},t_k)+\int_{t}^{\infty}(b+\eps\mathds{1})^{\T}\tilde{\Phi}(s,t_k)\ds\\
                      &=&(d+\eps\mathds{1})^{\T}+\sum_{i=k+1}^\infty(d+\eps\mathds{1})^{\T}\tilde{\Phi}(t_{i},t_k^+)\tilde{J}(k)+\int_{t}^{\infty}(b+\eps\mathds{1})^{\T}\tilde{\Phi}(s,t_k^+)\tilde{J}(k)\ds\\
                      &=&(d+\eps\mathds{1})^{\T}+\chi(t_k^+)^{\T}\tilde{J}(k)
    \end{array}
  \end{equation}
  which shows that $\chi$ in \eqref{eq:chiL1proof} also verifies the second inequality in \eqref{ea:kdsolkds;ldk}. This proves the desired result.}
\end{proof}

We are now in position to state the analogue of Theorem \ref{th:generalLinf} in the hybrid $L_1\times\ell_1$-performance case:
\begin{theorem}\label{th:generalL1}
  Assume that the sequence of impulse times $\frak{T}=\{t_k\}_{k\in\mathbb{Z}_{\ge1}}$ is given. Then, the following statements are equivalent:
  \begin{enumerate}[(a)]
    \item\label{st:L1general:1} The system \eqref{eq:mainsyst} is uniformly exponentially stable and the hybrid $L_1\times\ell_1$-gain of the transfer $(w_c,w_d)\mapsto(z_c,z_d)$ is at most $\gamma$.
    \item\label{st:L1general:2} The system \eqref{eq:mainsyst} is strictly (forward-in-time) dissipative with respect to the storage function $V(x,t)=\chi(t)^{\T}x$, $\bar\chi_1\le\chi(t)\le\bar\chi_2$ for some $0<\bar\chi_1\le\bar\chi_2$, and the supply-rates $s_c(w_c,z_c)=-\mathds{1}^{\T}z_c(t)+\gamma\mathds{1}^{\T}w_c(t)$ and $s_d(w_d,z_d)=-\mathds{1}^{\T}z_d(t)+\gamma\mathds{1}^{\T}w_d(t)$.
    \item\label{st:L1general:3} There exist vectors $\chi_1,\chi_2\in\mathbb{R}^n_{>0}$, $\chi_1\le\chi_2$, a continuously differentiable vector-valued function $\chi:[t_0,\infty)\mapsto\mathbb{R}_{>0}^n$, $\chi_1\le\chi(t)\le\chi_2$, $t\ge t_0$, and a scalar $\eps>0$ such that
        \begin{equation}\label{eq:L1cond}
          \begin{array}{rcl}
            \dot{\chi}(t)^{\T}+\chi(t)^{\T}\tilde{A}(t)+\mathds{1}^{\T}\tilde{C}_c(t)&<&0,\ t\in t\in(t_k,t_{k+1}],k\in\mathbb{Z}_{\ge0}\\
            \chi(t_k^+)^{\T}\tilde J(k)-\chi(t_k)^{\T}+\mathds{1}^{\T}\tilde{C}_d(k)&<&0,\ k\in\mathbb{Z}_{\ge1}\\
            \chi(t)^{\T}\tilde{E}_c(t)+\mathds{1}^{\T}\tilde{F}_c(t)-\gamma\mathds{1}^{\T}&<&0,\ t\ge t_0\\
            \chi(t_{k}^+)^{\T}\tilde{E}_d(k)+\mathds{1}^{\T}\tilde{F}_d(k)-\gamma\mathds{1}^{\T}&<&0,\ k\in\mathbb{Z}_{\ge1}
          \end{array}
        \end{equation}
        hold for all $0\le t_0<t_1$.
  \end{enumerate}
\end{theorem}
\begin{proof}
The proof of this result is inspired by \cite{Blanchini:15} for switched systems and extends it to the case of the impulsive systems.\\

\blue{  \noindent\textbf{Proof that \eqref{st:L1general:2} is equivalent to \eqref{st:L1general:3}.} Statement \eqref{st:L1general:2} is equivalent to saying that
  \begin{equation}
  \begin{array}{rcl}
    V(x(t),t)-V(x(s),s)&\le&\int_s^t\left[s_c(w_c(\theta),z_c(\theta))-\epsilon(||x(\theta)||_1+w_c(\theta)||_1)\right]\d\theta\\
    &&+\sum_{s\le t_i\le t}\left[s_d(w_d(i),z_d(i))-\epsilon(||x(t_i)||_1+||w_d(i)||_1)\right]
  \end{array}
  \end{equation}
  for some $\epsilon>0$ and for all $t_0\le s\le t$ and all $||(w_c,w_d)||_{L_1\times\ell_1}\le1$. This is equivalent to saying that
  \begin{equation}
    \dot{V}(x(t),t)-s_c(w_c(t),z_c(t))\le-\epsilon(||x(t)||_1+||w_c(t)||_1)
  \end{equation}
  and
  \begin{equation}
    V(x(t_i^+),t_i^+)-V(x(t_i),t_i)-s_d(w_d(i),z_d(i))\le-\epsilon(||x(t_i)||_1+||w_d(i)||_1)
  \end{equation}
  hold for all $t\in(t_k,t_{k+1}]$, $k\in\mathbb{Z}_{\ge0}$, $x(t),x(t_i),w_c(t),w_d(i)\ge0$, $i\ge 1$, and $||(w_c,w_d)||_{L_1\times\ell_1}\le1$. Expanding those expressions exactly yields
  \begin{equation}
    \left[\dot{\chi}(t)^{\T}+\chi(t)^{\T}\tilde{A}(t)+\mathds{1}^{\T}\tilde{C}_c(t)+\epsilon\mathds{1}^{\T}\right]x(t) +\left[\chi(t)^{\T}\tilde{E}_c(t)+\mathds{1}^{\T}\tilde{F}_c(t)-\gamma\mathds{1}^{\T}+\epsilon\mathds{1}^{\T}\right]w_c(t)\le0
  \end{equation}
and
   \begin{equation}
    \left[ \chi(t_i^+)^{\T}\tilde J(i)-\chi(t_i)^{\T}+\mathds{1}^{\T}\tilde{C}_d(i) +\epsilon\mathds{1}^{\T}\right]x(t_i) + \left[ \chi(t_{k}^+)^{\T}\tilde{E}_d(i)+\mathds{1}^{\T}\tilde{F}_d(i)-\gamma\mathds{1}^{\T}+\epsilon\mathds{1}^{\T}\right]w_d(i)\le0.
  \end{equation}
  which must hold for all $t\in(t_k,t_{k+1}]$, $k\in\mathbb{Z}_{\ge0}$, $x(t),x(t_i),w_c(t),w_d(i)\ge0$, $i\ge 1$, and $||(w_c,w_d)||_{L_1\times\ell_1}\le1$. This is readily seen to be equivalent to the conditions \eqref{eq:L1cond}.\\}

  \noindent\textbf{Proof that \eqref{st:L1general:3} implies \eqref{st:L1general:1}.} We first prove the stability.  To this aim let us the define the linear copositive Lyapunov function $V(x,t)=\chi(t)^{\T}x$. Note that $\min\{\chi_1\}||x||_1\le V(x,t)\le\max\{\chi_2\}||x||_1$. The uniform exponential stability of the system is immediate from Lemma \ref{lem:general_lin} and the two first inequalities in \eqref{eq:L1cond}. We, therefore, look at the performance. Multiplying the first and the third inequalities to the right by $x(t)$ and $w_c(t)$, respectively, yields
  \begin{equation}
    \dot{V}(x(t),t)<\mathds{1}^{\T}z_c(t)-\gamma\mathds{1}^{\T}w_c(t),t\in(t_k,t_{k+1}].
  \end{equation}
  Integrating from $t_k$ to $t_{k+1}$ yields
  \begin{equation}
    V(x(t_{k+1}),t_{k+1})-V(x(t_{k}^+),t_{k}^+)<\int_{t_k}^{t_{k+1}}\left[-\mathds{1}^{\T}z_c(s)+\gamma \mathds{1}^{\T}w_c(s)\right]\ds.
  \end{equation}
  Similarly, multiplying the second and the fourth second inequalities to the right by $x(t_k)$ and $w_d(k)$, respectively, yields
  \begin{equation}
     V(x,t_k^+)-V(x,t_k) < -\mathds{1}^{\T}z_d(k)+\gamma \mathds{1}^{\T}w_d(k)
  \end{equation}
  and summing the two expressions above yields
  \begin{equation}
    V(x(t_{k+1}),t_{k+1})-V(x(t_{k}),t_k)<\int_{t_k}^{t_{k+1}}\left(-\mathds{1}^{\T}z_c(s)+\gamma \mathds{1}^{\T}w_c(s)\right)\ds-\mathds{1}^{\T}z_d(k)+\gamma \mathds{1}^{\T}w_d(k).
  \end{equation}
  By summing now from $k=0$ to $\infty$ , we obtain
  \begin{equation}
    0<\int_{0}^{\infty}\left(-\mathds{1}^{\T}z_c(s)+\gamma \mathds{1}^{\T}w_c(s)\right)\ds+\sum_{k=0}^\infty\left[-\mathds{1}^{\T}z_d(k)+\gamma \mathds{1}^{\T}w_d(k)\right]
  \end{equation}
  where we have used the fact that $V(x_0,t)=0$ since $x_0=0$ and that $\textstyle\lim_{t\to\infty}V(x(t),t)=0$ since the system is uniformly exponentially stable and since the inputs $w_c\in L_1$ and $w_d\in\ell_1$ both asymptotically tend to 0. Reorganizing the terms yields the inequality
 \begin{equation}
 \begin{array}{rcl}
   0&<&\int_{0}^{\infty}\left(-\mathds{1}^{\T}z_c(s)+\gamma \mathds{1}^{\T}w_c(s)\right)\ds+\sum_{k=1}^\infty\left[-\mathds{1}^{\T}z_d(k)+\gamma \mathds{1}^{\T}w_d(k)\right]\\
   &=&-(||z_c||_{L_1}+||z_d||_{\ell_1})+\gamma(||w_c||_{L_1}+||w_d||_{\ell_1})
 \end{array}
  \end{equation}
  and, therefore, that
  \begin{equation}
    (||z_c||_{L_1}+||z_d||_{\ell_1})<\gamma(||w_c||_{L_1}+||w_d||_{\ell_1}).
  \end{equation}
  This proves the implication.\\

  \noindent\textbf{Proof that \eqref{st:L1general:1} implies \eqref{st:L1general:3}.} Since the system is uniformly exponentially stable and that its hybrid $L_1\times\ell_1$-gains is less than $\gamma$. Then, from Lemma \ref{lem:general_lin}, there exists a bounded positive solution $v(t)$ to
   \begin{equation}\label{eq:jdskljdlkslkdjsldkjsldjsalkdjaskldjaslkdjlakdjlak}
      \begin{array}{rcl}
        \dot{v}(t)^{\T}+v(t)^{\T}\tilde{A}(t)+b^{\T}&<&0,\ t\in t\in(t_k,t_{k+1}],k\in\mathbb{Z}_{\ge0}\\
        v(t_k^+)^{\T}J(k)-v(t_k)+d^{\T}&<&0,\ k\in\mathbb{Z}_{\ge1}.
      \end{array}
    \end{equation}
    \blue{We also define $\chi^*(t)=\epsilon v(t)+\tilde{r}(t)$ for some $\epsilon>0$ and where
    \begin{equation}
      \tilde{r}(t)^{\T}:=\int_t^{\infty}\mathds{1}^{\T}\tilde{C}_c(s)\tilde{\Phi}(s,t)\ds+\sum_{t_i\ge t}\mathds{1}^{\T}C_d(i)\tilde{\Phi}(t_i,t),t\ge t_0.
    \end{equation}
    Note that this function exists since the system is uniformly exponentially stable in the sense of Definition \ref{def:expstab} and since $\tilde C_c(s)$ and $\tilde C_d(k)$ are uniformly bounded. Differentiating $\chi^*(t)$ yields
    \begin{equation}\label{eq:kdsl;sdk1}
    \begin{array}{rcl}
      \dot{\chi}^*(t)&=&    \epsilon \dot{v}(t)-\mathds{1}^{\T}\tilde{C}_c(t)-\tilde{r}(t)^{\T}\tilde{A}(t)\\
                                    &<&   -\epsilon(v(t)^{\T}\tilde{A}(t)+b^{\T})-\mathds{1}^{\T}\tilde{C}_c(t)-\tilde{r}(t)^{\T}\tilde{A}(t)\\
                                    &=& -\dot{\chi}^*(t)^{\T}\tilde{A}(t)-\mathds{1}^{\T}\tilde{C}_c(t)-\epsilon b.
    \end{array}
    \end{equation}
    Similarly,
        \begin{equation}
    \begin{array}{rcl}
      \chi^*(t_k^+)^{\T}&=& \epsilon v(t_k^+)+\int_{t_k}^{\infty}\mathds{1}^{\T}\tilde{C}_c(s)\tilde{\Phi}(s,t_k^+)\ds+\mathds{1}^{\T}\sum_{i=k+1}^\infty C_d(i)\tilde{\Phi}(t_i,t_k^+)\\
      \chi^*(t_k)^{\T}&=& \epsilon v(t_k)+\int_{t_k}^{\infty}\mathds{1}^{\T}\tilde{C}_c(s)\tilde{\Phi}(s,t_k)\ds+\mathds{1}^{\T}\sum_{i=k}^\infty C_d(i)\tilde{\Phi}(t_i,t_k)\\
    &=& \epsilon v(t_k)+\int_{t_k}^{\infty}\mathds{1}^{\T}\tilde{C}_c(s)\tilde{\Phi}(s,t_k^+)\tilde{J}(k)\ds+\mathds{1}^{\T}C_d(k)+\sum_{i=k+1}^\infty C_d(i)\tilde{\Phi}(t_i,t_k^+)\tilde{J}(k)\\
    &=& \epsilon v(t_k)+\mathds{1}^{\T}C_d(k)+\tilde{r}(t_k^+)^{\T}\tilde{J}(k).
    \end{array}
    \end{equation}
    As a result,
    \begin{equation}
      \epsilon v(t_k)+\mathds{1}^{\T}C_d(k)+\tilde{r}(t_k^+)^{\T}\tilde{J}(k)-\chi^*(t_k)=0
    \end{equation}
    and, using \eqref{eq:jdskljdlkslkdjsldkjsldjsalkdjaskldjaslkdjlakdjlak}, this implies that
    \begin{equation}
      \epsilon(v(t_k^+)^{\T}J(k)+d^{\T})+\mathds{1}^{\T}C_d(k)+\tilde{r}(t_k^+)^{\T}\tilde{J}(k)-\chi^*(t_k)<0
    \end{equation}
    and
    \begin{equation}\label{eq:kdsl;sdk2}
      \epsilon d^{\T}+\mathds{1}^{\T}C_d(k)+\chi^*(t_k^+)^{\T}\tilde{J}(k)-\chi^*(t_k)<0.
    \end{equation}
    Now we can conclude that using the function $\chi^*$, the two first inequalities in \eqref{eq:L1cond} hold with $\chi=\chi^*$ since \eqref{eq:kdsl;sdk1} and \eqref{eq:kdsl;sdk2} are equivalent to
       \begin{equation}
          \dot{\chi}^*(t)+\dot{\chi}^*(t)^{\T}\tilde{A}(t)+\mathds{1}^{\T}\tilde{C}_c(t)<-\epsilon b
    \end{equation}
    and
    \begin{equation}
       \mathds{1}^{\T}C_d(k)+\chi^*(t_k^+)^{\T}\tilde{J}(k)-\chi^*(t_k)<-\epsilon d^{\T},
    \end{equation}
    respectively.\\

    We now show that the two last inequalities in \eqref{eq:L1cond} are also satisfied for $\chi=\chi^*$. Indeed, we have that
    \begin{equation}
    \begin{array}{rcl}
        \chi^*(t)^{\T}\tilde{E}_c(t)+\mathds{1}^{\T}\tilde{F}_c(t)&=&\epsilon v(t)^{\T}\tilde{E}_c(t)+\left(\int_t^{\infty}\mathds{1}^{\T}\tilde{C}_c(s)\tilde{\Phi}(s,t)\ds+\sum_{t_i\ge t}\mathds{1}^{\T}C_d(i)\tilde{\Phi}(t_i,t)\right)\tilde{E}_c(t)\\
        &&+\mathds{1}^{\T}\tilde{F}_c(t)
    \end{array}
    \end{equation}
    and, hence,
    \begin{equation}
    \begin{array}{rcl}
        S_c(w_c)&:=&\int_{t_0}^\infty[\chi^*(t)^{\T}\tilde{E}_c(t)+\mathds{1}^{\T}\tilde{F}_c(t)]w_c(t)\dt\\
        &=&\epsilon \int_{t_0}^\infty v(t)^{\T}\tilde{E}_c(t)w_c(t)\dt+\int_{t_0}^\infty\int_t^{\infty}\mathds{1}^{\T}\tilde{C}_c(s)\tilde{\Phi}(s,t)w_c(t)\ds\dt\\
        &&+\int_{t_0}^\infty \sum_{t_i\ge t}\mathds{1}^{\T}C_d(i)\tilde{\Phi}(t_i,t)\tilde{E}_c(t)w_c(t)\dt+\int_{t_0}^\infty\mathds{1}^{\T}\tilde{F}_c(t)w_c(t)\dt\\
        &=&\epsilon \int_{t_0}^\infty v(t)^{\T}\tilde{E}_c(t)\dt+\int_{t_0}^\infty\int_{t_0}^{s}\mathds{1}^{\T}\tilde{C}_c(s)\tilde{\Phi}(s,t)w_c(t)\dt\ds\\
        &&+\sum_{i\ge 1}\int_{t_0}^{t_i} \mathds{1}^{\T}C_d(i)\tilde{\Phi}(t_i,t)\tilde{E}_c(t)w_c(t)\dt+\int_{t_0}^\infty\mathds{1}^{\T}\tilde{F}_c(t)w_c(t)\dt
    \end{array}
    \end{equation}
    where we have exchanged sums and integrals to make appear the integral/sum of convolution operators. Similarly,
    \begin{equation}
      \begin{array}{rcl}
        \chi(t_{k}^+)^{\T}\tilde{E}_d(k)+\mathds{1}^{\T}\tilde{F}_d(k)&=&\left[\epsilon v(t_k^+)+ \int_{t_k}^{\infty}\mathds{1}^{\T}\tilde{C}_c(s)\tilde{\Phi}(s,t_k^+)\ds+\sum_{i=k+1}^\infty \mathds{1}^{\T}C_d(i)\tilde{\Phi}(t_i,t_k^+)\right]\tilde{E}_d(k)\\
        &&+\mathds{1}^{\T}\tilde{F}_d(k).
      \end{array}
    \end{equation}
    Hence, we have that
    \begin{equation}
      \begin{array}{rcl}
        S_d(w_d)&:=&\sum_{k\ge1}\left[\chi(t_{k}^+)^{\T}\tilde{E}_d(k)+\mathds{1}^{\T}\tilde{F}_d(k)\right]w_d(k)\\
        &=&\sum_{k\ge1} \epsilon v(t_k^+)\tilde{E}_d(k)w_d(k)+\sum_{k\ge1} \int_{t_k}^{\infty}\mathds{1}^{\T}\tilde{C}_c(s)\tilde{\Phi}(s,t_k^+)\tilde{E}_d(k)w_d(k)\ds\\
        &&+\sum_{k\ge1} \sum_{i=k+1}^\infty \mathds{1}^{\T}C_d(i)\tilde{\Phi}(t_i,t_k^+)\tilde{E}_d(k)w_d(k)+\sum_{k\ge1}\mathds{1}^{\T}\tilde{F}_d(k)w_d(k),\\
        &=&\sum_{k\ge1} \epsilon v(t_k^+)\tilde{E}_d(k)w_d(k)+ \int_{t_0}^{\infty}\sum_{t_1^+\le t_k<s} \mathds{1}^{\T}\tilde{C}_c(s)\tilde{\Phi}(s,t_k^+)\tilde{E}_d(k)w_d(k)\ds\\
        &&+\sum_{i\ge1} \sum_{k=0}^{i-1} \mathds{1}^{\T}C_d(i)\tilde{\Phi}(t_i,t_k^+)\tilde{E}_d(k)w_d(k)+\sum_{k\ge1}\mathds{1}^{\T}\tilde{F}_d(k)w_d(k)
      \end{array}
    \end{equation}
    where we have also exchanged sums and integrals to make appear the integral/sum of convolution operators. We then obtain that $S(w_c,w_d):=S_c(w_c)+S_d(w_d)$ is given by
        \begin{equation}
      \begin{array}{rcl}
        S(w_c,w_d)&=&\epsilon \int_{t_0}^\infty v(t)^{\T}\tilde{E}_c(t)\dt+\int_{t_0}^\infty\int_{t_0}^{s}\mathds{1}^{\T}\tilde{C}_c(s)\tilde{\Phi}(s,t)w_c(t)\dt\ds\\
        &&+\sum_{i\ge 1}\int_{t_0}^{t_i} \mathds{1}^{\T}C_d(i)\tilde{\Phi}(t_i,t)\tilde{E}_c(t)w_c(t)\dt+\int_{t_0}^\infty\mathds{1}^{\T}\tilde{F}_c(t)w_c(t)\dt\\
        &&+\sum_{k\ge1} \epsilon v(t_k^+)\tilde{E}_d(k)w_d(k)+ \int_{t_0}^{\infty}\sum_{t_1^+\le t_k<s} \mathds{1}^{\T}\tilde{C}_c(s)\tilde{\Phi}(s,t_k^+)\tilde{E}_d(k)w_d(k)\ds\\
        &&+\sum_{i\ge1} \sum_{k=0}^{i-1} \mathds{1}^{\T}C_d(i)\tilde{\Phi}(t_i,t_k^+)\tilde{E}_d(k)w_d(k)+\sum_{k\ge1}\mathds{1}^{\T}\tilde{F}_d(k)w_d(k)\\
        &=&\int_{t_0}^\infty \mathds{1}^{\T}z_c(s)\ds+\sum_{k\ge1}\mathds{1}^{\T}z_d(k)+\epsilon\left(\int_{t_0}^\infty v(t)^{\T}\tilde{E}_c(t)w_c(t)\dt+\sum_{k\ge1}v(t_k^+)\tilde{E}_d(k)w_d(k)\right)\\
        &<&\gamma+\epsilon\left(\int_{t_0}^\infty v(t)^{\T}\tilde{E}_c(t)w_c(t)\dt+\sum_{k\ge1}v(t_k^+)\tilde{E}_d(k)w_d(k)\right)
      \end{array}
    \end{equation}
    where we have used the fact $||(w_c,w_d)||_{L_1\times\ell_1}=1$. Since the above is true for all $\epsilon>0$, then we get that
    \begin{equation}
       \int_{t_0}^\infty[\chi^*(t)^{\T}\tilde{E}_c(t)+\mathds{1}^{\T}\tilde{F}_c(t)]w_c(t)\dt+\sum_{k\ge1}\left[\chi(t_{k}^+)^{\T}\tilde{E}_d(k)+\mathds{1}^{\T}\tilde{F}_d(k)\right]w_d(k)<\gamma
    \end{equation}
    and this holds for all $||(w_c,w_d)||_{L_1\times\ell_1}=1$. Picking now $w_c(t)=\delta(u-t)$, $u\ge t_0$, and $w_d\equiv0$, this implies that
    \begin{equation}
       \chi^*(u)^{\T}\tilde{E}_c(u)+\mathds{1}^{\T}\tilde{F}_c(u)<\gamma
    \end{equation}
    which must hold for all $u\ge t_0$. Picking, finally, $w_c\equiv0$ and $w_d(k)=\delta_{k,\ell}$, $\ell\ge1$, where $\delta_{k,\ell}$ is the Kronecker delta, we obtain that
    \begin{equation}
       \chi(t_{\ell}^+)^{\T}\tilde{E}_d(\ell)+\mathds{1}^{\T}\tilde{F}_d(\ell)<\gamma
    \end{equation}
    which must hold for all $\ell\ge1$. The proof is complete.}
\end{proof}

\subsection{Connection between the results: backward dissipativity of the adjoint system}

\blue{For completeness, it seems important to show how Theorem \ref{th:generalLinf} is connected to Theorem \ref{th:generalL1} and vice-versa. This connection can be clarified via the use of central results in operator theory. Those are recalled below for clarity.

\begin{define}[Adjoint of an operator \cite{Kreyszig:78}]
  Let ${X}$ and ${Y}$ be Banach spaces and let $T:{X}\mapsto{Y}$ be a bounded linear operator. The adjoint of $T$ is the linear operator $T^\times:{Y^*}\mapsto{X^*}$ defined for all $f\in{Y^*}$ by $T^\times(f)=f\circ T$; i.e. $(T^\times f)(u)=f(Tu)$, for all $f\in Y^*$ and all $u\in X$.
\end{define}

The following result states that the norms of an operator and its adjoint coincides with each other:
\begin{proposition}[\cite{Kreyszig:78}]
   Let $X$ and $Y$ be Banach spaces and let $T:X\mapsto Y$ be a bounded linear operator. Then $T^\times:Y^*\mapsto X^*$ is a bounded linear operator and $||T||=||T^\times||$ where the norms are the corresponding induced norm; e.g. $||T||=\sup_{||x||_X=1}||Tx||_Y$.
\end{proposition}


In the context of this paper, the system \eqref{eq:mainsyst} describes an operator with symbol $\Sigma_{t_0,\frak{T}}$ that maps $L_\infty\times\ell_\infty$ to $L_\infty\times\ell_\infty$ and where $t_0$ denotes the initial time and $\frak{T}$ indicates the dependence of the operator on an explicit impulse time sequence which may belong to some family $\mathbb{T}$ of impulse times sequence. In this regard, one can invoke the above proposition to state that
  \begin{equation}\label{eq:operators}
  ||\Sigma_{t_0,\frak{T}}||_{\mathbb{T},L_\infty\times\ell_\infty}=||\Sigma_{t_0,\frak{T}}^\times ||_{\mathbb{T},L_1\times\ell_1}.
\end{equation}

The associated adjoint system $\Sigma^\times _{t_0,\frak{T}}$ (sometimes also called the dual system) is given by \cite{Lawrence:12}
 \begin{equation}\label{eq:transposed}
\begin{array}{rcl}
  \dot{x}^\times (t)&=&-\tilde A(t)^{\T} x^\times (t)-\tilde C_c(t)^{\T} w^\times _c(t)\\
  x^\times (t_k)&=&\tilde J(k)^{\T}x^\times (t_k^+)+\tilde C_d(k)^{\T}w_d^\times (k)\\
  z^\times _c(t)&=&\tilde E_c(t)^{\T}x^\times (t)+\tilde F_c(t)^{\T}w_c^\times (t)\\
  z^\times _d(k)&=&\tilde E_d(k)^{\T}x^\times (t_k^+)+\tilde F_d(k)^{\T}w_d^\times (k)
\end{array}
\end{equation}
where $x^\times (t)\in\mathbb{R}^n,w^\times _c(t)\in\mathbb{R}^{q_c},w^\times _d(k)\in\mathbb{R}^{q_d},z^\times _c(t)\in\mathbb{R}^{p_c}$ and $ z_d^\times (k)\in\mathbb{R}^{p_d}$, $t\in\mathbb{R}_{\ge0}$, $k\in\mathbb{Z}_{\ge0}$, are the adjoint state, the adjoint continuous and discrete exogenous inputs, and the adjoint controlled continuous and discrete exogenous outputs. The adjoint system is slightly different (note the presence of negative signs in the continuous-time part and the fact that past is obtained from the future in the discrete-time part) from the transposed system used in \cite{Briat:11g,Briat:11h} to characterize the $L_\infty$-gain of LTI positive systems in terms of the $L_\infty$-gain of their corresponding transposed systems. This was based on the observation that for LTI positive systems the $L_1$-gain corresponds to the 1-norm of the static gain of the system whereas the $L_\infty$-gain corresponds to the $\infty$-norm of the very same matrix. Since the $\infty$-norm of a matrix coincides with the 1-norm of its transposed, one can then compute the $L_\infty$-gain of a LTI positive system from the $L_1$-gain of the associated transposed system. Unfortunately, this correspondence does not hold in the time-varying setting and one cannot use this result when dealing with LTV systems.

This leads to the following result:
\begin{theorem}
  Let us consider a family of impulse time sequences $\mathbb{T}$ and an initial time $t_0\ge0$. Then, the following statements are equivalent:
  \begin{enumerate}[(a)]
    \item The internally positive system \eqref{eq:mainsyst} is uniformly exponentially stable when $w_c,w_d\equiv0$ and whenever $x_0$ we have that
     \begin{equation}
 \left|\left|\Sigma_{t_0,\frak{T}}\begin{bmatrix}
      w_c\\
      w_d
    \end{bmatrix}\right|\right|_{\mathbb{T},L_\infty\times\ell_\infty}\le\gamma\left|\left|\begin{bmatrix}
      w_c\\
      w_d
    \end{bmatrix}\right|\right|_{L_\infty\times\ell_\infty}
  \end{equation}
  for all $(w_c,w_d)\in L_\infty\times \ell_\infty$.
\item The conditions in Theorem \ref{th:generalLinf} hold.
    \item The adjoint system \eqref{eq:transposed} is backward-in-time uniformly exponentially stable and its backward-in-time hybrid $L_1\times\ell_1$-gain of the transfer $(w^\times _c,w^\times _d)\mapsto(z^\times _c,z^\times _d)$ is at most $\gamma$.
    \item The adjoint system \eqref{eq:transposed} is backward-in-time strictly dissipative with respect to the dual storage function
    \begin{equation}\label{eq:djskdjsd}
      \tilde{V}(x^\times ,t)=\chi(t)^{\T}x^\times
    \end{equation}
    and the dual supply-rates
\begin{equation}
      s^\times _c(z^\times _c ,w^\times _c )=\mathds{1}^{\T}z^\times _c -\gamma\mathds{1}^{\T}w^\times _c
    \end{equation}
    and
    \begin{equation}
      s^\times _d(z^\times _d ,w^\times _d )=\mathds{1}^{\T}z^\times _d -\gamma\mathds{1}^{\T}w^\times _d .
    \end{equation}
  \end{enumerate}
\end{theorem}
\begin{proof}
The proof that (a) and (b) are equivalent follows from Theorem \ref{th:generalLinf}. The equivalence between (c) and (d) comes from the definition of backward-in-time stability and dissipativity, and Theorem \ref{th:generalL1}. We prove now the equivalence between (b) and (d).  The backward-in-time strict dissipativity conditions write
\begin{equation}
    \dot{V}(x^\times(t),t)-s_c^\times(w_c^\times(t),z_c^\times(t))\ge\epsilon(||x^\times(t)||_1+||w^\times_c(t)||_1)
  \end{equation}
  and
  \begin{equation}
    V(x^\times(t_i),t_i)-V(x^\times(t_i^+),t_i^+)-s_d^\times(w_d^\times(i),z_d^\times(i))\ge\epsilon(||x^\times(t_i)||_1+||w^\times_d(i)||_1)
  \end{equation}
  hold for all $t\in(t_k,t_{k+1}]$, $k\in\mathbb{Z}_{\ge0}$, $x^\times(t),x^\times(t_i),w^\times_c(t),w^\times_d(i)\ge0$, $i\ge 1$, and $||(w_c^\times,w_d^\times)||_{L_1\times\ell_1}\le1$. Expanding those expressions exactly yields
  \begin{equation}
    \left[\dot{\chi}(t)^{\T}+\chi(t)^{\T}\tilde{A}(t)+\mathds{1}^{\T}\tilde{C}_c(t)+\epsilon\mathds{1}^{\T}\right]x^\times(t) +\left[\chi(t)^{\T}\tilde{E}_c(t)+\mathds{1}^{\T}\tilde{F}_c(t)-\gamma\mathds{1}^{\T}+\epsilon\mathds{1}^{\T}\right]w^\times_c(t)\le0
  \end{equation}
and
   \begin{equation}
    \left[ \chi(t_i^+)^{\T}\tilde J(i)-\chi(t_i)^{\T}+\mathds{1}^{\T}\tilde{C}_d(i) +\epsilon\mathds{1}^{\T}\right]x^\times(t_i) + \left[ \chi(t_{k}^+)^{\T}\tilde{E}_d(i)+\mathds{1}^{\T}\tilde{F}_d(i)-\gamma\mathds{1}^{\T}+\epsilon\mathds{1}^{\T}\right]w^\times_d(i)\le0.
  \end{equation}
  which must hold for all $t\in(t_k,t_{k+1}]$, $k\in\mathbb{Z}_{\ge0}$, $x^\times(t),x^\times(t_i),w^\times_c(t),w^\times_d(i)\ge0$, $i\ge 1$, and $||(w_c,w_d)||_{L_1\times\ell_1}\le1$. This is readily seen to be equivalent to the conditions \eqref{eq:L1cond}. This proves the desired result.
\end{proof}}

\section{Stability and hybrid $L_\infty\times\ell_\infty$-performance analysis of linear positive impulsive systems under dwell-time constraints}\label{sec:Linf_imp_stab}

The objective of this section is to use the results derived in the previous one in order to derive stability and performance results under various dwell-time constraints, namely, constant, minimum, rang, and arbitrary dwell-time constraints. The conditions being infinite-dimensional, a relaxation based on sum of squares programming is proposed. The accuracy of the obtained conditions are then illustrated through simple examples.

\subsection{Preliminaries}

In particular, we will also be interested in the following class of \emph{timer-dependent} linear impulsive systems taking the form
\begin{equation}\label{eq:mainsyst2}
\begin{array}{rcl}
  \partial_\tau x(t_k+\tau)&=&A(\tau)x(t_k+\tau)+B_c(\tau)u_c(t_k+\tau)+E_c(\tau)w_c(t_k+\tau),\tau\in(0,T_k],k\in\mathbb{Z}_{\ge 0}\\
  x(t_k^+)&=&Jx(t_k)+B_du_d(k)+E_dw_d(k), k\in\mathbb{Z}_{\ge 1}\\
  z_c(t_k+\tau)&=&C_c(\tau)x(t_k+\tau)+D_c(\tau)u_c(t_k+\tau)+F_c(\tau)w_c(t_k+\tau),\ \tau\in(0,T_k],k\in\mathbb{Z}_{\ge 0}\\
  z_d(k)&=&C_dx(t_k)+D_du_d(k)+F_dw_d(k), k\in\mathbb{Z}_{\ge 1}\\
  x(0)&=&x(0^+)=x_0
\end{array}
\end{equation}
where the values of the matrix-valued functions only depend on the time elapsed since the last impulse (when the name "timer-dependent" system) and the matrices present in the discrete-time part of the system depend on the last value of the dwell-time. The reason why such a class of systems is considered is that when impulsive systems are controlled or observed under a range or minimum dwell-time constraint, then the obtained controller/observer gains are timer-dependent, making the model for the closed-loop system or the estimation error of the form \eqref{eq:mainsyst2}. Note that the system \eqref{eq:mainsyst2} can be expressed as \eqref{eq:mainsyst} by gluing the matrix-valued functions of the continuous-time part  as $\tilde{A}(t_k+\tau)=A(\tau)$ for $\tau\in(0,T_k]$. This class of systems arises naturally when designing controllers or observers for impulsive and switched systems under constant, minimum, maximum and range dwell-time constraints.


\subsection{Stability and performance analysis under constant dwell-time}

The following results formulate a necessary and sufficient condition for the system \eqref{eq:mainsyst2} with $u_c,u_d\equiv0$ to be exponentially stable and to have an hybrid $L_\infty\times\ell_\infty$-gain of at most $\gamma$ under constant dwell-time sequences defined as
\begin{equation}
  \mathfrak{T}_{0}\in\{\tk:t_{k+1}-t_k=\bar{T},t_0=0,k\in\mathbb{Z}_{\ge0}\}
\end{equation}
with constant dwell-time $\bar T$. In this case, the system \eqref{eq:mainsyst2} becomes a periodically jumping system and this leads to the following result:
\begin{theorem}\label{th:cstDTLinfty}
  Let us consider the system \eqref{eq:mainsyst2} with $u_c,u_d\equiv0$ and assume that it is internally positive. Then, the following statements are equivalent:
  \begin{enumerate}[(a)]
    \item\label{st:cstDTLinfty1} The system \eqref{eq:mainsyst2} is asymptotically stable under constant dwell-time $\bar{T}$ and the hybrid $L_\infty\times\ell_\infty$-gain of the transfer $(w_c,w_d)\mapsto(z_c,z_d)$ is at most $\gamma$.
        \blue{\item\label{st:cstDTLinfty3} There exist a vector $\lambda\in\mathbb{R}^n_{>0}$ and a scalar $\gamma>0$ such that the conditions
 \begin{equation}\label{eq:cstDTLinfty3_1}
   (J\Phi(\bar T,0)-I)\lambda+J\int_0^{\bar T}\Phi(\bar T,s)E_c(s)\mathds{1}\ds+E_d\mathds{1}<0,
 \end{equation}
 \begin{equation}\label{eq:cstDTLinfty3_2}
   C_c(\tau)\left(\Phi(\tau,0)\lambda+\int_0^{\tau}\Phi(\tau,s)E_c(s)\mathds{1}\ds\right)+F_c(\tau)\mathds{1}<\gamma\mathds{1}
 \end{equation}
 and
 \begin{equation}\label{eq:cstDTLinfty3_3}
  C_d\left(\Phi(\bar T,0)\lambda+\int_0^{\bar T}\Phi(\bar T,s)E_c(s)\mathds{1}\ds\right)+F_d\mathds{1}<\gamma\mathds{1}
 \end{equation}
 hold for all $\tau\in[0,\bar T]$.}
    \item\label{st:cstDTLinfty2} There exist a differentiable vector-valued function $\zeta:\mathbb{R}_{\ge0}\mapsto\mathbb{R}^n$, $\zeta(0)>0$, and scalars $\eps,\gamma>0$ such that the conditions
   \begin{equation}\label{eq:cstDTLinfty1}
          \begin{array}{rcl}
            -\dot\zeta(\tau)+A(\tau)\zeta(\tau)+E_c(\tau)\mathds{1}&<&0\\
            C_c(\tau)\zeta(\tau)+F_c(\tau)\mathds{1}-\gamma\mathds{1}&<&0
  \end{array}
  \end{equation}
  and
    \begin{equation}\label{eq:cstDTLinfty2}
   \begin{array}{rcl}
            J\zeta(\bar{T})-\zeta(0)+E_d\mathds{1}&<&0\\
           C_d\zeta(\bar{T})+F_d\mathds{1}-\gamma\mathds{1}&<&0
          \end{array}
        \end{equation}
        hold for all $\tau\in[0,\bar{T}]$.
 \end{enumerate}
\end{theorem}
\begin{proof}
\blue{\noindent\textbf{Proof that \eqref{st:cstDTLinfty2} implies \eqref{st:cstDTLinfty3}.} Integrating the first inequality in \eqref{eq:cstDTLinfty1} from 0 to $\tau$ yields
   \begin{equation}
    \zeta(\tau)>\Phi(\tau,0)\zeta(0)+\int_0^\tau\Phi(t,s)E_c(s)\mathds{1}\ds.
  \end{equation}
  Substituting this expression in the first inequality in \eqref{eq:cstDTLinfty2} yields \eqref{eq:cstDTLinfty3_1}, in the second inequality in \eqref{eq:cstDTLinfty1} yields \eqref{eq:cstDTLinfty3_2}, and in the second inequality in \eqref{eq:cstDTLinfty2} yields \eqref{eq:cstDTLinfty3_3} where $\zeta(0)=\lambda$.\\

\noindent\textbf{Proof that \eqref{st:cstDTLinfty3} implies \eqref{st:cstDTLinfty2}.} Define $\zeta(\tau)=\Phi(\tau,0)\lambda+\int_0^\tau\Phi(t,s)E_c(s)\mathds{1}\ds+\eps\mathds{1}$. It is immediate to see that this is $\zeta$ satisfies the conditions in the statement \eqref{st:cstDTLinfty2} provided that the conditions in the statement \eqref{st:cstDTLinfty3} hold.\\

\noindent\textbf{Proof that \eqref{st:cstDTLinfty1} is equivalent to \eqref{st:cstDTLinfty2}.}  Since the system is periodic (or equivalent the associated discrete-time system is LTI), it is necessary and sufficient to find a periodic Lyapunov function where
\begin{equation}
  \xi(t_k+\tau)=\zeta(\tau),\tau\in(0,\bar T] \textnormal{ and }  \xi(t_k^+)=\zeta(0), k\in\mathbb{Z}_{\ge0}.
\end{equation}
This yields the result.}
\end{proof}

\blue{The above result states the equivalence between several statements. In particular, we have the correspondence between the discrete-time stability criterion of statement \eqref{st:cstDTLinfty3} and the hybrid stability criterion of statement \eqref{st:cstDTLinfty2}. While the discrete-time stability condition is semi-infinite dimensional since it has to be verified for all the values for the timer, the optimization problem itself is finite-dimensional since the decision variable belongs to a finite-dimensional space; i.e. the Euclidian space. Unfortunately, those conditions are unlikely to be verifiable in the general setting due to the dependence on the state-transition matrix which is known to be difficult to compute in the general case and the presence of integral terms. Note, however, that since the state-transition matrix admits a closed form in the timer-invariant setting, those conditions may be checked in this very particular case. As a final remark, it is important to state that the discrete-time stability conditions are difficult to consider for design purposes or when the system is subject to uncertainties, even in the linear time-invariant case, because of the strong nonlinear dependence of the conditions on the matrices of the system.

All those difficulties are resolved by the use of hybrid stability conditions at the expense of trading a finite-dimensional semi-infinite optimization problem for a infinite-dimensional LMI problem. Indeed, the conditions are now affine in the matrices of the system, hence convex, a fact that strongly suggests the possibility of using those conditions for design purposes or to tackle uncertain systems. This will be exploited in Section \ref{sec:Linf_imp_stabz} pertaining to the stabilization of impulsive systems.}

\subsection{Stability and performance analysis under minimum dwell-time}

The following result formulates a sufficient condition for the system \eqref{eq:mainsyst2} with $u_c,u_d\equiv0$ to be exponentially stable and to have an hybrid $L_\infty\times\ell_\infty$-gain of at most $\gamma$ under minimum dwell-time sequences defined as
\begin{equation}
  \mathfrak{T}_0\in\{\tk:t_{k+1}-t_k\ge\bar{T},t_0=0,k\in\mathbb{Z}_{\ge0}\}
\end{equation}
for some minimum dwell-time value $\bar T>0$. We then have the following result:

\begin{theorem}\label{th:minDTLinfty}
    Let us consider the system \eqref{eq:mainsyst2} with $u_c,u_d\equiv0$ and assume that it is internally positive and such that $A(\tau)=A(\bar{T})$, $E_c(\tau)=E_c(\bar{T})$, $C_c(\tau)=C_c(\bar{T})$ and $F_c(\tau)=F_c(\bar{T})$ for all $\tau\ge \bar{T}$ for some scalar $\bar{T}>0$. Then, the following statements are equivalent:
    \begin{enumerate}[(a)]
    \blue{\item\label{st:minDTLinfty3}  There exist a vector $\lambda\in\mathbb{R}^n_{>0}$ and a scalar $\gamma>0$ such that the conditions
       \begin{equation}\label{eq:minDTLinfty3_1}
          \begin{array}{rcl}
           A(\bar{T})\zeta(\bar{T})+E_c(\bar{T})\mathds{1}&<&0\\
            C_c(\bar{T})\zeta(\bar{T})+F_c(\bar{T})\mathds{1}-\gamma\mathds{1}&<&0,
            \end{array}
            \end{equation}
 \begin{equation}\label{eq:minDTLinfty3_2}
 (J\Phi(\theta,0)-I)\lambda+J\int_0^{\theta}\Phi(\theta,s)E_c(s)\mathds{1}\ds+E_d\mathds{1}<0,
 \end{equation}
 \begin{equation}\label{eq:minDTLinfty3_3}
   C_c(\tau)\left(\Phi(\tau,0)\lambda+\int_0^{\tau}\Phi(\tau,s)E_c(s)\mathds{1}\ds\right)+F_c(\tau)\mathds{1}<\gamma\mathds{1}
 \end{equation}
 and
 \begin{equation}\label{eq:minDTLinfty3_4}
  C_d\left(\Phi(\theta,0)\lambda+\int_0^{\theta}\Phi(\theta,s)E_c(s)\mathds{1}\ds\right)+F_d\mathds{1}<\gamma\mathds{1}
 \end{equation}
 hold for all $\tau\in[0,\infty)$ and all $\theta\ge\bar T$.}
  \blue{\item\label{st:minDTLinfty2}  There exist a vector $\lambda\in\mathbb{R}^n_{>0}$ and a scalar $\gamma>0$ such that the conditions
       \begin{equation}\label{eq:minDTLinfty2_1}
          \begin{array}{rcl}
           A(\bar{T})\left(\Phi(\bar T,0)\lambda+\int_0^{\bar T}\Phi(\bar T,s)E_c(s)\mathds{1}\ds\right)+E_c(\bar{T})\mathds{1}&<&0\\
            C_c(\bar{T})\left(\Phi(\bar T,0)\lambda+\int_0^{\bar T}\Phi(\bar T,s)E_c(s)\mathds{1}\ds\right)+F_c(\bar{T})\mathds{1}-\gamma\mathds{1}&<&0,
            \end{array}
            \end{equation}
 \begin{equation}\label{eq:minDTLinfty2_2}
   (J\Phi(\bar T,0)-I)\lambda+J\int_0^{\bar T}\Phi(\bar T,s)E_c(s)\mathds{1}\ds+E_d\mathds{1}<0,
 \end{equation}
 \begin{equation}\label{eq:minDTLinfty2_3}
   C_c(\tau)\left(\Phi(\tau,0)\lambda+\int_0^{\tau}\Phi(\tau,s)E_c(s)\mathds{1}\ds\right)+F_c(\tau)\mathds{1}<\gamma\mathds{1}
 \end{equation}
 and
 \begin{equation}\label{eq:minDTLinfty2_3}
  C_d\left(\Phi(\bar T,0)\lambda+\int_0^{\bar T}\Phi(t,s)E_c(s)\mathds{1}\ds\right)+F_d\mathds{1}<\gamma\mathds{1}
 \end{equation}
 hold for all $\tau\in[0,\bar T]$.}
      \item\label{st:minDTLinfty1} There exist a differentiable vector-valued function $\zeta:\mathbb{R}_{\ge0}\mapsto\mathbb{R}^n$, $\zeta(0)>0$, $\dot{\zeta}(\bar T)=0$, and scalars $\eps,\gamma>0$ such that the conditions
   \begin{equation}\label{eq:minDTLinfty1_1}
          \begin{array}{rcl}
            -\dot\zeta(\tau)+A(\tau)\zeta(\tau)+E_c(\tau)\mathds{1}&<&0\\
            C_c(\tau)\zeta(\tau)+F_c(\tau)\mathds{1}-\gamma\mathds{1}&<&0
  \end{array}
  \end{equation}
     \begin{equation}\label{eq:minDTLinfty1_2}
          \begin{array}{rcl}
           A(\bar{T})\zeta(\bar{T})+E_c(\bar{T})\mathds{1}&<&0\\
            C_c(\bar{T})\zeta(\bar{T})+F_c(\bar{T})\mathds{1}-\gamma\mathds{1}&<&0
            \end{array}
        \end{equation}
     and
        \begin{equation}\label{eq:minDTLinfty1_3}
          \begin{array}{rcl}
           J\zeta(\bar{T})-\zeta(0)+E_d\mathds{1}&<&0\\
           C_d\zeta(\bar{T})+F_d\mathds{1}-\gamma\mathds{1}&<&0
          \end{array}
        \end{equation}
    hold for all $\tau\in[0,\bar T]$.
    \end{enumerate}
    Moreover, when one of the above statements holds, then the positive impulsive system \eqref{eq:mainsyst2} with $u_c,u_d\equiv0$ is asymptotically stable under minimum dwell-time $\bar T$ and the hybrid $L_\infty\times\ell_\infty$-gain of the transfer $(w_c,w_d)\mapsto(z_c,z_d)$ is at most $\gamma$.
\end{theorem}
\begin{proof}
\blue{\noindent\textbf{Proof that \eqref{st:minDTLinfty2} is equivalent to \eqref{st:minDTLinfty3}.} The fact that \eqref{st:minDTLinfty3} implies \eqref{st:minDTLinfty2} is immediate. So, let us focus on the reverse implication. First note that we have $\Phi(T_k,s)=e^{A(\bar T)(T_k-\bar T)}\Phi(\bar T,s)$ for any $s\le\bar T\le T_k$. If we compute the derivative with respect to $\theta$ of the left-hand side of \eqref{eq:minDTLinfty3_2}, we get that
\begin{equation}
  D(\theta):=JA(\theta)\left(\Phi(\theta,0)\lambda+\int_0^{\theta}\Phi(\theta,s)E_c(s)\mathds{1}\ds\right)
\end{equation}
and for $\theta\ge\bar T$, we get that
\begin{equation}
\begin{array}{rcl}
  D(\theta)&=&JA(\bar T)\left(e^{A(\bar T)(\theta-\bar T)}\Phi(\bar T,0)\lambda+e^{A(\bar T)(\theta-\bar T)}\int_0^{\bar T}\Phi(\bar T,s)E_c(s)\mathds{1}\ds+\int_{\bar T}^{\theta}e^{A(\bar T)(\theta-s)}E_c(\bar T)\mathds{1}\ds\right)\\
  &=&JA(\bar T)\left(e^{A(\bar T)(\theta-\bar T)}\Phi(\bar T,0)\lambda+e^{A(\bar T)(\theta-\bar T)}\int_0^{\bar T}\Phi(\bar T,s)E_c(s)\mathds{1}\ds\right)+J\int_{\bar T}^{\theta}A(\bar T)e^{A(\bar T)(\theta-s)}E_c(\bar T)\mathds{1}\ds\\
  &=&JA(\bar T)\left(e^{A(\bar T)(\theta-\bar T)}\Phi(\bar T,0)\lambda+e^{A(\bar T)(\theta-\bar T)}\int_0^{\bar T}\Phi(\bar T,s)E_c(s)\mathds{1}\ds\right)+J\left(e^{A(\bar T)(\theta-\bar T)}-I\right)E_c(\bar T)\\
  &=&Je^{A(\bar T)(\theta-\bar T)}\left[A(\bar T)\left(\Phi(\bar T,0)\lambda+\int_0^{\bar T}\Phi(\bar T,s)E_c(s)\mathds{1}\ds\right)+E_c(\bar T)\mathds{1}\right]-JE_c(\bar T)\mathds{1}\\
\end{array}
\end{equation}
From the first inequality in \eqref{eq:minDTLinfty3_1} and the facts that $-JE_c(\bar T)\mathds{1}\le0$ and $Je^{A(\bar T)(\theta-\bar T)}\ge0$ for all $\theta\ge\bar T$, then we get that $D(\theta)\le0$ for all $\theta\ge\bar T$. As a result, the expression in the left-hand side of \eqref{eq:minDTLinfty3_2} is a nonincreasing function of $\theta\ge\bar T$, and therefore we have that \eqref{eq:minDTLinfty3_2}  holds for all $\theta\ge\bar T$ provided that \eqref{eq:minDTLinfty3_1} and \eqref{eq:minDTLinfty2_2} hold. The proofs of the other implications follow from similar manipulations and are thus omitted here.\\

\noindent\textbf{Proof that \eqref{st:minDTLinfty1} is equivalent to \eqref{st:minDTLinfty2}.} This follows from the same operations as for Theorem \ref{th:cstDTLinfty}.\\

\noindent\textbf{Proof that \eqref{st:minDTLinfty1} implies the uniform exponential stability and the hybrid $L_\infty\times\ell_\infty$ performance.} To prove this, we choose the following structure for $\xi$ in Theorem \ref{th:generalLinf}
\begin{equation}\label{eq:nonsmooth}
\begin{array}{rcl}
  \xi(t_k+\tau)&=&\left\{\begin{array}{rcl}
    \zeta(\tau) &&\textnormal{if }\tau\in[0,\bar{T}]\\
    \zeta(\bar{T}) &&\textnormal{if }\tau\in[\bar{T},T_k]
  \end{array}\right.\\
  \xi(t_k^+)&=&\zeta(0),\ k\in\mathbb{Z}_{\ge1},
\end{array}
\end{equation}
together with the differentiability condition at $\tau=\bar T$ given by $\dot{\zeta}(\bar T)=0$, then we obtain the conditions of the statement \eqref{st:minDTLinfty1}. This proves the result.}
\end{proof}

\blue{The assumption that the matrices describing the system become constant after some time $\bar T$ comes from the fact that stabilizing controllers in the minimum dwell-time cases satisfy such a property, as already emphasized in \cite{Allerhand:11,Briat:13d,Briat:16c}, and that those controllers can be characterized in a convex manner.

It is important to also state the following complementary result:
\begin{proposition}\label{prop:nodiff}
  The constraint that $\dot{\zeta}(\bar T)=0$ can be dropped in Theorem \ref{th:minDTLinfty}, (c).
\end{proposition}
\begin{proof}
  Dropping this constraint means dropping the differentiability of the function $\xi(t)$ in Theorem \ref{th:generalLinf}, which would violate one of the conditions of the theorem. However, it is proven in \cite{Holicki:19} in the context of LMIs that if there exists an absolutely continuous function of the form \eqref{eq:nonsmooth} satisfying the conditions of Theorem \ref{th:minDTLinfty} without the differentiability constraint at $\tau=\bar T$, then there exists a continuously differentiable function satisfying the same conditions. Hence, the constraint that $\dot{\zeta}(\bar T)=0$ can be dropped.
\end{proof}

The above result may seem superficial at first sight but it, in fact, greatly simplifies the problem. Indeed, the constraint $\dot{\zeta}(\bar T)=0$ imposes a much stricter structure on $\zeta$ and dropping this constraint results in more flexibility for its choice. Another way of seeing this result is that the conditions of Theorem \ref{th:generalLinf} can be weakened by seeking for an absolutely continuous function $\xi$ and by replacing the gradient of the Lyapunov function by Clarke's generalized subgradient.}

\subsection{Stability and performance analysis under range dwell-time}

The following result formulates a sufficient condition for the system \eqref{eq:mainsyst2} with $u_c,u_d\equiv0$ to be exponentially stable and to have an hybrid $L_\infty\times\ell_\infty$-gain of at most $\gamma$ under range dwell-time sequences defined as
\begin{equation}
  \mathfrak{T}_0\in\{\tk:t_{k+1}-t_k\in[T_{\min},T_{\max}],t_0=0,k\in\mathbb{Z}_{\ge0}\}
\end{equation}
for some dwell-time bounds $0<T_{\min}\le T_{\max}<\infty$. We then have the following result:
\begin{theorem}\label{th:rangeDTLinfty}
  Let us consider the system \eqref{eq:mainsyst2} with $u_c,u_d\equiv0$ and assume that it is internally positive. Then, the following statements are equivalent:
  \begin{enumerate}[(a)]
  \blue{\item\label{st:rangeDTLinfty3} There exist a vector $\lambda\in\mathbb{R}^n_{>0}$ and a scalar $\gamma>0$ such that the conditions
 \begin{equation}\label{eq:rangeDTLinfty3_1}
 (J\Phi(\theta,0)-I)\lambda+J\int_0^{\theta}\Phi(\theta,s)E_c(s)\mathds{1}\ds+E_d\mathds{1}<0,
 \end{equation}
 \begin{equation}\label{eq:rangeDTLinfty3_2}
   C_c(\tau)\left(\Phi(\tau,0)\lambda+\int_0^{\tau}\Phi(\tau,s)E_c(s)\mathds{1}\ds\right)+F_c(\tau)\mathds{1}<\gamma\mathds{1}
 \end{equation}
 and
 \begin{equation}\label{eq:rangeDTLinfty3_3}
  C_d\left(\Phi(\theta,0)\lambda+\int_0^{\theta}\Phi(\theta,s)E_c(s)\mathds{1}\ds\right)+F_d\mathds{1}<\gamma\mathds{1}
 \end{equation}
 hold for all $\tau\in[0,\Tmax]$ and all $\theta\in[\Tmin,\Tmax]$.}
    \item\label{st:rangeDTLinfty1} There exist a differentiable vector-valued function $\zeta:\mathbb{R}_{\ge0}\mapsto\mathbb{R}^n$, $\zeta(0)>0$, and scalars $\eps,\gamma>0$ such that the conditions
   \begin{equation}\label{eq:rangeDTLinftyZ}
          \begin{array}{rcl}
            -\dot\zeta(\tau)+A(\tau)\zeta(\tau)+E_c(\tau)\mathds{1}&<&0\\
            C_c(\tau)\zeta(\tau)+F_c(\tau)\mathds{1}-\gamma\mathds{1}&<&0
  \end{array}
  \end{equation}
  and
    \begin{equation}\label{eq:rangeDTLinfty}
   \begin{array}{rcl}
            J\zeta(\theta)-\zeta(0)+E_d\mathds{1}&<&0\\
           C_d\zeta(\theta)+F_d\mathds{1}-\gamma\mathds{1}&<&0
          \end{array}
        \end{equation}
        hold for all $\tau\in[0,\Tmax]$ and $\theta\in[\Tmin,\Tmax]$.
    \item\label{st:rangeDTLinfty2} There exist a differentiable vector-valued function $\zeta:\mathbb{R}_{\ge0}\mapsto\mathbb{R}^n$, $\zeta(0)>0$, a vector-valued function $\mu:[\Tmin,\Tmax]\mapsto\mathbb{R}^n$ and scalars $\eps,\gamma>0$ such that the conditions
   \begin{equation}
          \begin{array}{rcl}
            -\dot\zeta(\tau)+A(\tau)\zeta(\tau)+E_c(\tau)\mathds{1}&<&0\\
            C_c(\tau)\zeta(\tau)+F_c(\tau)\mathds{1}-\gamma\mathds{1}&<&0
  \end{array}
  \end{equation}
  and
    \begin{equation}
   \begin{array}{rcl}
        \zeta(\theta)-\mu(\theta)&\le&0\\
            J\mu(\theta)-\zeta(0)+E_d\mathds{1}&<&0\\
           C_d\mu(\theta)+F_d\mathds{1}-\gamma\mathds{1}&<&0
          \end{array}
        \end{equation}
        hold for all $\tau\in[0,\Tmax]$ and $\theta\in[\Tmin,\Tmax]$.
  \end{enumerate}
    Moreover, if one of the above statements hold, then the system \eqref{eq:mainsyst2} with $u_c,u_d\equiv0$ is asymptotically stable under range dwell-time $[\Tmin,\Tmax]$ and the hybrid $L_\infty\times\ell_\infty$-gain of the transfer $(w_c,w_d)\mapsto(z_c,z_d)$ is at most $\gamma$.
\end{theorem}
\begin{proof}
The equivalence between the two first statements follows from the same lines as for the proof of Theorem \ref{th:cstDTLinfty}. To prove statement (b), we just need to choose the following structure for $\xi$ as $\xi(t_k+\tau)=\zeta(\tau)$, that is, $\xi(t_k^+)=\zeta(0)$ and $\xi(t_{k+1})=\zeta(T_k)$ in Theorem \ref{th:generalLinf}. The third statement can be proven from simple substitutions. It can also be proved using the S-procedure on the conditions obtained by applying dissipativity theory on the adjoint system.
\end{proof}
The second statement may seem superfluous, which is indeed the case when stability needs to be analyzed. However, the second statement will play a crucial role in obtaining in Section \ref{sec:Linf_imp_stabz_range} controllers that are independent of the dwell-times values by setting the $\mu(\cdot)$ to be constant function.

\subsection{Stability and performance analysis under arbitrary dwell-time}

Let us address first the case of arbitrary dwell-times sequence. That is, we consider here the family
\begin{equation}
  \mathfrak{T}_0\in\{\tk:t_{k+1}-t_k\in(0,\infty),k\in\mathbb{Z}_{\ge0}, t_0=0,\lim_{k\to\infty}t_k=\infty\}.
\end{equation}

We have the following result:
\begin{theorem}\label{th:arbDTLinfty}
  Let us consider the positive system \eqref{eq:mainsyst2} with $u_c,u_d\equiv0$ and constant matrices. Assume that there exist a vector $\lambda\in\mathbb{R}_{>0}$ and scalars $\gamma>0$ such that the conditions
   \begin{equation}\label{eq:arbDTLinfty1}
          \begin{array}{rcl}
            A\lambda+E_c\mathds{1}&<&0\\
            C_c\lambda+F_c\mathds{1}-\gamma\mathds{1}&<&0
  \end{array}
  \end{equation}
  and
    \begin{equation}\label{eq:arbDTLinfty2}
   \begin{array}{rcl}
            (J-I)\lambda+E_d\mathds{1}&<&0\\
           C_d\lambda+F_d\mathds{1}-\gamma\mathds{1}&<&0
          \end{array}
        \end{equation}
        hold. Then, the system \eqref{eq:mainsyst2} with constant matrices is uniformly exponentially stable under arbitrary dwell-time and the hybrid $L_\infty\times\ell_\infty$-gain of the transfer $(w_c,w_d)\mapsto(z_c,z_d)$ is at most $\gamma$.
\end{theorem}
\blue{\begin{proof}
From Theorem \ref{th:rangeDTLinfty}, the system \eqref{eq:mainsyst2} with constant matrices is uniformly exponentially stable under arbitrary dwell-time and the hybrid $L_\infty\times\ell_\infty$-gain of the transfer $(w_c,w_d)\mapsto(z_c,z_d)$ is at most $\gamma$ if the conditions in Theorem \ref{th:rangeDTLinfty}, (a), hold for all $\theta\ge0$. Let us define by $D(\theta)$ the left-hand side of \eqref{eq:rangeDTLinfty3_1}. Letting $\theta=\theta_0+\eps$ and using a Taylor expansion yields
\begin{equation}
 D(\theta_0+\eps)=D(\theta_0)+\eps Je^{A\theta_0}\left(A\lambda +E_c\mathds{1}\right)+o(\eps)
\end{equation}
for any sufficiently small $\eps$. Therefore, one can see that if $A\lambda +E_c\mathds{1}<0$ (which is identical to the first inequality in \eqref{eq:arbDTLinfty1}), then $A$ is Hurwitz stable and $D(\theta+\eps)$ is locally nonincreasing at $\theta=\theta_0$ since $Je^{A\theta_0}\ge 0$. However, this also holds true for all $\theta_0\ge0$. In this regard, one can see that the first inequality in \eqref{eq:arbDTLinfty1} implies that \eqref{eq:rangeDTLinfty3_1} holds for all $\theta\ge0$. Since we have proven that the maximum of $D(\theta)$ is at $\theta=0$, then we need to have that $D(0)<0$ together with the first inequality in \eqref{eq:arbDTLinfty1} imply that $D(\theta)<0$ for all $\theta\ge0$. We note now that $D(0)<0$ is exactly the first inequality in \eqref{eq:arbDTLinfty2}. Proceeding similarly for \eqref{eq:rangeDTLinfty3_2} and \eqref{eq:rangeDTLinfty3_3} allows us to show, again, that under the assumption that $A\lambda +E_c\mathds{1}<0$, the maximum for those expressions is attained at $\theta=0$ and, therefore, the conditions must only be satisfied at $\theta=0$, which leads to the second inequalities in both \eqref{eq:arbDTLinfty1} and \eqref{eq:arbDTLinfty2}. This proves the result.
\end{proof}}

\blue{It is also worth mentioning it would have been theoretically possible to simply consider the range dwell-time result with dwell-time values $\theta\in(0,\infty)$ to obtain an arbitrary dwell-time result. However, by doing so, we would have run into technical problems while using an SOS-based approach. To explain this, let us consider the LTI case without performance, that is, we only consider the first conditions of \eqref{eq:rangeDTLinftyZ} and \eqref{eq:rangeDTLinfty} with $E_c=E_d=0$. The solution for $\xi(t)$ solving the differential inequality in \eqref{eq:rangeDTLinftyZ} can be expressed as $\textstyle\xi(\theta)=e^{A\theta}\xi(0)+\epsilon\int_0^\theta e^{As}\mathds{1}\ds$ for any $\epsilon>0$. In this regard, one can observe that the SOS program needs to approximate this function by a polynomial to yield accurate results. Unfortunately, this can only be done on compact sets (this is the Weierstrass approximation theorem) and it is not possible to achieve this on the interval $(0,\infty)$. A second problem is that a polynomial necessarily diverge at infinity while $\xi(\theta)$ must tend to a constant at time goes to infinity otherwise the first condition in \eqref{eq:rangeDTLinfty} cannot hold. This justifies the consideration of the above result.}

\subsection{Computational considerations}

Several methods can be used to check the conditions stated in Theorem \ref{th:rangeDTLinfty}: the piecewise linear discretization approach \cite{Allerhand:11,Xiang:15a,Briat:16c}, Handelman's Theorem \cite{Handelman:88,Briat:11h,Briat:16c} and a sum of squares approach based on Putinar's Positivstellensatz \cite{Putinar:93} and semidefinite programming \cite{Parrilo:00}. \blue{Based on the extensive numerical comparisons between those methods performed in \cite{Briat:16c} in the context of the analysis of linear positive impulsive and switched systems, the method based on SOS programming will be considered here. Indeed, in all the numerical examples discussed in \cite{Briat:16c}, the SOS approach was always the one with the lowest computational complexity and the smallest solving time compared to the two other approaches for a similar accuracy. In other words, best accuracy at a lowest computational cost was always achieved by the SOS approach. The second approach in terms of performance was the one based on Handelman's theorem while the piecewise linear approach was the one that always performed the worst. A common denominator to those approaches is their lack of scalability in the sense that the computational complexity will grow very quickly with the size of the system and the degrees of the polynomials involved, resulting in intractable problems, even for systems of moderate size -- which is also a problem of LMI-based methods in general. Note, however, that all the polynomials are univariate in the present case and, hence, the size of the problem will grow linearly with the degree of the polynomials as opposed to the multivariate case where the number of monomials grows exponentially with the degree. In this regard, the limiting factor will be the size of the system. Possible ways to solve "large" SOS programs exist: one can try to exploit the structure of the problem such as chordal sparsity \cite{Zheng:19} or use (potentially very conservative) linear or second-order cone programming approximations of the underlying semidefinite program \cite{Ahmadi:19}.}

Before stating the main results of this section, we need to introduce few definitions. A multivariate polynomial $p(x)$ is said to be a sum-of-squares (SOS) polynomial if it can be written as $\textstyle p(x)=\sum_{i}q_i(x)^2$ for some polynomials $q_i(x)$. A polynomial matrix $p(x)\in\mathbb{R}^{n\times m}$ is said to \emph{componentwise sum-of-squares} (CSOS) if each of its entries is an SOS polynomial. Checking whether a polynomial is SOS can be exactly cast as a semidefinite program \cite{Parrilo:00} that can be easily solved using semidefinite programming solvers such as SeDuMi \cite{Sturm:01a}. The package SOSTOOLS \cite{sostools3} can be used to formulate and solve SOS programs in a convenient way.

\blue{In the SOS approach the conservatism may come from two different sources: the polynomial approximation of the infinite-dimensional decision variables, the choice of the degrees of the polynomial approximations of the infinite-dimensional decision variables, the degrees of the multipliers of the Positivstellensatz,  and the consideration of SOS constraints as a proxy for nonnegativity constraints on polynomials. We show below that considering polynomials is not restrictive in the present case:}

\blue{\begin{theorem}\label{th:polySOS}
  The following statements are equivalent:
  \begin{enumerate}[(a)]
    \item There exist a continuously differentiable vector-valued function $\zeta:\mathbb{R}\mapsto\mathbb{R}^n$ and a scalar $\gamma$ such that the conditions of Theorem \ref{th:rangeDTLinfty}, (a) hold.
    \item There exist a polynomial vector-valued function $\zeta:\mathbb{R}\mapsto\mathbb{R}^n$ and a scalar $\gamma$ such that the conditions of Theorem \ref{th:rangeDTLinfty}, (a) hold.
  \end{enumerate}
\end{theorem}}
\begin{proof}
  The proof that (b) implies (a) is immediate. To prove the converse, we simply need to prove that polynomials can approximate arbitrarily closely any continuously differentiable function and its derivative over a compact set. We prove that in the univariate case, that is, we pick the compact set $\mathcal{C}:=[a,b]$ and we consider the function $f$ which is differentiable on $(a,b)$. By Weierstrass's approximation Theorem, for any $\epsilon>0$ there exists a polynomial $q_n$ of degree $n$ such that
\begin{equation}
  \sup_{x\in[a,b]}|f'(x)-q_n(x)|\le\epsilon.
\end{equation}
Define $p_{n+1}(x):=f(\alpha)+\int_\alpha^xq_n(s)ds$. Clearly, we have that $p_{n+1}'(x)=q_n(x)$ and
\begin{equation}
  \begin{array}{rcl}
    |f(x)-p(x)| &=&     \left|\int_\alpha^x(f'(s)-q_n(s))ds\right|\le \int_\alpha^x|f'(s)-q_n(s))|\ds\\
                       &\le&      \epsilon|x-y|\le \epsilon(b-a).
  \end{array}
\end{equation}
This proves the result.
\end{proof}

\blue{Assuming the choice of the degrees of the polynomials is free, the second source of conservatism is that polynomial nonnegativity constraints are substituted by SOS constraints as they can be easily checked using semidefinite programming. However, since a univariate polynomial is nonnegative if and only if it a sum of squares, then replacing the nonnegativity constraints by SOS constraints does not introduce any extra conservatism. This leads to the following result:}
\blue{\begin{theorem}\label{prop:SOS1}
The following statements are equivalent:
\begin{enumerate}[(a)]
  \item  There exist a continuously differentiable vector-valued function $\zeta:\mathbb{R}\mapsto\mathbb{R}^n$ and a scalar $\gamma$ such that the conditions of Theorem \ref{th:rangeDTLinfty}, (a) hold.
  \item There exist sufficiently large degrees $d$ and $r$, a polynomial vector $\zeta:\mathbb{R}\mapsto\mathbb{R}^n$, of degree $d$, polynomial vectors $\psi_1,\psi_2:\mathbb{R}\mapsto\mathbb{R}^{n}$, $\psi_3:\mathbb{R}\mapsto\mathbb{R}^{p_c}$, $\psi_4:\mathbb{R}\mapsto\mathbb{R}^{p_d}$ of degree $2r$, and a scalar $\gamma\ge0$ such that
  \begin{enumerate}[(b1)]
    \item $\psi_1(\tau),\psi_2(\tau),\psi_3(\tau)$ and $\psi_4(\tau)$ are CSOS,
    \item $\zeta(0)-\epsilon \mathds{1}\ge0$ (or is CSOS),
    \item $\dot{\zeta}(\tau)-A(\tau)\zeta(\tau)-E_c(\tau)\mathds{1}-f(\tau)\psi_1(\tau)-\epsilon \mathds{1}$    is CSOS,
    \item $-J\zeta(\theta)+\zeta(0)-E_d\mathds{1}-g(\theta)\psi_2(\theta)-\epsilon \mathds{1}$   is CSOS,
     \item $-C_c(\tau)\zeta(\tau)-F_c(\tau)\mathds{1}+\gamma \mathds{1}-f(\tau)\psi_3(\tau)-\epsilon \mathds{1}$ is CSOS,
     \item $-C_d\zeta(\theta)-F_d\mathds{1}+\gamma \mathds{1}-g(\theta)\psi_4(\theta)-\epsilon \mathds{1}$  is CSOS
  \end{enumerate}
  where $f(\tau):=\tau(\Tmax-\tau)$ and $g(\theta):=(\theta-\Tmin)(\Tmax-\theta)$.
\end{enumerate}
\end{theorem}
\begin{proof}
\textbf{Proof that (b) implies (a).} This is fairly standard and uses the fact that the feasibility of the SOS constraints imply that of the constraints in Theorem \ref{th:rangeDTLinfty}, (a). For instance, (b3) implies that $\dot{\zeta}(\tau)+A(\tau)\zeta(\tau)\mathds{1}+E_c(\tau)\mathds{1}$ is nonpositive over $\tau\in[0,\Tmax]$. The same reasoning applies to the other conditions.\\

\noindent\textbf{Proof that (a) implies (b).} To prove the converse, we first note that by virtue of Theorem \ref{th:polySOS}, one can consider a polynomial vector-valued function $\zeta:\mathbb{R}\mapsto\mathbb{R}^n$ without introducing any conservatism. Then, using the fact that the negativity conditions in Theorem \ref{th:rangeDTLinfty}, (a) can be closed using a small enough $\epsilon>0$. Then, using the fact that a univariate polynomial is nonnegative if and only if the polynomial is a sum of squares yields the result.
\end{proof}}

\blue{It is worth mentioning that necessity is only achieved if we let the degrees of the polynomials to be arbitrary. If they are fixed, then the conditions (b) is only sufficient for (a). In this regard, for any fixed polynomial degree, one can only compute an upper-bound on the actual $L_\infty\times\ell_\infty$ performance level. Note, however, that this bound may be accurate enough as this will be further illustrated in the examples of the next section where the computed performance levels will be shown to be very close to the ones that can be obtained with this method. Numerical simulations will also provide lower bounds for the performance levels, which will turn out to be close to the computed upper-bounds, demonstrating then the accuracy of the approach.}

%

\subsection{Examples}

\begin{example}[Arbitrary and minimum dwell-time]
  Let us consider here the system \eqref{eq:mainsyst2} with the matrices:
\begin{equation}\label{eq:ex1}
\begin{array}{rclcrclcrclcrcl}
    A&=&\begin{bmatrix}
    -1 & 0\\
    1 & -2
  \end{bmatrix}, &&E_c&=&\begin{bmatrix}
    0.1\\
    1.1
  \end{bmatrix},&& J&=&\begin{bmatrix}
0.1 & 1\\
0.1  & 0.1
  \end{bmatrix},&&  E_d&=&\begin{bmatrix}
    0.3\\
    0.3
  \end{bmatrix},\\
    C_c&=&\begin{bmatrix}
    0 & 1
  \end{bmatrix},&& F_c&=&0.3,&& C_d&=&\begin{bmatrix}
    0 & 1
  \end{bmatrix},&& F_d&=&0.03.
\end{array}
\end{equation}
\textbf{Arbitrary dwell-time.} Solving for the conditions of Theorem \ref{th:arbDTLinfty}, we get the value $\gamma=1.925$ as the minimum upper-bound for the hybrid $L_\infty\times\ell_\infty$-gain. The SOS program has 11/8 primal/dual variables, and solves in 0.73 seconds, including preprocessing. 

\noindent\textbf{Minimum dwell-time.} We compare now to the case where we impose some lower bound on the dwell-time value. To this aim, we evaluate the minimum upper-bound for the hybrid $L_\infty\times\ell_\infty$-gain for several minimum dwell-time values using Theorem \ref{th:minDTLinfty} and Proposition \ref{prop:nodiff} (no constraint on the differentiability of $\zeta(\tau)$ at $\tau=\bar T$). We obtain the results depicted in Figure \ref{fig:gridding} where we can observe the decrease of the gain as the minimum dwell-time increases. We use polynomials of degree 4. This leads to an SOS program having \blue{96/31} primal/dual variables which solves in approximately 1 second, including preprocessing. When the system is purely continuous, the $L_\infty$-gain is given by $-C_cA^{-1}E_c+F_c= 0.9$ which means that we should have that the exact gain should converge to this value as the minimum dwell-time goes to infinity. Conversely, when the minimum dwell-time goes to zero, we should, and actually here, we do recover the value of 1.925 obtained in the arbitrary dwell-time case.

\begin{figure}
  \centering
  \includegraphics[width=0.8\textwidth]{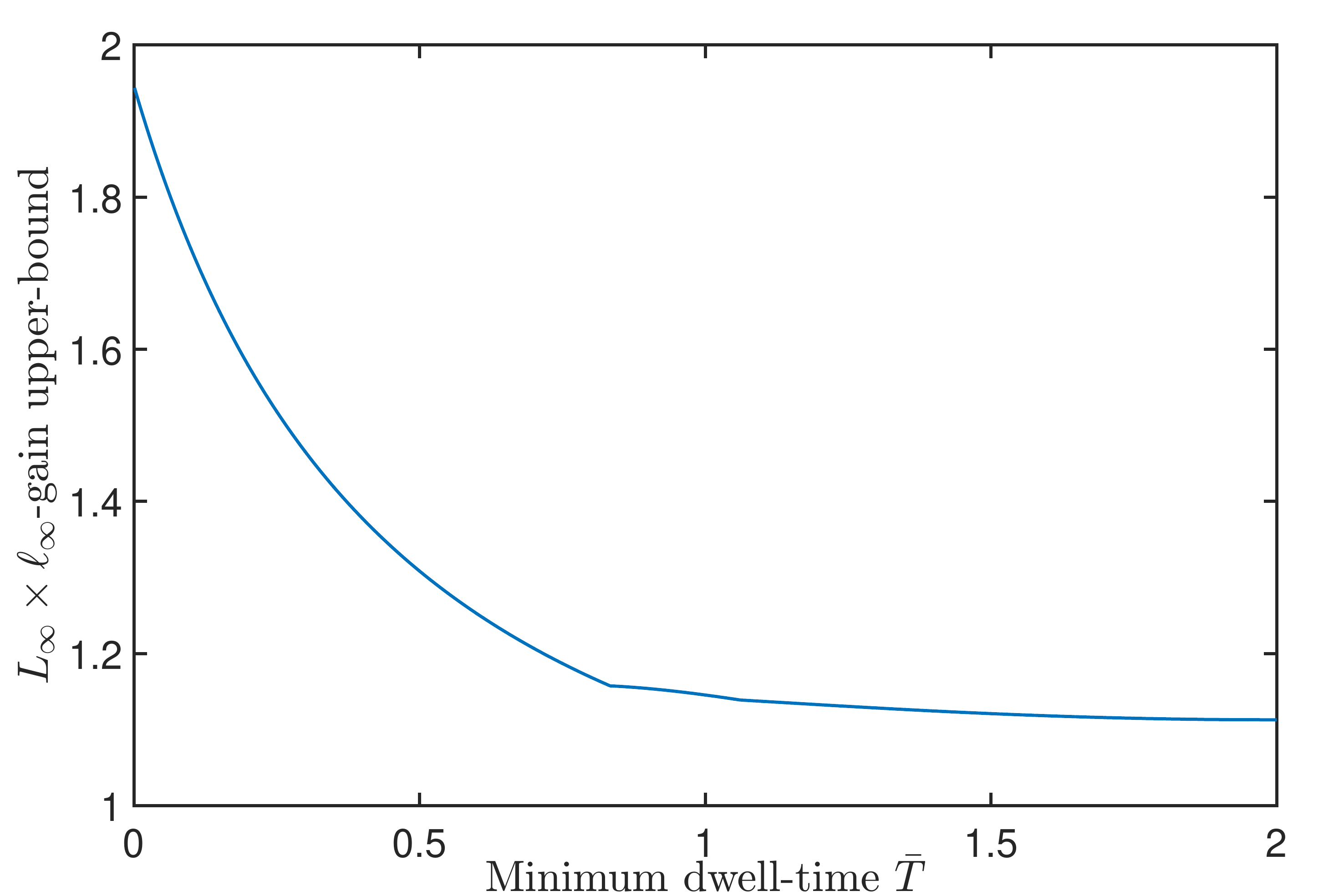}
  \caption{Evolution of the computer upper-bound for the hybrid $L_\infty\times\ell_\infty$-gain of the system \eqref{eq:ex1} as a function of the minimum dwell-time value $\bar T$ computed using Theorem \ref{th:minDTLinfty} using polynomials of degree 4.}\label{fig:gridding}
\end{figure}
\end{example}

\begin{example}[Constant and range dwell-time]
  Let us consider now the system \eqref{eq:mainsyst2} with matrices
\begin{equation}\label{eq:ex2}
\begin{array}{rclcrclcrclcrcl}
    A&=&\begin{bmatrix}
    -2-\tau & 1\\
    0 & 1+\frac{\tau}{2}
  \end{bmatrix},&& E_c&=&\begin{bmatrix}
    0.1+\tau\\
    0.1+\tau^2
  \end{bmatrix},&&J&=&\begin{bmatrix}
1.1 & 0\\
0 & 0.1
  \end{bmatrix},&& E_d&=&\begin{bmatrix}
    0.3\\
    0.3
  \end{bmatrix},\\
  C_c&=&\begin{bmatrix}
    0 & 1+\tau^2/2
  \end{bmatrix}, &&F_c&=&0.1(1+\tau), &&C_d&=&\begin{bmatrix}
    0 & 1
  \end{bmatrix},&&F_d&=&0.1.
\end{array}
\end{equation}
\noindent\textbf{Constant dwell-time.} Using polynomials of degree 4 and solving for the conditions of Theorem \ref{th:cstDTLinfty} with either $\bar T=0.3$ or $\bar T=0.5$. We get the values $\gamma=0.70386$ and $\gamma=1.0517$ as the minimum upper-bound for the hybrid $L_\infty\times\ell_\infty$-gain, respectively. In both cases, the SOS program has 93/28 primal/dual variables, and solves in 2.01 seconds, including preprocessing. Simulating the system 100 times with $x_0=0$ together with the inputs $w_c\equiv 1$ and $w_d\equiv 1$, and random dwell-time sequences satisfying the minimum dwell-time condition, we obtain the lower-bound on the hybrid $L_\infty\times\ell_\infty$-gain given by 1.0131, demonstrating the accuracy of the obtained numerical results, in spite of its conservatism.

\noindent\textbf{Range dwell-time.} Using polynomials of degree 4 and solving for the conditions of Theorem \ref{th:rangeDTLinfty} with $\Tmin=0.3$ and $\Tmax=0.5$, we get the value $\gamma=1.2239$ as the minimum upper-bound for the hybrid $L_\infty\times\ell_\infty$-gain. The SOS program has 182/50 primal/dual variables, and solves in 2.91 seconds, including preprocessing. Choosing polynomials of degree 6 yields the value $\gamma=1.0855$ and a solving time of roughly 5 seconds together with 296/64 as the number of primal/dual variables. We can see that this latter value is very close to both the lower bound obtained by simulation and the value obtained in the constant dwell-time case. This demonstrates the accuracy of the approach for this system.
\end{example}

\begin{example}[Constant and minimum dwell-time]
Let us consider here the example from \cite{Briat:13d} to which we add disturbances. The matrices of the system are given by
\begin{equation}
\begin{array}{rclcrclcrclcrcl}
    A&=&\begin{bmatrix}
    -1 & 0\\
    1+\tau & -2-\tau^2
  \end{bmatrix}, &&E_c&=&\begin{bmatrix}
    0.1\\
    0.1+\tau^2
  \end{bmatrix},&& J&=&\begin{bmatrix}
2 & 1\\
1 & 3
  \end{bmatrix},&&  E_d&=&\begin{bmatrix}
    0.3\\
    0.3
  \end{bmatrix},\\
    C_c&=&\begin{bmatrix}
    \tau & 1
  \end{bmatrix},&& F_c&=&0.03+0.1\tau,  &&C_d&=&\begin{bmatrix}
    0 & 1
  \end{bmatrix},&& F_d&=&0.03.
\end{array}
\end{equation}
\noindent\textbf{Constant dwell-time.} Using polynomials of degree 4 and solving for the conditions of Theorem \ref{th:cstDTLinfty} with $\bar T=2$, we get the value $\gamma=3.2278$ as the minimum upper-bound for the hybrid $L_\infty\times\ell_\infty$-gain. The SOS program has 93 primal and 28 dual variables and solves in 1.74 seconds, including preprocessing. Simulating the system 100 times with $x_0=0$ together with the inputs $w_c\equiv 1$ and $w_d\equiv 1$, and random dwell-time sequences satisfying the minimum dwell-time condition, we obtain the lower-bound on the hybrid $L_\infty\times\ell_\infty$-gain given by 3.1951, which suggests that the computed value is rather tight.

\noindent\textbf{Minimum dwell-time.} Using polynomials of degree 4 and solving for the conditions of Theorem \ref{th:minDTLinfty} and Proposition \ref{prop:nodiff} (no constraint on the differentiability of $\zeta(\tau)$ at $\tau=\bar T$) with $\bar T=2$, we get the value $\gamma=3.2364$ as the minimum upper-bound for the hybrid $L_\infty\times\ell_\infty$-gain. The SOS program has 96 primal and 31 dual variables and solves in 3.31 seconds, including preprocessing. The obtained value is very close to both the lower bound obtained by simulation and the value obtained in the constant dwell-time case.
\end{example}

\section{Stabilization with hybrid $L_\infty\times\ell_\infty$-performance of linear positive impulsive systems under dwell-time constraints}\label{sec:Linf_imp_stabz}

This section is devoted to the application of the stability and performance analysis results obtained in the previous sections. We, therefore, derive state-feedback stabilization conditions for timer-dependent linear positive impulsive systems of the form \eqref{eq:mainsyst2} under arbitrary, constant,range and minimum dwell-time constraints.

\subsection{Arbitrary dwell-time}\label{sec:Linf_imp_stabz_arb}

We consider here the following class of state-feedback control law
\begin{equation}\label{eq:arbDTK}
\begin{array}{rcl}
    u_c(t)&=&K_cx(t)\\
    u_d(k)&=&K_dx(t_k)
\end{array}
\end{equation}
where $K_c\in\mathbb{R}^{m_c\times n}$ and $K_d\in\mathbb{R}^{m_d\times n}$ are the gains of the controller to be determined. The rationale for considering such structure is to allow for the derivation of convex design conditions. This leads to the following result:
\begin{theorem}\label{th:stabzcstDTLinft}
Assume that there exist matrices $X\in\mathbb{D}_{\succ0}^n$, $U\in\mathbb{R}^{m_c\times n}$, $U_d\in\mathbb{R}^{m_d\times n}$ and scalars $\eps,\alpha,\gamma>0$ such that the conditions
    \begin{equation}\label{eq:stabz_arb1}
\begin{array}{rl}
            AX+B_cU_c+\alpha I_n&\ge0\\
           JX+B_dU_d&\ge0\\
           C_cX+D_cU_c&\ge0\\
            C_dX+B_dU_d&\ge0
  \end{array}
\end{equation}
           and
 \begin{equation}\label{eq:stabz_arb2}
\begin{array}{rl}
           \left[AX+B_c U_c\right]\mathds{1}+E_c \mathds{1}&\le0\\
           \left[JX+B_dU_d-X+\eps I\right]\mathds{1}+E_d\mathds{1}&\le0\\
           \left[C_c X+D_c U_c\right]\mathds{1}+F_c \mathds{1}-\gamma \mathds{1}&\le0\\
           \left[C_cX+D_dU_d\right]\mathds{1}+F_d\mathds{1}-\gamma \mathds{1}&\le0.
             \end{array}
\end{equation}

Then, the closed-loop system \eqref{eq:mainsyst2}, \eqref{eq:arbDTK} is asymptotically stable and has hybrid $L_\infty\times\ell_\infty$-gain less than $\gamma$ with the controller gains $K_c= U_cX^{-1}$ and $K_d=X^{-1}U_d$.
\end{theorem}
\begin{proof}
  From the diagonal structure of the matrix-valued function $X(\cdot)$, the controller gain changes of variables and Proposition \ref{prop:positive}, we can observe that the inequalities \eqref{eq:stabz_arb1} are readily equivalent to saying that the closed-loop system is positive. Using now the changes of variables $\zeta=:X\mathds{1}$ and the controller gain changes of variables, we get that the feasibility of \eqref{eq:stabz_arb2} is equivalent to saying that the closed-loop system dynamics verifies the stability and performance conditions of Theorem \ref{th:arbDTLinfty} with the same $\zeta$. The proof is completed.
\end{proof}

It seems interesting to point out that less conservative results could be obtained by considering a timer-dependent continuous-time state-feedback controller. However, since polynomial solutions are sought, this will result in arbitrarily large values for the gain when the dwell-time is large. This is the reason why a constant gain may be more interesting from a practical standpoint.

\subsection{Constant dwell-time}\label{sec:Linf_imp_stabz_cst}

 We consider here the following class of state-feedback control law
\begin{equation}\label{eq:cstDTK}
\begin{array}{rcl}
    u_c(t_k+\tau)&=&K_c(\tau)x(t_k+\tau),\ \tau\in(0,\bar{T}]\\
    u_d(k)&=&K_dx(t_k)
\end{array}
\end{equation}
where $K_c:[0,\bar{T}]\mapsto\mathbb{R}^{m_c\times n}$ and $K_d\in\mathbb{R}^{m_d\times n}$ are the gains of the controller to be determined. The rationale for considering such structure is to allow for the derivation of convex design conditions. This leads to the following result:
\begin{theorem}\label{th:stabzcstDTLinfty}
  The following statements are equivalent:
  \begin{enumerate}[(a)]
    \item The closed-loop system \eqref{eq:mainsyst2},\eqref{eq:cstDTK} is exponentially stable under constant dwell-time $\bar{T}$ and the hybrid $L_\infty\times\ell_\infty$-gain of the transfer $(w_c,w_d)\mapsto(z_c,z_d)$ is at most $\gamma$.
    \item There exist a differentiable matrix-valued function $X:[0,\bar{T}]\mapsto\mathbb{D}^n$, $X(0)\succ0$, a matrix-valued function $U_c:[0,\bar{T}]\mapsto\mathbb{R}^{m_c\times n}$, a matrix $U_d\in\mathbb{R}^{m_d\times n}$ and scalars $\eps,\alpha,\gamma>0$ such that the conditions
    \begin{equation}\label{eq:stabz_cst1}
\begin{array}{rl}
            A(\tau)X(\tau)+B_c(\tau)U_c(\tau)+\alpha I_n&\ge0\\
           JX(\bar{T})+B_dU_d&\ge0\\
           C_c(\tau)X(\tau)+D_c(\tau)U_c(\tau)&\ge0\\
            C_dX(\bar{T})+B_dU_d&\ge0
  \end{array}
\end{equation}
           and
 \begin{equation}\label{eq:stabz_cst2}
\begin{array}{rl}
           \left[\dot{X}(\tau)+A(\tau)X(\tau)+B_c(\tau)U_c(\tau)\right]\mathds{1}+E_c(\tau)\mathds{1}&\le0\\
           \left[JX(\bar{T})+B_dU_d-X(0)+\eps I\right]\mathds{1}+E_d\mathds{1}&\le0\\
           \left[C_c(\tau)X(\tau)+D_c(\tau)U_c(\tau)\right]\mathds{1}+F_c(\tau)\mathds{1}-\gamma \mathds{1}&\le0\\
           \left[C_cX(\bar{T})+D_dU_d\right]\mathds{1}+F_d\mathds{1}-\gamma \mathds{1}&\le0
             \end{array}
\end{equation}
hold for all $\tau\in[0,\bar{T}]$. Moreover, the corresponding controller gains are given by $K_c(\tau)= U_c(\tau)X(\tau)^{-1}$ and $K_d=X(\bar{T})^{-1}U_d$.
  \end{enumerate}
\end{theorem}
\begin{proof}
From the diagonal structure of the matrix-valued function $X(\cdot)$, the controller gain changes of variables and Proposition \ref{prop:positive}, we can observe that the inequalities \eqref{eq:stabz_cst1} are readily equivalent to saying that the closed-loop system is positive. Using now the changes of variables $\zeta(\tau)=X(\tau)\mathds{1}$ and the controller gain changes of variables, we get that the feasibility of \eqref{eq:stabz_cst2} is equivalent to saying that the closed-loop system dynamics verifies the constant dwell-time conditions of Theorem \ref{th:cstDTLinfty} with the same $\zeta(\tau)$. The proof is completed.
\end{proof}

\subsection{Range dwell-time}\label{sec:Linf_imp_stabz_range}

We consider here the following class of state-feedback control law
\begin{equation}\label{eq:rDTK}
\begin{array}{rcl}
    u_c(t_k+\tau)&=&K_c(\tau)x(t_k+\tau),\ \tau\in(0,T_k]\\
    u_d(k)&=&K_d(T_{k-1})x(t_k)
\end{array}
\end{equation}
where $K_c:[0,\Tmax]\mapsto\mathbb{R}^{m_c\times n}$ and $K_d:[\Tmin,\Tmax]\mapsto\mathbb{R}^{m_d\times n}$ are the gains of the controller to be determined. The rationale for considering such structure is to allow for the derivation of convex design conditions. This leads to the following result:
\begin{theorem}\label{th:stabzrDT}
Assume that there exist a differentiable matrix-valued function $X:[0,\Tmax]\mapsto\mathbb{D}^n$, $X(0)\succ0$, matrix-valued functions $U_c:[0,\Tmax]\mapsto\mathbb{R}^{m_c\times n}$, $U_d:[\Tmin,\Tmax]\mapsto\mathbb{R}^{m_d\times n}$ and scalars $\eps,\alpha,\gamma>0$ such that the conditions
\begin{equation}\label{eq:th1a}
\begin{array}{rl}
            A(\tau)X(\tau)+B_c(\tau)U_c(\tau)+\alpha I_n&\ge0\\
           JX(\theta)+B_dU_d\theta&\ge0\\
           C_c(\tau)X(\tau)+D_c(\tau)U_c(\tau)&\ge0\\
            C_dX(\theta)+B_dU_d\theta&\ge0
  \end{array}
\end{equation}
           and
 \begin{equation}\label{eq:th1b}
\begin{array}{rl}
           \left[\dot{X}(\tau)+A(\tau)X(\tau)+B_c(\tau)U_c(\tau)\right]\mathds{1}+E_c(\tau)\mathds{1}&\le0\\
           \left[JX(\theta)+B_dU_d(\theta)-X(0)+\eps I\right]\mathds{1}+E_d\mathds{1}&\le0\\
           \left[C_c(\tau)X(\tau)+D_c(\tau)U_c(\tau)\right]\mathds{1}+F_c(\tau)\mathds{1}-\gamma \mathds{1}&\le0\\
           \left[C_dX(\theta)+D_dU_d(\theta)\right]\mathds{1}+F_d\mathds{1}-\gamma \mathds{1}&\le0
             \end{array}
\end{equation}
%
%
        %
hold for all $\tau\in[0,\Tmax]$ and all $\theta\in[\Tmin,\Tmax]$. Then, the closed-loop system \eqref{eq:mainsyst2}, \eqref{eq:rDTK} is asymptotically stable and has hybrid $L_\infty\times\ell_\infty$-gain less than $\gamma$ with the controller gains $K_c(\tau)= U_c(\tau)X(\tau)^{-1}$ and $K_d(\theta)=X(\theta)^{-1}U_d(\theta)$.
\end{theorem}
\begin{proof}
From the diagonal structure of the matrix-valued function $X(\cdot)$, the controller gain changes of variables and Proposition \ref{prop:positive}, we can observe that the inequalities \eqref{eq:th1a} are readily equivalent to saying that the closed-loop system is positive. Using now the changes of variables $\zeta(\tau)=X(\tau)\mathds{1}$ and the controller gain changes of variables, we get that the feasibility of \eqref{eq:th1b} is equivalent to saying that the closed-loop system dynamics verifies the range dwell-time conditions of Theorem \ref{th:rangeDTLinfty}, (a), with the same $\zeta(\tau)$. The proof is completed.
\end{proof}

Interestingly, it is possible to obtain a condition for the design of dwell-time independent controller through the use of Finsler's Lemma or, equivalent, the $S$-procedure. This is stated in the result below:
\begin{theorem}\label{th:stabzrDT2}
Assume that there exist a differentiable matrix-valued function $X:[0,\Tmax]\mapsto\mathbb{D}^n$, $X(0)\succ0$, a matrix-valued function $U_c:[0,\Tmax]\mapsto\mathbb{R}^{m_c\times n}$, some matrices $M\in\mathbb{D}^n_{\succ0}$, $U_d\in\mathbb{R}^{m_d\times n}$ and some scalars $\eps,\alpha,\gamma>0$ such that the conditions $X(\theta)\le M$,
\begin{equation}
\begin{array}{rl}
            A(\tau)X(\tau)+B_c(\tau)U_c(\tau)+\alpha I_n&\ge0\\
           JM+B_dU_d&\ge0\\
           C_c(\tau)X(\tau)+D_c(\tau)U_c(\tau)&\ge0\\
            C_dM+B_dU_d&\ge0
  \end{array}
\end{equation}
           and
 \begin{equation}
\begin{array}{rl}
           \left[\dot{X}(\tau)+A(\tau)X(\tau)+B_c(\tau)U_c(\tau)\right]\mathds{1}+E_c(\tau)\mathds{1}&\le0\\
           \left[JM+B_dU_d(\theta)-X(0)+\eps I\right]\mathds{1}+E_d\mathds{1}&\le0\\
           \left[C_c(\tau)X(\tau)+D_c(\tau)U_c(\tau)\right]\mathds{1}+F_c(\tau)\mathds{1}-\gamma \mathds{1}&\le0\\
           \left[C_dM+D_dU_d(\theta)\right]\mathds{1}+F_d\mathds{1}-\gamma \mathds{1}&\le0
             \end{array}
\end{equation}
hold for all $\tau\in[0,\Tmax]$ and all $\theta\in[\Tmin,\Tmax]$. Then, the closed-loop system \eqref{eq:mainsyst2}, \eqref{eq:rDTK} is asymptotically stable and has hybrid $L_\infty\times\ell_\infty$-gain less than $\gamma$ with the controller gains $K_c(\tau)= U_c(\tau)X(\tau)^{-1}$ and $K_d(\theta)=X(\theta)^{-1}U_d(\theta)$.
\end{theorem}
\begin{proof}
  The proof follows from the same lines as that of Theorem \ref{th:stabzrDT} except that we use Theorem \ref{th:rangeDTLinfty}, (b) with $\mu(\theta)=M\mathds{1}$.
\end{proof}

It seems important to stress that it is unclear how to design a controller that would be independent of the timer variable.

\subsection{Minimum dwell-time}\label{sec:Linf_imp_stabz_min}

We consider here the following class of state-feedback control law
\begin{equation}\label{eq:minDTK}
\begin{array}{rcl}
    u_c(t_k+\tau)&=&\left\{\begin{array}{l}
      K_c(\tau)x(t_k+\tau),\ \tau\in(0,\bar T]\\
      K_c(\bar T)x(t_k+\tau),\ \tau\in(\bar T, T_k]
    \end{array}\right.\\
    u_d(k)&=&K_dx(t_k)
\end{array}
\end{equation}
where $K_c:[0,\Tmax]\mapsto\mathbb{R}^{m_c\times n}$ and $K_d\in\mathbb{R}^{m_d\times n}$ are the gains of the controller to be determined. The rationale for considering such structure is to allow for the derivation of convex design conditions.

\begin{theorem}\label{th:stabzmDT}
There exists a differentiable matrix-valued function $X:[0,\bar T]\mapsto\mathbb{D}^n$, $X(\bar T)\succ0$, a matrix-valued function $U_c:[0,\bar T]\mapsto\mathbb{R}^{m_c\times n}$, a matrix $U_d\in\mathbb{R}^{m_d\times n}$ and scalars $\eps,\alpha,\gamma>0$ such that the conditions
\begin{subequations}\label{eq:th2a}
\begin{alignat}{4}
             A(\tau)X(\tau)+B_c(\tau)U_c(\tau)+\alpha I_n&\ge0\\
           JX(0)+B_dU_d&\ge0\label{eq:th1:2}\\
           C_c(\tau)X(\tau)+D_c(\tau)U_c(\tau)&\ge0\\
            C_dX(0)+B_dU_d&\ge0
  \end{alignat}
\end{subequations}
and
    \begin{subequations}\label{eq:th2b}
\begin{alignat*}{4}
           \left[A(\bar T)X(\bar T)+B_c(\bar T)U_c(\bar T)+\eps I_n\right]\mathds{1}+E_c(\bar T)\mathds{1}&\le0\\
           \left[C_c(\bar T)X(\bar T)+D_c(\bar T)U_c(\bar T)\right]\mathds{1}+F_c(\bar T)\mathds{1}-\gamma \mathds{1}^{\T}&\le0\\
           \left[\dot{X}(\tau)+A(\tau)X(\tau)+B_c(\tau)U_c(\tau)\right]\mathds{1}+E_c(\tau)\mathds{1}&\le0\\
           \left[JX(\bar{T})+B_dU_d-X(0)+\eps I\right]\mathds{1}+E_d\mathds{1}&\le0\\
           \left[C_c(\tau)X(\tau)+D_c(\tau)U_c(\tau)\right]\mathds{1}+F_c(\tau)\mathds{1}-\gamma \mathds{1}&\le0\\
           \left[C_cX(\bar{T})+D_dU_d\right]\mathds{1}+F_d\mathds{1}-\gamma \mathds{1}&\le0
\end{alignat*}
\end{subequations}
hold for all $\tau\in[0,\bar T]$. Then, the closed-loop system \eqref{eq:mainsyst2}, \eqref{eq:minDTK} is asymptotically stable and has hybrid $L_\infty\times\ell_\infty$-gain less than $\gamma$ with the controller gains $K_c(\tau)= U_c(\tau)X(\tau)^{-1}$ and $K_d=X(\bar{T})^{-1}U_d$.
\end{theorem}

\subsection{Examples}

\begin{example}[Constant and range dwell-time]
  Let us consider here the system
\begin{equation}\label{eq:ex2}
\begin{array}{l}
    A=\begin{bmatrix}
    1 & 1\\
    0 & 1
  \end{bmatrix}, B_c=\begin{bmatrix}
    1\\0
  \end{bmatrix},E_c=\begin{bmatrix}
    0.2\\
    0.3
  \end{bmatrix},\\
  J=\begin{bmatrix}
1.2 & 0\\
1 & 0.1
  \end{bmatrix},  B_d=\begin{bmatrix}
    1\\1
  \end{bmatrix},E_d=\begin{bmatrix}
    0.3\\
    0.3
  \end{bmatrix},\\
    C_c=C_d=\begin{bmatrix}
    0 & 1
  \end{bmatrix}, F_c=F_d=0.1,D_c=D_d=0.
\end{array}
\end{equation}
\textbf{Constant dwell-time.} Solving the range dwell-time conditions of Theorem \ref{th:stabzrDT} with polynomials of degree 2 and $\bar{T}=0.1$ and $\bar{T}=0.3$, we find that the hybrid $L_\infty\times\ell_\infty$-gain of the closed-loop system is less than $0.5095$ and $0.69199$, respectively. The number of primal/dual variables is 140/55 and the SOS program solves in 1.65 seconds. For $\bar{T}=0.1$, the controller gains are given by
\begin{equation}
K_c(\tau)=\begin{bmatrix}
\frac {- 1.6244- 0.90003 \tau+ 0.51618 {\tau}^{2}}{ 1.3206+ 1.4047 \tau+ 0.48782 {\tau}^{2}}\\
\frac { 0.36674- 0.36543 \tau- 0.48944 {\tau}^{2}}{ 0.34094+ 0.64094 \tau+ 0.33734 {\tau}^{2}}
  \end{bmatrix}^{\T}\textnormal{ and } K_d=\begin{bmatrix}
    -1 & 0
  \end{bmatrix}
\end{equation}
whereas for $\bar{T}=0.3$, they are given by
\begin{equation}
K_c(\tau)=\begin{bmatrix}
\frac {- 1.4178- 0.48705 \tau+ 0.20764 {\tau}^{2}}{ 0.9062+ 0.73759 \tau+ 0.44814 {\tau}^{2}}\\
\frac { 0.13191- 0.32503 \tau- 0.44001 {\tau}^{2}}{ 0.3593+ 0.6593 \tau+ 0.38782 {\tau}^{2}}
  \end{bmatrix}^{\T}\textnormal{ and } K_d=\begin{bmatrix}
    -1 & 0
  \end{bmatrix}.
\end{equation}

\noindent\textbf{Range dwell-time.} Solving the range dwell-time conditions of Theorem \ref{th:stabzrDT} with $(\Tmin,\Tmax)=(0.1,0.3)$ and polynomials of degree 2, we find that the hybrid $L_\infty\times\ell_\infty$-gain of the closed-loop system is less than $0.69199$ (which is accurate as it is equal to the value obtained in the constant dwell-tine case) with the controller gains
\begin{equation}\label{eq:rangeKc1a}
K_c(\tau)=\begin{bmatrix}
{\frac {- 1.4349- 0.6155 \tau+ 0.29095 {\tau}^{2}}{ 1.0223+ 0.78716 \tau+ 0.39358 {\tau}^{2}}}\\
{\frac { 0.11222- 0.34431 \tau- 0.45596 {\tau}^{2}}{ 0.3593+ 0.6593 \tau+ 0.38782 {\tau}^{2}}}
  \end{bmatrix}^{\T}
\end{equation}
and
\begin{equation}\label{eq:rangeKc1b}
K_d(\theta)=\begin{bmatrix}
{\frac {- 1.0153- 0.80458 \theta- 0.41365 {\theta}^{2}}{ 1.0223+ 0.78716 \theta+ 0.39358 {\theta}^{2}}}&{\frac { 0.008416- 0.02624 \theta- 0.006043 {\theta}^{2}}{ 0.3593+ 0.6593 \theta+ 0.38782 {\theta}^{2}}}
  \end{bmatrix}^{\T}
\end{equation}
The number of primal and dual variables is 247 and 90, respectively. The SOS program solves  in 2.1 seconds. For simulation purposes, we pick $x_0=(4,2)$, $w_c(t)=(1+\sin(t))/2$ and $w_d\sim\mathcal{U}(0,1)$ (i.e. drawn from a uniform distribution). The simulation results are depicted in Fig.~\ref{fig:rangeDT} where we can observe that the system is indeed stabilized.
\begin{figure}
  \centering
  \includegraphics[width=0.8\textwidth]{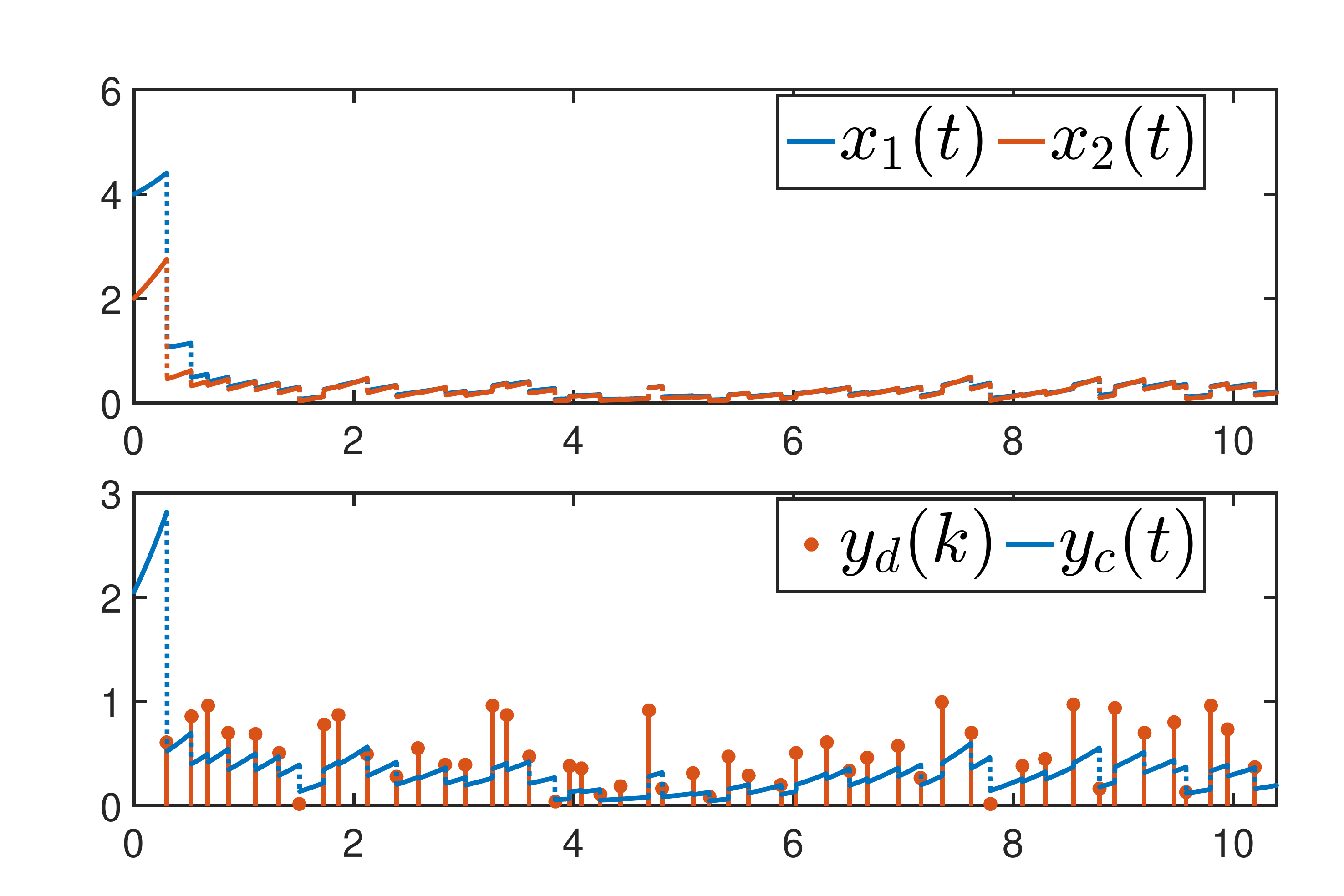}
  \caption{Trajectories of the system \eqref{eq:mainsyst2}, \eqref{eq:ex2}, \eqref{eq:rangeKc1a}, \eqref{eq:rangeKc1b} for some randomly chosen impulse times satisfying the range dwell-time $[0.1,\ 0.3]$.}\label{fig:rangeDT}
\end{figure}
\begin{figure}
  \centering
  \includegraphics[width=0.8\textwidth]{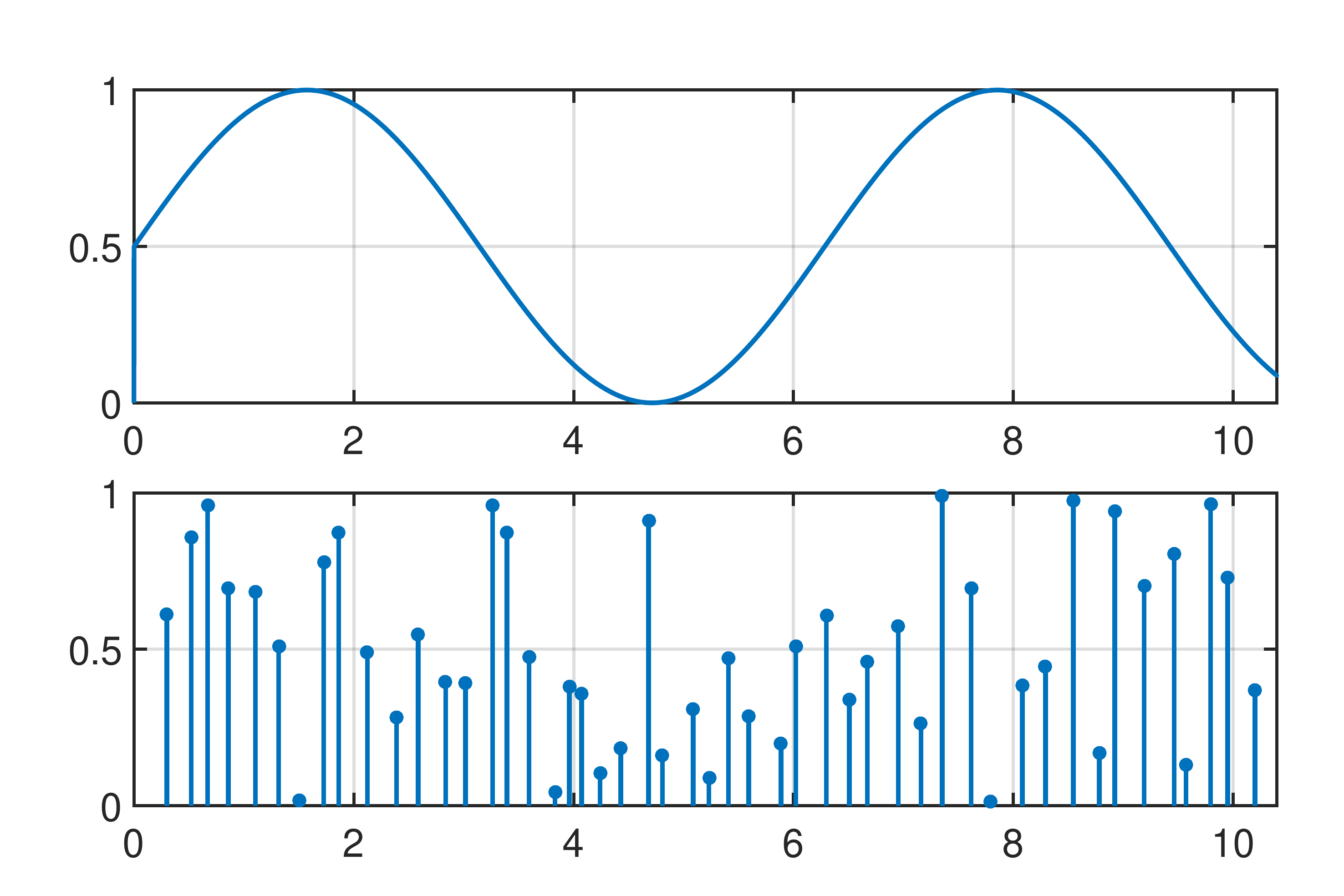}
  \caption{Inputs to the system \eqref{eq:mainsyst2}, \eqref{eq:ex2}, \eqref{eq:rangeKc1a}, \eqref{eq:rangeKc1b}  for some randomly chosen impulse times satisfying the range dwell-time $[0.1,\ 0.3]$.}\label{fig:rangeDTi}
\end{figure}

We now use Theorem \ref{th:stabzrDT2} in order to design a controller where $K_d$ does not depend on the dwell-time. Choosing  $(\Tmin,\Tmax)=(0.1,0.3)$ and polynomials of degree 2, we find that the hybrid $L_\infty\times\ell_\infty$-gain of the closed-loop system is less than $0.69199$, as previously, with the controller gains
\begin{equation}\label{eq:rangeKc2a}
  K_c(\tau)=\begin{bmatrix}
    \dfrac {- 1.2567- 0.30204\tau+ 0.26299{\tau}^{2}}{ 0.7466+ 0.29068\tau+ 0.19336{\tau}^{2}}\\
    \dfrac {- 0.064035- 0.43875\tau- 0.4312{\tau}^{2}}{ 0.3593+ 0.6593\tau+ 0.38782{\tau}^{2}}
  \end{bmatrix}^{\T}\textnormal{ and }K_d=\begin{bmatrix}
  -1 & 0
\end{bmatrix}.
\end{equation}
The number of primal and dual variables is 197 and 74, respectively. The SOS program solves  in 2.1 seconds. We use the same simulation settings as before and we get the results depicted in Figure \ref{fig:rangeDT2} and Figure \ref{fig:rangeDTi2}.
\begin{figure}
  \centering
  \includegraphics[width=0.8\textwidth]{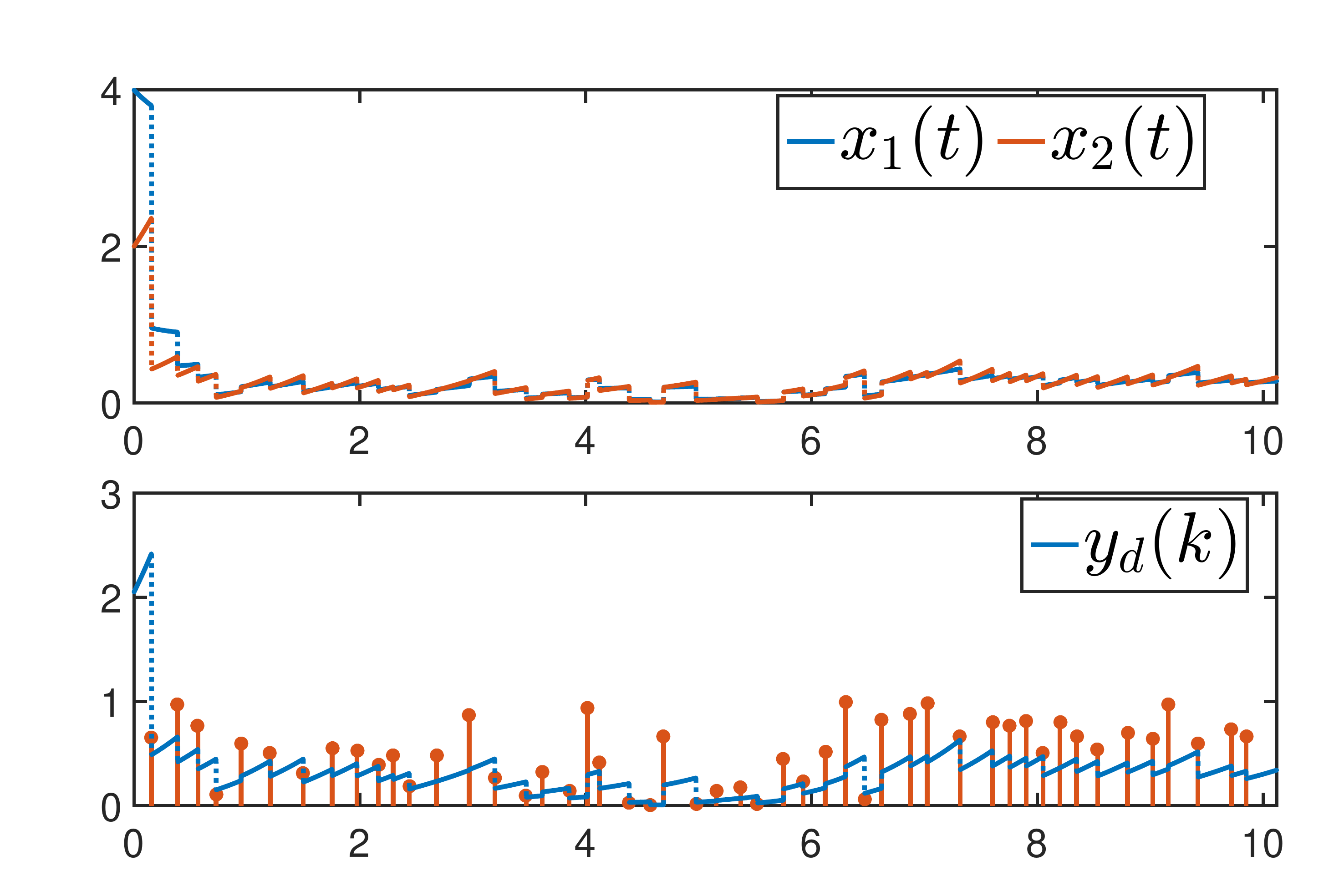}
  \caption{Trajectories of the system \eqref{eq:mainsyst2}, \eqref{eq:ex2}, \eqref{eq:rangeKc2a} for some randomly chosen impulse times satisfying the range dwell-time $[0.1,\ 0.3]$.}\label{fig:rangeDT2}
\end{figure}
\begin{figure}
  \centering
  \includegraphics[width=0.8\textwidth]{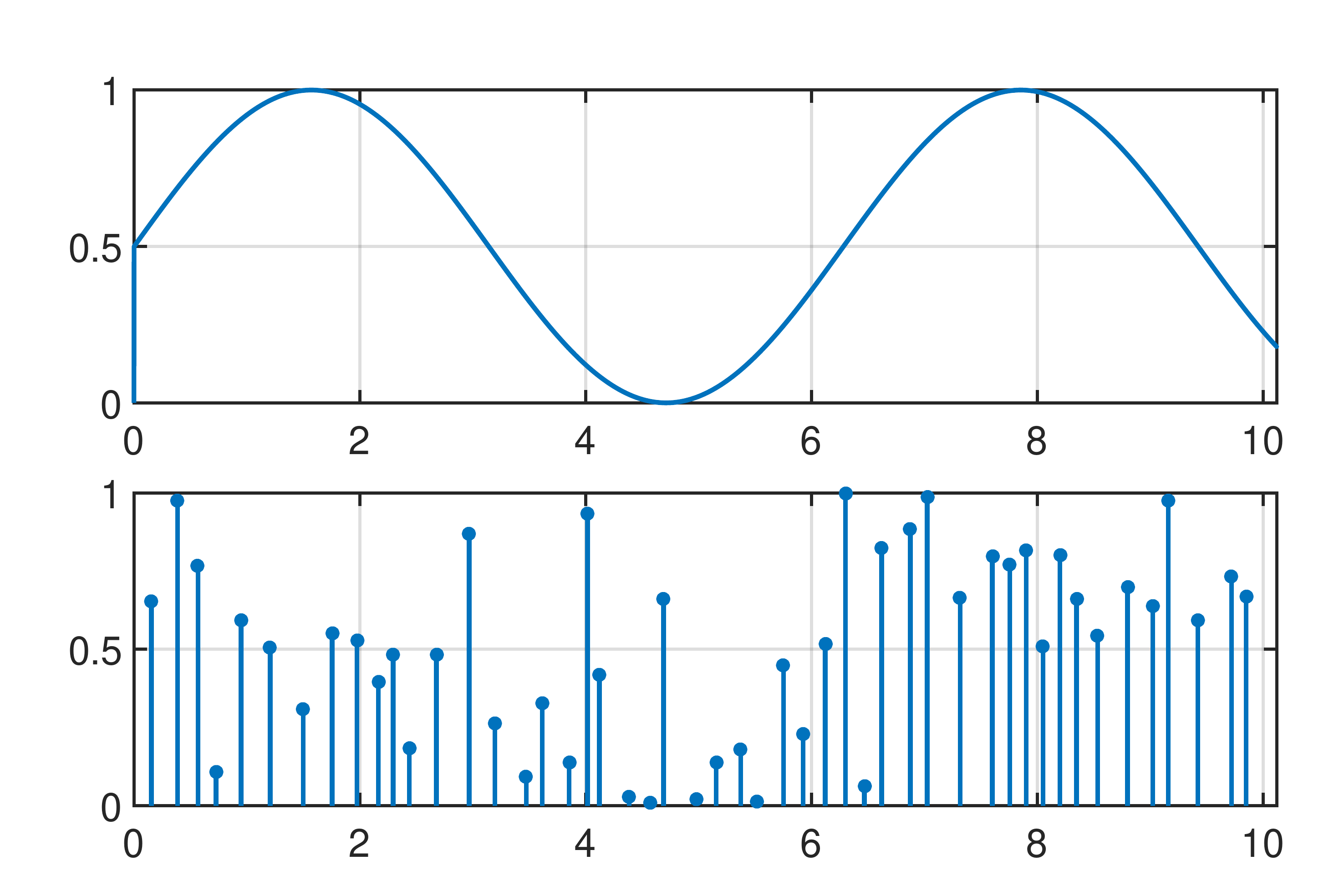}
  \caption{Inputs to the system \eqref{eq:mainsyst2}, \eqref{eq:ex2}, \eqref{eq:rangeKc2a} for some randomly chosen impulse times satisfying the range dwell-time $[0.1,\ 0.3]$.}\label{fig:rangeDTi2}
\end{figure}
\end{example}

\begin{example}[Constant and minimum dwell-time]
Let us consider here the system
\begin{equation}
\begin{array}{l}
    A=\begin{bmatrix}
    3 & -1\\
    2 & -1
  \end{bmatrix}, B_c=\begin{bmatrix}
    1\\0
  \end{bmatrix},E_c=\begin{bmatrix}
    0.5\\
    0.1
  \end{bmatrix},\\
  J=\begin{bmatrix}
2 & 1\\
0 & 1.2
  \end{bmatrix},  B_d=\begin{bmatrix}
    1\\1
  \end{bmatrix},E_d=\begin{bmatrix}
    0.8\\
    0.3
  \end{bmatrix},\\
    C_c=C_d=\begin{bmatrix}
    0 & 1
  \end{bmatrix}, F_c=F_d=0.1, D_c=D_d=0.
\end{array}
\end{equation}
\noindent \textbf{Constant dwell-time.}  Solving the range dwell-time conditions of Theorem \ref{th:stabzrDT} with polynomials of degree 2 and $\bar{T}=0.2$, we find that the hybrid $L_\infty\times\ell_\infty$-gain of the closed-loop system is less than $0.5999$.  The number of primal/dual variables is 140/55 and the SOS program solves in 1.71 seconds. The controller gains are given by
\begin{equation}
K_c(\tau)=\begin{bmatrix}
4.240\frac {- 11.29+ 10.91 \tau+ 0.8850 {\tau}^{2}}{ 3.404- 18.03 \tau+ 5.148 {\tau}^{2}}\\
-\frac { 2.476+ 3.176 \tau- 5.154 {\tau}^{2}}{- 0.3992- 1.307 \tau+ 4.240 {\tau}^{2}}
  \end{bmatrix}^{\T}\textnormal{ and } K_d=\begin{bmatrix}
    0 & -1
  \end{bmatrix}.
\end{equation}

\noindent \textbf{Minimum dwell-time.} Solving the minimum dwell-time conditions of Theorem \ref{th:stabzrDT} with $\bar T=0.2$ and polynomials of degree 2, we find that the hybrid $L_\infty\times\ell_\infty$-gain of the closed-loop system is less than $0.84401$ with the controller gains
\begin{equation}\label{eq:minDT_K}
K_c(\tau)=\begin{bmatrix}
5.6562\frac {- 14.608+ 84.504 \tau- 98.019 {\tau}^{2}}{ 6.6359- 55.759 \tau+ 129.06 {\tau}^{2}}\\
- \frac { 0.93781- 2.5446 \tau+ 8.3343 {\tau}^{2}}{- 0.59204- 1.8543 \tau+ 5.6562 {\tau}^{2}}
  \end{bmatrix}^{\T}\textnormal{ and }K_d = \begin{bmatrix}
  0 & -0.8037
\end{bmatrix}.
\end{equation}
The number of primal and dual variables is 192 and 78, respectively. The SOS program solves  in approximately 4 seconds. For simulation purposes, we pick $x_0=(4,2)$, $w_c(t)=(1+\sin(t))/2$ and $w_d\sim\mathcal{U}(0,1)$ (i.e. drawn from a uniform distribution). The simulation results are depicted in Fig.~\ref{fig:minDT} where we can observe that the system is indeed stabilized.
\begin{figure}
  \centering
  \includegraphics[width=0.8\textwidth]{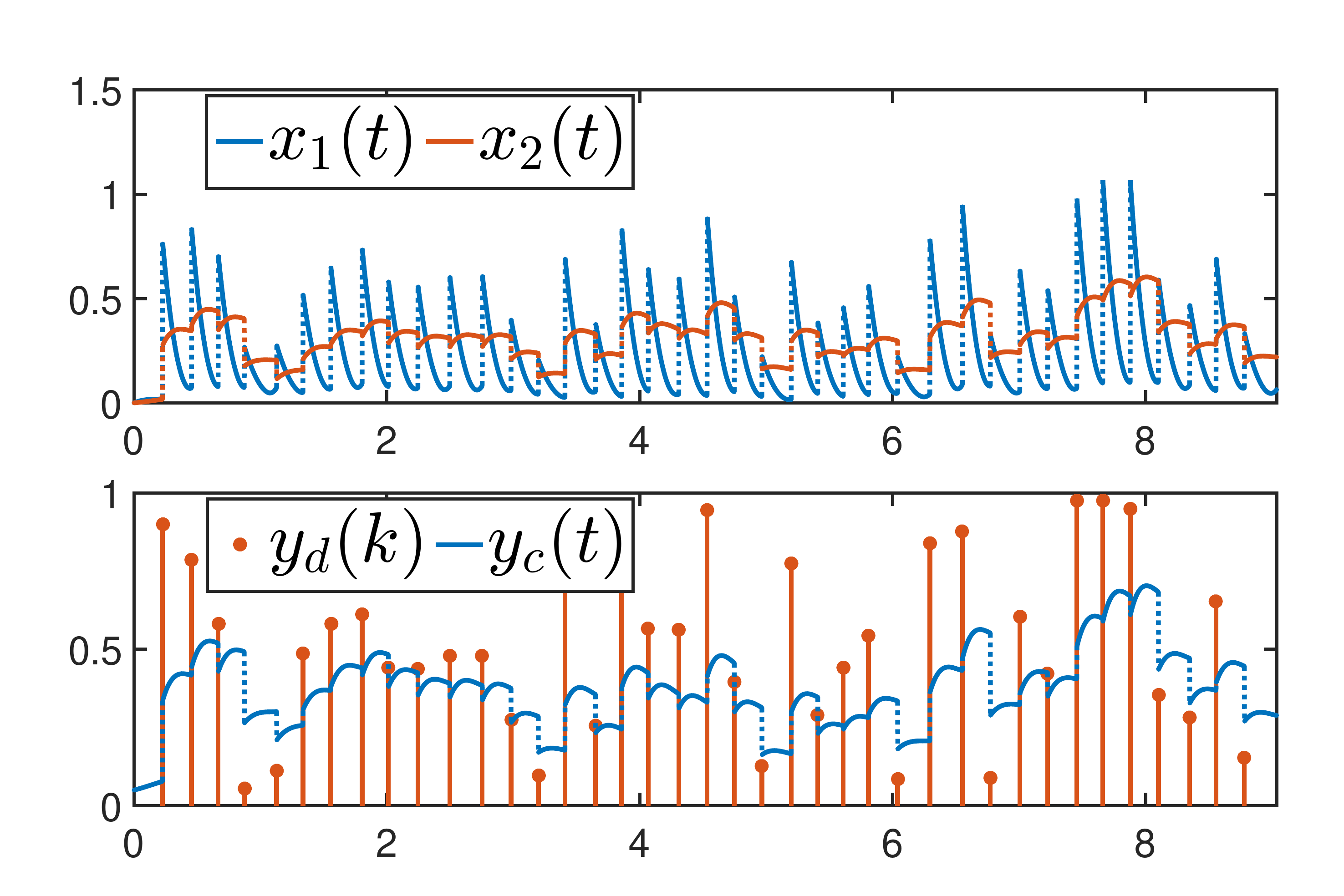}
  \caption{Trajectories of the system \eqref{eq:mainsyst2}, \eqref{eq:ex2}, \eqref{eq:minDT_K} for some randomly chosen impulse times satisfying the minimum dwell-time $\bar T=0.2$.}\label{fig:minDT}
\end{figure}
\begin{figure}
  \centering
  \includegraphics[width=0.8\textwidth]{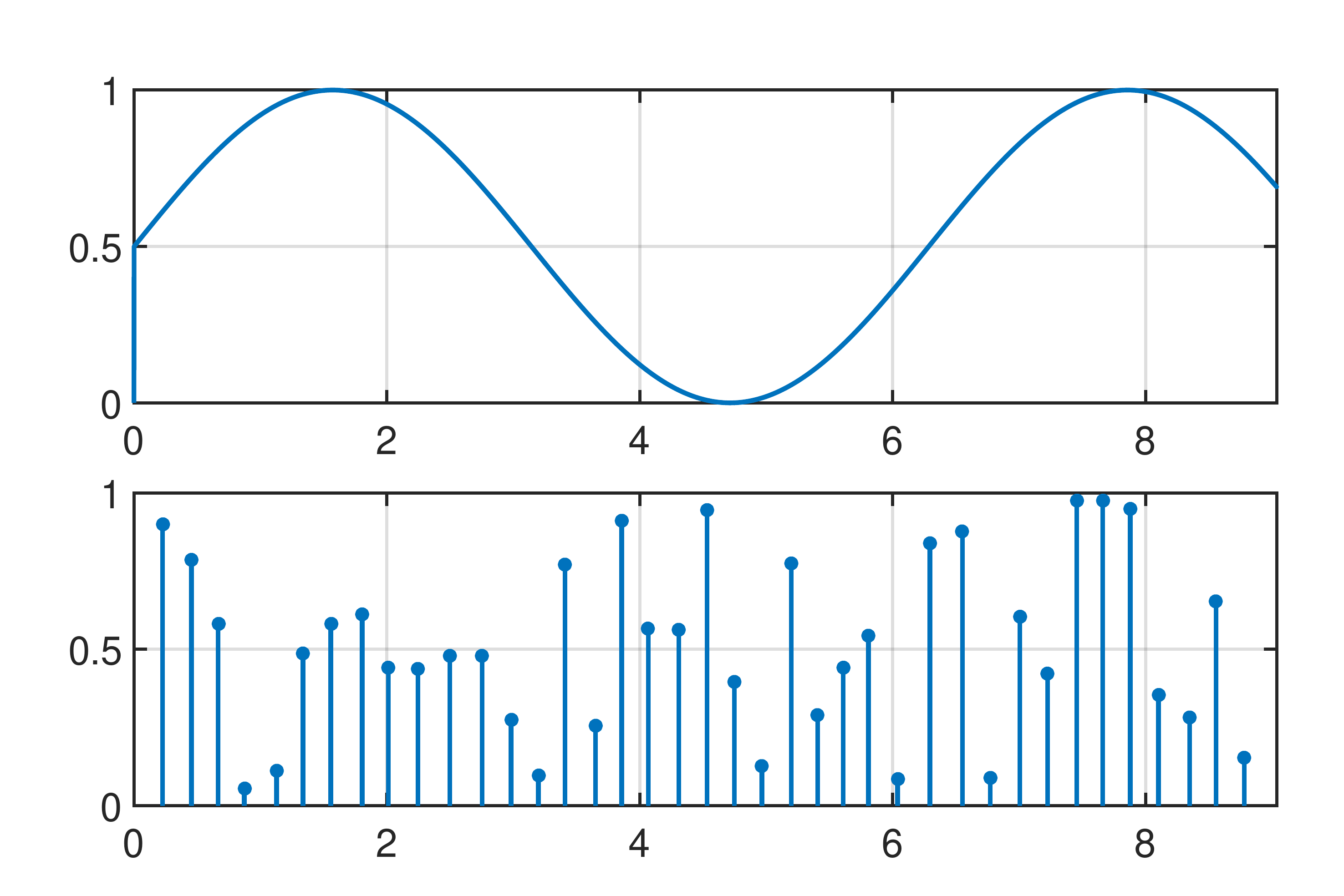}
  \caption{Inputs to the system \eqref{eq:mainsyst2}, \eqref{eq:ex2}, \eqref{eq:minDT_K} for some randomly chosen impulse times satisfying the minimum dwell-time $\bar T=0.2$.}\label{fig:minDTi}
\end{figure}
\end{example}

\section{Stability analysis and stabilization with $L_\infty$-performance of continuous-time linear positive switched systems}\label{sec:Linf_sw}

Let us consider here the switched system
\begin{equation}\label{eq:switched}
\begin{array}{rcl}
\partial_\tau \bar{x}(t_k+\tau) &=& \bar{A}_{\sigma(t_k+\tau)}(\tau)\bar{x}(t_k+\tau)+\bar{B}_{\sigma(t_k+\tau)}(\tau)\bar{u}(t_k+\tau)+\bar{E}_{\sigma(t_k+\tau)}(\tau)w(t_k+\tau)\\
     \bar{z}(t_k+\tau)&=&\bar{C}_{\sigma(t_k+\tau)}(\tau)\bar{x}(t_k+\tau) +\bar{D}_{\sigma(t_k+\tau)}(\tau)\bar{u}(t_k+\tau)+\bar{F}_{\sigma(t_k+\tau)}(\tau)w(t_k+\tau)
\end{array}
\end{equation}
where  $\sigma:\mathbb{R}_{\ge0}\mapsto\{1,\ldots,N\}$ is the switching signal, $\bar x\in\mathbb{R}^n$ is the state of the system, $\bar{w}\in\mathbb{R}^p$ is the exogenous input, $\bar{u}\in\mathbb{R}^m$ is the control input and $\bar{z}\in\mathbb{R}^q$ is the controlled output. The switching signal $\sigma$ is assumed to change values at the times in the sequence $\{t_k\}_{k\in\mathbb{Z}_{\ge1}}$. The above system can be rewritten as the following impulsive system with multiple jump maps
\begin{equation}\label{eq:switchedimp}
\begin{array}{rcl}
  \partial_\tau x(t_k+\tau)&=&A(\tau)x(t_k+\tau)+B(\tau)u(t_k+\tau)+E (\tau)w(t_k+\tau)\\
  z(t_k+\tau)&=&C(\tau)x(t_k+\tau) +D(\tau)u(t_k+\tau)+F(\tau)w(t_k+\tau)\\
  x(t_k^+)&=&(e_ie_j^{\T}\otimes I_n)x(t_k),i,j=1,\ldots,N, i\ne j
\end{array}
\end{equation}
where $\textstyle A(\tau)=\diag_{i=1}^N(\bar A_i)$, $\textstyle B(\tau)=\diag_{i=1}^N(B_i)$, $\textstyle E(\tau)=\col_{i=1}^N(E_i)$, $\textstyle C(\tau)=\diag_{i=1}^N(C_i)$, $\textstyle D(\tau)=\diag_{i=1}^N(D_i)$, and $\textstyle F(\tau)=\col_{i=1}^N(F_i)$. Based on this reformulation, one can apply the results obtained for impulsive systems to switched systems.

\subsection{Stability and performance analysis under minimum dwell-time}

We have the following result:
\begin{theorem}\label{th:minDTLinfty_sw_stab}
  Let us consider the switched system \eqref{eq:switched} with $u_c\equiv0$ and assume that it is internally positive. Then, the following statements are equivalent:
  \begin{enumerate}[(a)]
   \blue{\item There exist some vectors $\lambda_i\in\mathbb{R}^n_{>0}$, $i=1,\ldots,N$, and a scalar $\gamma>0$ such that the conditions
       \begin{equation}\label{eq:minDTLinfty_sw_stab1_1}
          \begin{array}{rcl}
           A_i(\bar{T})\lambda_i+E_i(\bar{T})\mathds{1}&<&0\\
            C_i(\bar{T})\lambda_i+F_i(\bar{T})\mathds{1}-\gamma\mathds{1}&<&0,
            \end{array}
            \end{equation}
 \begin{equation}\label{eq:minDTLinfty_sw_stab1_2}
   \Phi_i(\bar T,0)\lambda_j-\lambda_i+\int_0^{\bar T}\Phi_i(\bar T,s)E_i(s)\mathds{1}\ds<0,
 \end{equation}
 and
 \begin{equation}\label{eq:minDTLinfty_sw_stab1_3}
   C_i(\tau)\left(\Phi_i(\tau,0)\lambda_j+\int_0^{\tau}\Phi_i(\tau,s)E_i(s)\mathds{1}\ds\right)+F_i(\tau)\mathds{1}<\gamma\mathds{1}
 \end{equation}
 hold for all $\tau\in[0,\bar T]$ and all $i,j=1,\ldots,N$, $i\ne j$.}
    \item Assume further that there exist differentiable vector-valued functions $\zeta_i:\mathbb{R}_{\ge0}\mapsto\mathbb{R}^n$, $\zeta_i(0)>0$, $i=1,\ldots,N$, and a scalar $\gamma>0$ such that the conditions
    \begin{equation}
          \begin{array}{rcl}
            -\dot\zeta_i(\tau)+A_i(\tau)\zeta(\tau)+E_i(\tau)\mathds{1}&<&0\\
            C_i(\tau)\zeta(\tau)+F_i(\tau)\mathds{1}-\gamma\mathds{1}&<&0
  \end{array}
  \end{equation}
  and
     \begin{equation}
          \begin{array}{rcl}
           A_i(\bar{T})\zeta_i(\bar{T})+\zeta_i(\bar{T})^{\T}E_i(\bar{T})&<&0\\
            C_i(\bar{T})\zeta(\bar{T})+F_i(\bar{T})\mathds{1}-\gamma\mathds{1}&<&0\\
           \zeta_j(\bar{T})-\zeta_i(0)&\le&0
          \end{array}
        \end{equation}
    hold for all $\tau\in[0,\bar T]$ and all $i,j=1,\ldots,N$, $i\ne j$.
  \end{enumerate}

    Then, the positive switched system \eqref{eq:switched} with $u_c\equiv0$ is uniformly exponentially stable under minimum dwell-time $\bar T$ and the $L_\infty$-gain of the transfer $w_c\mapsto z_c$ is at most $\gamma$.
\end{theorem}
\begin{proof}
  The proof is based on Theorem \ref{th:minDTLinfty} and Remark \ref{rem:relax} together with the system \eqref{eq:switchedimp}. Note that Theorem \ref{th:minDTLinfty} trivially generalizes to the multiple jump maps case by simply considering the discrete-time condition for all values of the jump matrices.
\end{proof}


\blue{For comparison, we recall the following result:
\begin{theorem}[Theorem 4.11 in \cite{Blanchini:15}]\label{th:blanchini}
  Let us consider the switched system \eqref{eq:switched} with $u_c\equiv0$ and assume that it is internally positive and timer independent. Assume further that there exist vectors $\lambda_i\in\mathbb{R}^n$, $i=1,\ldots,N$, and a scalar $\gamma>0$ such that the conditions
  \begin{equation}\label{eq:blanchini1}
    \begin{array}{rcl}
      A_i\lambda_i+E_i\mathds{1}&<&0,\ i=1,\ldots,N,\\
      e^{A_i\bar T}\lambda_j-\lambda_i+\int_0^{\bar T}e^{A_is}E_i\mathds{1}\ds&<&0,\ i,j=1,\ldots, j\ne i,
    \end{array}
  \end{equation}
  and
  \begin{equation}\label{eq:blanchini2}
    C_i\left[\lambda_i+e^{A_i\tau}(\lambda_j-\lambda_i)\right]+F_i\mathds{1}-\gamma\mathds{1}<0,\ i=1,\ldots,N,\tau\in[0,\bar T]
  \end{equation}
hold. Then, the timer-independent version of the positive switched system \eqref{eq:switched} with $u_c\equiv0$ is uniformly exponentially stable under minimum dwell-time $\bar T$ and the $L_\infty$-gain of the transfer $w_c\mapsto z_c$ is at most $\gamma$.
\end{theorem}

In the light of the previous results, this results is essentially a discrete-time stability results and we refer to the discussion below Theorem \ref{th:cstDTLinfty} for the benefits and drawbacks of such a formulation. The main issues are the difficulty for considering them when the matrices of the system are not time-invariant and for design purposes. Those difficulties were the motivation for developing the concepts of looped-functionals \cite{Seuret:12,Briat:11l} and clock-dependent Lyapunov functions \cite{Briat:13d,Briat:16c}.}

\blue{The result below clarifies the relationship between Theorem \ref{th:minDTLinfty_sw_stab} and Theorem \ref{th:blanchini}:
\begin{proposition}\label{prop:equivalence}
   Let us consider the switched system \eqref{eq:switched} with $u_c\equiv0$ and assume that it is internally positive and timer independent. Assume further that the conditions of Theorem \ref{th:blanchini} are satisfied for some $\gamma$. Then, the conditions of Theorem \ref{th:minDTLinfty_sw_stab} hold as well. Moreover, whenever $-A_i^{-1}E_i\mathds{1}>0$ for all $i=1,\ldots,N$, then the converse also holds.
\end{proposition}
\begin{proof}
  Clearly, the first condition in \eqref{eq:minDTLinfty_sw_stab1_1} is readily seen to be identical to the first inequality in \eqref{eq:blanchini1} while is nothing else but the time-invariant counterpart of \eqref{eq:minDTLinfty_sw_stab1_2}. To see that \eqref{eq:blanchini2} implies \eqref{eq:minDTLinfty_sw_stab1_3}, first note that $A_i\lambda_i+E_i\mathds{1}<0$ implies that $e^{A_is}(A_i\lambda_i+E_i\mathds{1})<0$ and, therefore,
  \begin{equation}\label{eq:dkapssdk;l'sk;ldasdk;lsdk;l}
    \int_0^\tau e^{A_is}(A_i\lambda_i+E_i\mathds{1})\ds=(e^{A_i\tau}-I)\lambda_i+\int_0^\tau e^{A_is}E_i\mathds{1}\ds<0
  \end{equation}
  for all $\tau\in[0,\bar T]$, and thus,
  \begin{equation}
    \int_0^\tau e^{A_is}E_i\mathds{1}\ds<(I-e^{A_i\tau})\lambda_i
  \end{equation}
  which, together with \eqref{eq:blanchini2} is readily seen to yield the condition \eqref{eq:minDTLinfty_sw_stab1_3}, which proves the implication.

  We prove now that the converse holds when $-A_i^{-1}E_i\mathds{1}>0$ for all $i=1,\ldots,N$. Letting $\lambda_i=-A_i^{-1}E_i\mathds{1}$, yields
  \begin{equation}
    (e^{A_i\tau}-I)\lambda_i+\int_0^\tau e^{A_is}E_i\mathds{1}\ds=0
  \end{equation}
  for all $\tau\in[0,\bar T]$ and, we have equality between $\int_0^\tau e^{A_is}E_i\mathds{1}\ds$ and $(I-e^{A_i\tau})\lambda_i$. Note, however, that this choice for $\lambda_i$ seems to violate the first inequality of \eqref{eq:minDTLinfty_sw_stab1_1}  and \eqref{eq:blanchini1} because those expressions become 0 instead to being negative. This is not an issue here since the inequality can be closed because of the fact that $-A_i^{-1}E_i\mathds{1}>0$. This is analogous to the fact that one can check the stability of a linear system by solving the Lyapunov equation $A^{\T}P+PA+EE^{\T}=0$ under the condition that the pair $(A,E)$ is controllable even though $EE^{\T}$ is not full-rank.
\end{proof}

From a numerical standpoint, whenever  $-A_i^{-1}E_i\mathds{1}\ngtr0$ for some $i=1,\ldots,N$, we can approach arbitrarily closely the results obtained using Theorem \ref{th:minDTLinfty_sw_stab} by the results of Theorem \ref{th:blanchini} using $\lambda_i=-A_i^{-1}(E_i\mathds{1}+\eps\mathds{1})$ with some sufficiently small $\eps>0$. This shows that, in the end, there is little gap between those results when the matrices of the system are time-invariant. However, the proof of Theorem \ref{th:blanchini} heavily relies on the time-invariance of the system through \eqref{eq:dkapssdk;l'sk;ldasdk;lsdk;l} which is not valid anymore when the matrices of the system become timer-dependent. In this regard, this result is not valid in this case and it is unclear whether it can be generalized to the case of timer-dependent matrices.}

\subsection{Stabilization}

As previously, we consider the following state-feedback controller
\begin{equation}\label{eq:minDTKsw}
\begin{array}{rcl}
    \bar{u}(t_k+\tau)&=&\left\{\begin{array}{l}
      \bar{K}_{\sigma(t_k+\tau)}(\tau)\bar x(t_k+\tau),\ \tau\in(0,\bar T]\\
      \bar{K}_{\sigma(t_k+\tau)}(\bar T)\bar x(t_k+\tau),\ \tau\in(\bar T, T_k]
    \end{array}\right.\\
    u_d(k)&=&K_dx(t_k)
\end{array}
\end{equation}
where $\bar{K}_i:[0,\Tmax]\mapsto\mathbb{R}^{m\times n}$, $i=1,\ldots,N$, are the gains of the controller to be determined. We can therefore build the augmented gain matrix as $K(\tau)=\diag_{i=1}^N(\bar{K}_i)$ in correspondence with the system \eqref{eq:switchedimp}. This leads to the following result:


\begin{theorem}\label{th:minDTLinfty_sw}
Let us consider the system  \eqref{eq:mainsyst2} with $\bar A_i(\tau)=\bar A_i(\bar{T})$, $\bar B_{i}(\tau)=\bar B_{i}(\bar{T})$, $\bar E_{i}(\tau)=\bar E_{i}(\bar{T})$, $\bar C_{i}(\tau)=\bar C_{i}(\bar{T})$, $\bar D_{i}(\tau)=\bar D_{i}(\bar{T})$ and $\bar F_{i}(\tau)=\bar F_{i}(\bar{T})$ for all $\tau\ge\bar{T}>0$, where $\bar T>0$ is given, and assume that it is internally positive. Assume further that there exist a differentiable matrix-valued function $X_i:[0,\bar T]\mapsto\mathbb{D}^n$, $X_i(\bar T)\succ0$, $i=1,\ldots,N$, and scalars $\eps,\gamma,\alpha>0$ such that the conditions
\begin{equation}
     \begin{array}{rcl}
           \bar A_i(\tau)X_i(\tau)+B_i(\tau)U_i(\tau)+\alpha I&\ge&0\\
           \bar C_{i}(\tau)X_i(\tau)+D_i(\tau)U_i(\tau)&\ge&0
          \end{array}
\end{equation}
and
 \begin{equation}
          \begin{array}{rcl}
           \left[\bar A_i(\bar{T})X_i(\bar{T})+B_i(\bar{T})U_i(\bar{T})\right]\mathds{1}+\bar E_{i}(\bar{T})\mathds{1}&<&0\\
           \left [\bar C_{i}(\bar{T})X_i(\bar{T})+B_i(\bar{T})U_i(\bar{T})\right]\mathds{1}+\bar F_{i}(\bar{T})\mathds{1}-\gamma\mathds{1}&<&0\\
              \left[\dot X_i(\tau)+\bar A_i(\tau)X_i(\tau)+B_i(\tau)U_i(\tau)\right]\mathds{1}+\bar E_{i}(\tau)\mathds{1}&<&0\\
            \left[\bar C_{i}(\tau)X_i(\tau)+D_i(\tau)U_i(\tau)\right]\mathds{1}+\bar F_{i}(\tau)\mathds{1}-\gamma\mathds{1}&<&0\\
           X_i(\bar{T})-X_j(0)&\le&0
          \end{array}
        \end{equation}
hold for all $\tau\in[0,\bar T]$ and all $i,j=1,\ldots,N$, $i \ne j$. Then, the closed-loop system \eqref{eq:switched}-\eqref{eq:minDTKsw} is asymptotically stable under minimum dwell-time $\bar T$ with the controller $\bar{K}_i(\tau)=U_i(\tau)X_i(\tau)^{-1}$ and the $L_\infty$-gain of the transfer $w_c\mapsto z_c$ is at most $\gamma$\hfill\mendth
\end{theorem}
\begin{proof}
The proof is an immediate application of Theorem \ref{th:minDTLinfty_sw_stab} and the change of variables used in the minimum dwell-time case.
\end{proof}

\subsection{Example}

  Let us consider the system \eqref{eq:switched} with $\bar u\equiv0$ and the matrices
\begin{equation}\label{eq:ex3}
\begin{array}{l}
    \bar A_1=\begin{bmatrix}
    -1 & 0\\
    1 & -2
  \end{bmatrix}, \bar E_1=\begin{bmatrix}
    0.1\\
    0.1
  \end{bmatrix},\\ \bar A_2=\begin{bmatrix}
-1 & 1\\
1 & -6
  \end{bmatrix}, \bar E_2=\begin{bmatrix}
    0.5\\
    0
  \end{bmatrix},
    \\ \bar C_1=\bar C_2=\begin{bmatrix}
    0 & 1
  \end{bmatrix}, \bar F_1=\bar F_2=0.1.
\end{array}
\end{equation}
Using polynomials of degree 4 and solving for the conditions of Theorem \ref{th:minDTLinfty_sw_stab} with $\bar T=0.1$, we get the value $\gamma=0.50753$ as the minimum upper-bound for the $L_\infty$-gain. The SOS program has 322 primal and 49 dual variables and solves in 4.15 seconds, including preprocessing. Simulating the system 100 times with $x_0=0$ together with the inputs $w_c\equiv 1$ and $w_d\equiv 1$, and random dwell-time sequences satisfying the minimum dwell-time condition, we obtain the lower-bound on the $L_\infty$-gain given by 0.3012, which suggests that the computed value may not be as tight as for the other examples. The simulation results are depicted in Figure \ref{fig:minDTswitched_States} and Figure \ref{fig:minDTswitched_Inputs}.

\blue{It is possible to compare this result with Theorem 4.11 in \cite{Blanchini:15}, recalled here as Theorem \ref{th:blanchini}. Since the conditions in this result are not polynomial in the timer variable, we cannot use SOS programming to verify them directly, even though we could replace the exponential terms by sufficiently accurate polynomial approximations. For simplicity, we use a gridding method here and consider $N_p=101$ equidistant points in the interval $[0,\bar{T}]$, and solve the linear program at those points only (making the approach only necessary and not sufficient; see e.g. \cite{Briat:book1}). The resulting linear program has 5 decision variables, 217 linear constraints, and the program solves in less than a second, which is much better than the proposed approach. Using this result, we obtain 0.50674 as smallest upper-bound for the $L_\infty$-gain for this system, which is pretty much the same as the value obtained using the proposed approach. This is not surprising as Proposition \ref{prop:equivalence} has shown that the two approaches are equivalent when $-A_i^{-1}E_i\mathds{1}>0$ for all $i=1,\ldots,N$, which is the case here.}

\begin{figure}
  \centering
  \includegraphics[width=0.8\textwidth]{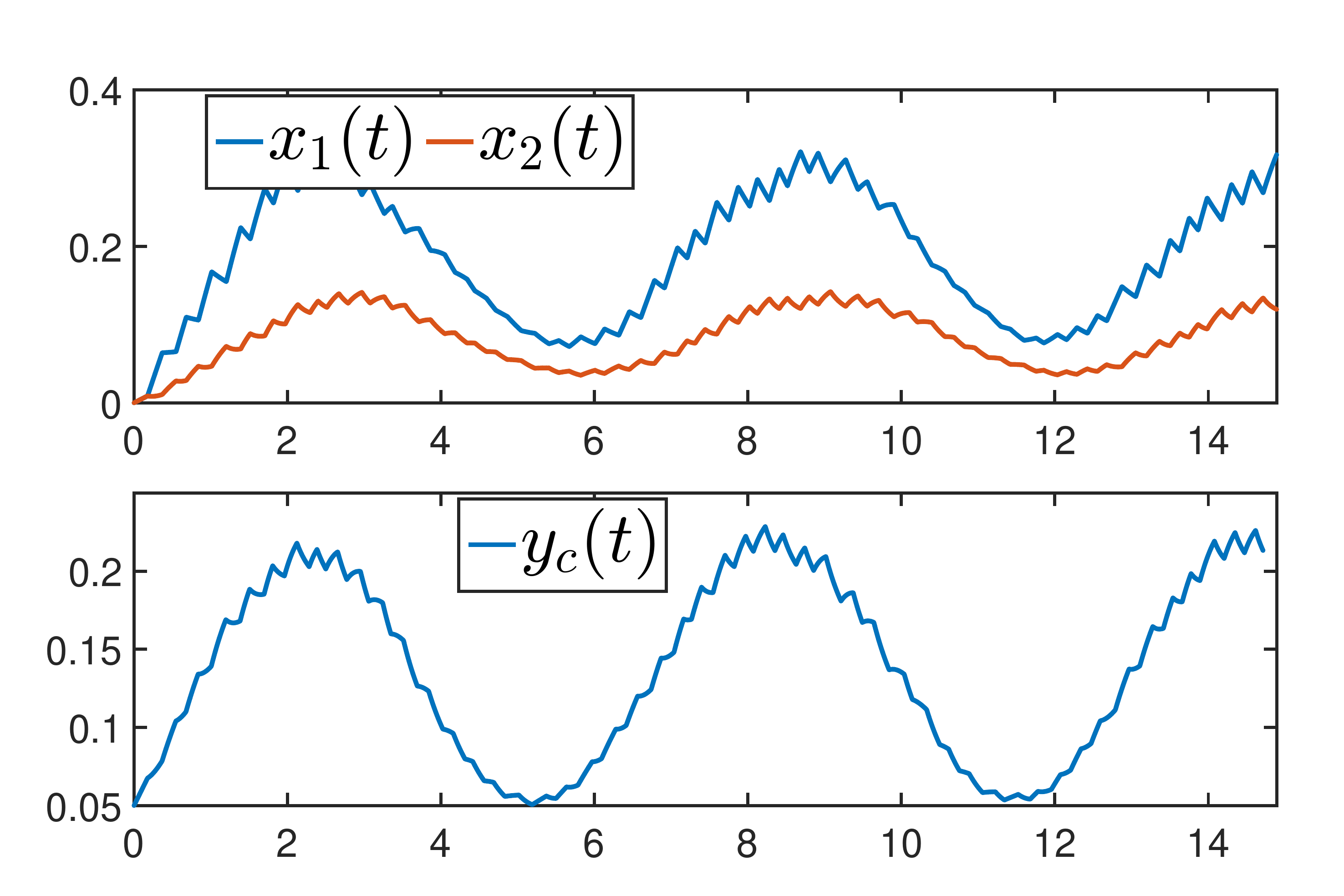}
  \caption{Trajectories of the state (top) and the output (bottom) of the system \eqref{eq:mainsyst2}, \eqref{eq:ex2} and $\bar T=0.1$.}\label{fig:minDTswitched_States}
\end{figure}
\begin{figure}
  \centering
  \includegraphics[width=0.8\textwidth]{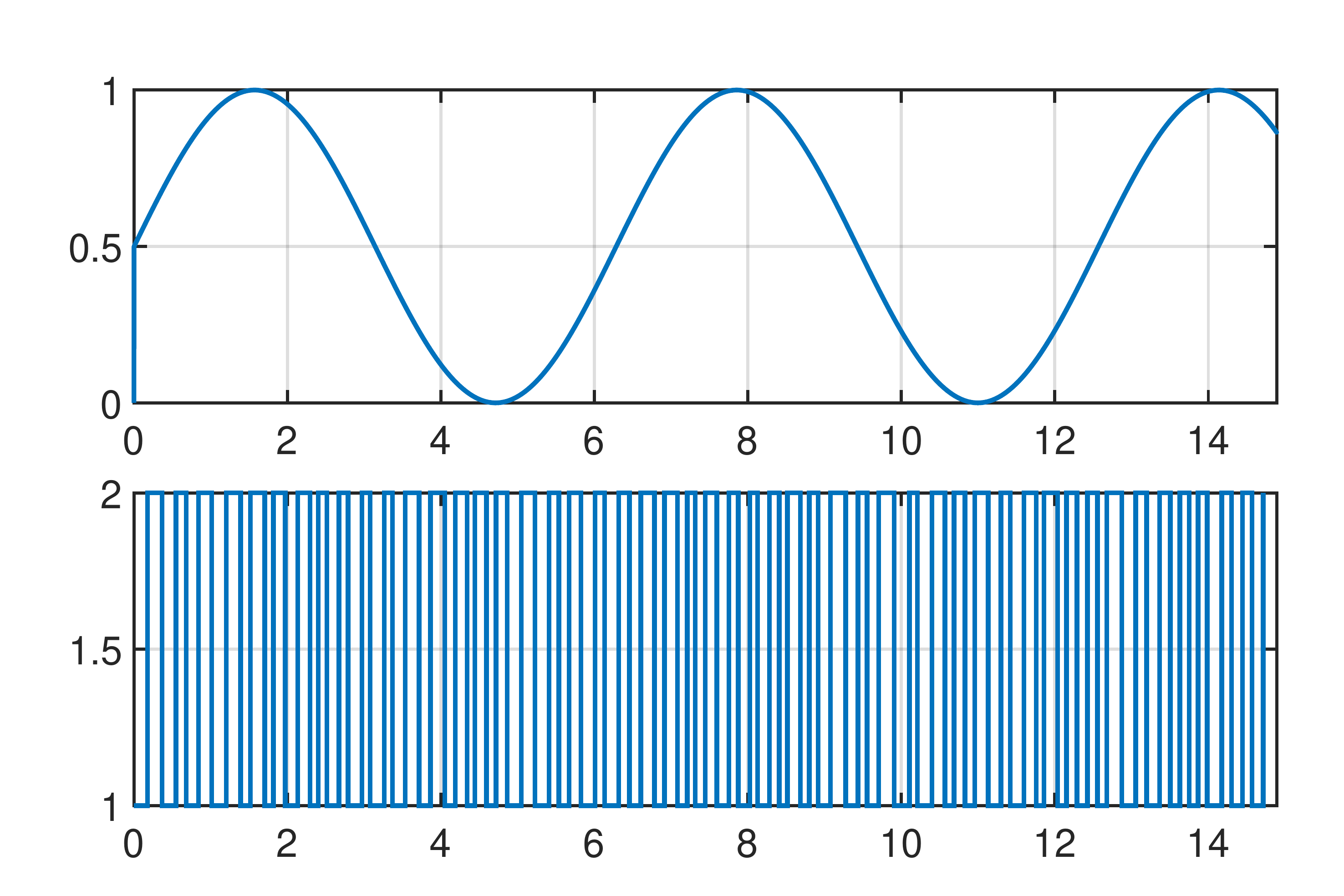}
  \caption{Exogenous input (top) and switching signal (bottom) for the system \eqref{eq:mainsyst2}, \eqref{eq:ex2} with the dwell-time $\bar T=0.1$.}\label{fig:minDTswitched_Inputs}
\end{figure}

%


%

\section{Particular cases}\label{sec:part}

The objective of this section is to show that many of the results in the literature can be recovered by the proposed very general framework provided by Theorem \ref{th:generalLinf}. 

\subsection{LTV, periodic, and LTI continuous-time positive systems}

Let us define here the following continuous-time LTV positive system
\begin{equation}\label{eq:CT}
      \begin{array}{rcl}
        \dot{x}(t)&=&A(t)x(t)+E(t)w(t)\\
        z(t)&=&C(t)x(t)+F(t)w(t)\\
        x(0)&=&x_0>0
      \end{array}
\end{equation}
where the matrices are piecewise continuous and uniformly bounded.

\subsubsection{$L_\infty$ results}

The following result addresses the case of the $L_\infty$ performance continuous-time LTV positive systems that does not seem to have been addressed so far:
\begin{theorem}
  The following statements are equivalent:
  \begin{enumerate}[(a)]
    \item The LTV positive system \eqref{eq:CT} is uniformly exponentially stable and the $L_\infty$-gain of the transfer $w\mapsto z$ is at most  $\gamma$.
    \item There exists a differentiable vector-valued function $\xi:\mathbb{R}_{\ge0}\mapsto\mathbb{R}^n_{>0}$,  $0<\bar\xi_1\le \xi(t)\le \bar{\xi}_2<\infty$  such that the conditions
    \begin{equation}
    \begin{array}{rcl}
       -\dot{\xi}(t)+A(t)\xi(t)+E(t)\mathds{1}&<&0\\
      C(t)\xi(t)+F(t)\mathds{1}-\gamma\mathds{1}&<&0
    \end{array}
    \end{equation}
    hold for all $t\ge t_0$ and all $t_0\in\mathbb{R}_{\ge0}$.
  \end{enumerate}
\end{theorem}
\begin{proof}
  This result is obtained by considering $\tilde A(t)=A(t)$, $\tilde E_c(t)=E(t)$,  $\tilde C_c(t)=C(t)$,  $\tilde F_c(t)=F(t)$, $\tilde{J}(k)=I$, $\tilde{E}_d(k)=\tilde{C}_d(k)=\tilde{F}_d(k)=0$ in Theorem \ref{th:generalLinf}. Following Remark \ref{rem:relax}, we can close the inequality corresponding to jumps since they have no effect on the stability of the system.
\end{proof}

\begin{theorem}
  The following statements are equivalent:
  \begin{enumerate}[(a)]
    \item The $T$-periodic version of the positive system \eqref{eq:CT} is uniformly exponentially stable and the $L_\infty$-gain of the transfer $w\mapsto z$ is at most  $\gamma$.
    \item There exists a differentiable vector-valued function $\xi:[0,T]\mapsto\mathbb{R}^n_{>0}$, $\xi(0)=\xi(T)$, $0<\bar\xi_1\le \xi(t)\le \bar{\xi}_2<\infty$ such that the conditions
    \begin{equation}
    \begin{array}{rcl}
    -\dot{\xi}(t)+A(t)\xi(t)+E(t)\mathds{1}&<&0\\
      C(t)\xi(t)+F(t)\mathds{1}-\gamma\mathds{1}&<&0
    \end{array}
    \end{equation}
    hold for all $t\in[0,T]$.
  \end{enumerate}
\end{theorem}
\begin{proof}
  This result is obtained by considering $\tau=t$ in Theorem \ref{th:cstDTLinfty} together with $E_c=E$, $C_c=C$, $F_c=F$, $J(k)=I$, $E_d(k)=C_d(k)=F_d(k)=0$. Following Remark \ref{rem:relax},  we can close the inequality corresponding to jumps since they have no effect on the stability of the system.
\end{proof}

As a corollary, we recover the result in the LTI case first proposed in \cite{Briat:11g,Briat:11h}:
\begin{corollary}
    The following statements are equivalent:
  \begin{enumerate}[(a)]
    \item The LTI version of the system \eqref{eq:CT} is exponentially stable and the $L_\infty$-gain of the transfer $w\mapsto z$ is at most  $\gamma$.
    \item There exists a vector $\xi\in\mathbb{R}^n_{>0}$ such that the conditions
    \begin{equation}
    \begin{array}{rcl}
       A\xi+E\mathds{1}&<&0\\
      C\xi+F\mathds{1}-\gamma\mathds{1}&<&0.
    \end{array}
    \end{equation}
  \end{enumerate}
\end{corollary}

\subsubsection{$L_1$ results}

The following result addresses the case of the $L_1$-performance for continuous-time LTV positive systems that does not seem to have been addressed so far:
\begin{theorem}
  \begin{enumerate}[(a)]
    \item The LTV positive system \eqref{eq:CT} is uniformly exponentially stable and the $L_1$-gain of the transfer $w\mapsto z$ is at most  $\gamma$.
    \item There exists a differentiable vector-valued function $\chi:\mathbb{R}_{\ge0}\mapsto\mathbb{R}^n_{>0}$,  $0<\bar\chi_1\le \xi(t)\le \bar{\chi}_2<\infty$  such that the conditions
    \begin{equation}
    \begin{array}{rcl}
       \dot{\chi}(t)^{\T}+\chi(t)^{\T}A(t)+\mathds{1}^{\T}C(t)\mathds{1}&<&0\\
      \chi(t)^{\T}E(t)+\mathds{1}^{\T}F(t)-\gamma\mathds{1}^{\T}&<&0
    \end{array}
    \end{equation}
    hold for all $t\ge t_0$ and all $t_0\in\mathbb{R}_{\ge0}$.
  \end{enumerate}
\end{theorem}
\begin{proof}
  This result is obtained by considering $\tilde A(t)=A(t)$, $\tilde E_c(t)=E(t)$,  $\tilde C_c(t)=C(t)$,  $\tilde F_c(t)=F(t)$, $\tilde{J}(k)=I$, $\tilde{E}_d(k)=\tilde{C}_d(k)=\tilde{F}_d(k)=0$ in Theorem \ref{th:generalL1}. Following Remark \ref{rem:relax}, we can close the inequality corresponding to jumps since they have no effect on the stability of the system.
\end{proof}

\begin{theorem}
  The following statements are equivalent:
  \begin{enumerate}[(a)]
    \item The $T$-periodic version of the positive system \eqref{eq:CT} is uniformly exponentially stable and the $L_1$-gain of the transfer $w\mapsto z$ is at most  $\gamma$.
    \item There exists a differentiable vector-valued function $\chi:[0,T]\mapsto\mathbb{R}^n_{>0}$, $\chi(0)=\chi(T)$, $0<\bar\chi_1\le \chi(t)\le \bar{\chi}_2<\infty$ such that the conditions
    \begin{equation}
    \begin{array}{rcl}
    \dot{\chi}(t)^{\T}+\chi(t)^{\T}A(t)+\chi(t)^{\T}C(t)\mathds{1}&<&0\\
      \chi(t)^{\T}E(t)+\mathds{1}^{\T}F(t)-\mathds{1}^{\T}\gamma&<&0
    \end{array}
    \end{equation}
    hold for all $t\in[0,T]$.
  \end{enumerate}
\end{theorem}
\begin{proof}
  This result is obtained by considering $\tau=t$ in the "$L_1\times\ell_1$-version" of Theorem \ref{th:cstDTLinfty} (which is a particular case of the results in \cite{Briat:18:IntImp}) together with $E_c=E$, $C_c=C$, $F_c=F$, $\tilde{J}(k)=I$, $\tilde{E}_d(k)=\tilde{C}_d(k)=\tilde{F}_d(k)=0$. Following Remark \ref{rem:relax}, we can close the inequality corresponding to jumps since they have no effect on the stability of the system.
\end{proof}

As a corollary, we recover the result in the LTI case first proposed in \cite{Briat:11g,Ebihara:11,Briat:11h}:
\begin{corollary}
    The following statements are equivalent:
  \begin{enumerate}[(a)]
    \item The LTI version of the system \eqref{eq:CT} is exponentially stable and the $L_1$-gain of the transfer $w\mapsto z$ is at most  $\gamma$.
    \item There exists a vector $\chi\in\mathbb{R}^n_{>0}$ such that the conditions
    \begin{equation}
    \begin{array}{rcl}
       \chi^{\T}A+\mathds{1}^{\T}C&<&0\\
      \chi^{\T}C+\mathds{1}^{\T}F-\gamma\mathds{1}^{\T}&<&0.
    \end{array}
    \end{equation}
  \end{enumerate}
\end{corollary}

\subsection{LTV, periodic, and LTI  discrete-time positive systems}

Let us define here the following discrete-time LTV positive system
    \begin{equation}\label{eq:DT}
      \begin{array}{rcl}
        x(k+1)&=&A(k)x(k)+E(k)w(k)\\
        z(k)&=&C(k)x(k)+F(k)w(k)\\
        x(0)&=&x_0>0
      \end{array}
    \end{equation}
where the matrices are piecewise continuous and uniformly bounded.

\subsubsection{$\ell_\infty$ results}

As for the continuous-time case, we obtain the following result that does not seem to have been obtained so far in the literature:
\begin{theorem}
  The following statements are equivalent:
  \begin{enumerate}[(a)]
    \item The LTV positive system \eqref{eq:DT} is uniformly exponentially stable and the $\ell_\infty$-gain of the transfer $w\mapsto z$ is at most  $\gamma$.
    \item There exists a vector-valued function $\xi:\mathbb{Z}_{\ge0}\mapsto\mathbb{R}^n_{>0}$,  $0<\bar\xi_1\le \xi(k)\le \bar{\xi}_2<\infty$  such that the conditions
    \begin{equation}
    \begin{array}{rcl}
    A(k)\xi(k)-\xi(k+1)+E(k)\mathds{1}&<&0\\
      C(k)\xi(k)+F(k)\mathds{1}-\gamma\mathds{1}&<&0
    \end{array}
    \end{equation}
    hold for all $k\ge k_0$ and all $k_0\in\mathbb{Z}_{\ge0}$.
  \end{enumerate}
\end{theorem}
\begin{proof}
    This result is obtained by considering $\tilde A(t)=\tilde E_c(t)=\tilde C_c(t)=\tilde F_c(t)=0$, $\tilde{J}(k)=A(k)$, $\tilde{E}_d(k)=E(k)$, $\tilde{C}_d(k)=C(k)$ and $\tilde{F}_d(k)=F(k)$ in Theorem \ref{th:generalLinf}. Following Remark \ref{rem:relax}, we can close the inequality corresponding to the flow since it has no effect on the stability of the system.
\end{proof}

\begin{theorem}
  The following statements are equivalent:
  \begin{enumerate}[(a)]
    \item The linear $T$-periodic version of the positive system \eqref{eq:DT} is uniformly exponentially stable and the $\ell_\infty$-gain of the transfer $w\mapsto z$ is at most  $\gamma$.
    \item There exists a differentiable vector-valued function $\xi:[0,T]\mapsto\mathbb{R}^n_{>0}$, $\xi(0)=\xi(T)$, $0<\bar\xi_1\le \xi(k)\le \bar{\xi}_2<\infty$  such that the conditions
    \begin{equation}
    \begin{array}{rcl}
A(k)\xi(k)-\xi(k+1)+E(k)\mathds{1}&<&0\\
      C(k)\xi(k)+F(k)\mathds{1}-\gamma\mathds{1}&<&0
    \end{array}
    \end{equation}
    hold for all $k\in\{0,\ldots,T-1\}$.
  \end{enumerate}
\end{theorem}
\begin{proof}
  This result can be obtained by considering Theorem \ref{th:generalLinf} and $\tilde A(t)=\tilde E_c(t)=\tilde C_c(t)=\tilde F_c(t)=0$, $\tilde{J}(k)=A(k)$, $\tilde{E}_d(k)=E(k)$, $\tilde{C}_d(k)=C(k)$ and $\tilde{F}_d(k)=F(k)$ in Theorem \ref{th:generalLinf}. Following Remark \ref{rem:relax}, we can close the inequality corresponding to the flow since it has no effect on the stability of the system.
\end{proof}

We have the following corollary in the LTI case \cite{Naghnaeian:14}:
\begin{corollary}
    The following statements are equivalent:
  \begin{enumerate}[(a)]
        \item The LTI version of the positive system \eqref{eq:DT} is uniformly exponentially stable and the $\ell_\infty$-gain of the transfer $w\mapsto z$ is at most  $\gamma$.
    \item There exists a vector $\xi\in\mathbb{R}^n_{>0}$ such that the conditions
    \begin{equation}
    \begin{array}{rcl}
       (A-I_n)\xi+E\mathds{1}&<&0\\
      C\xi+F\mathds{1}-\gamma\mathds{1}&<&0
    \end{array}
    \end{equation}
    hold.
  \end{enumerate}
\end{corollary}

\subsubsection{$\ell_1$ results}

As for the continuous-time case, we obtain the following result that does not seem to have been obtained so far in the literature:
\begin{theorem}
  The following statements are equivalent:
  \begin{enumerate}[(a)]
       \item The LTV positive system \eqref{eq:DT} is uniformly exponentially stable and the $\ell_1$-gain of the transfer $w\mapsto z$ is at most  $\gamma$.
    \item There exists a vector-valued function $\chi:\mathbb{Z}_{\ge0}\mapsto\mathbb{R}^n_{>0}$,  $0<\bar\chi_1\le \chi(k)\le \bar{\chi}_2<\infty$  such that the conditions
    \begin{equation}
    \begin{array}{rcl}
    \chi(k+1)^{\T}A(k)-\chi(k)+\mathds{1}^{\T}C(k)&<&0\\
      \chi(k)^{\T}E(k)+\mathds{1}^{\T}F(k)-\gamma\mathds{1}^{\T}&<&0
    \end{array}
    \end{equation}
    hold for all $k\in\mathbb{Z}_{\ge0}$.
  \end{enumerate}
\end{theorem}
\begin{proof}
    This result is obtained by considering $\tilde A(t)=\tilde E_c(t)=\tilde C_c(t)=\tilde F_c(t)=0$, $\tilde{J}(k)=A(k)$, $\tilde{E}_d(k)=E(k)$, $\tilde{C}_d(k)=C(k)$ and $\tilde{F}_d(k)=F(k)$ in Theorem \ref{th:generalL1}. Following Remark \ref{rem:relax}, we can close the inequality corresponding to the flow since it has no effect on the stability of the system.
\end{proof}

\begin{theorem}
  The following statements are equivalent:
  \begin{enumerate}[(a)]
    \item The linear $T$-periodic version of the positive system \eqref{eq:DT} is uniformly exponentially stable and the $\ell_1$-gain of the transfer $w\mapsto z$ is at most  $\gamma$.
    \item There exists a differentiable vector-valued function $\chi:[0,T]\mapsto\mathbb{R}^n_{>0}$, $\chi(0)=\chi(T)$, $0<\bar\chi_1\le \chi(k)\le \bar{\chi}_2<\infty$  such that the conditions
    \begin{equation}
    \begin{array}{rcl}
    \chi(k+1)^{\T}A(k)-\chi(k)^{\T}+\mathds{1}^{\T}C(k)&<&0\\
     \chi(k)^{\T}E(k)+\mathds{1}^{\T}F(k)-\gamma\mathds{1}^{\T}&<&0
    \end{array}
    \end{equation}
    hold for all $k\in\{0,\ldots,T-1\}$.
  \end{enumerate}
\end{theorem}
\begin{proof}
  This result can be obtained by considering Theorem \ref{th:generalL1} and $\tilde A(t)=\tilde E_c(t)=\tilde C_c(t)=\tilde F_c(t)=0$, $\tilde{J}(k)=A(k)$, $\tilde{E}_d(k)=E(k)$, $\tilde{C}_d(k)=C(k)$ and $\tilde{F}_d(k)=F(k)$ in Theorem \ref{th:generalLinf}. Following Remark \ref{rem:relax}, we can close the inequality corresponding to the flow since it has no effect on the stability of the system.
\end{proof}

We have the following corollary in the LTI case:
\begin{corollary}
    The following statements are equivalent:
  \begin{enumerate}[(a)]
       \item The LTI version of the positive system \eqref{eq:DT} is exponentially stable and the $\ell_1$-gain of the transfer $w\mapsto z$ is at most  $\gamma$.
    \item There exists a vector $\xi\in\mathbb{R}^n_{>0}$ such that the conditions
    \begin{equation}
    \begin{array}{rcl}
    \chi^{\T}A-\chi^{\T}+\mathds{1}^{\T}C&<&0\\
    \chi^{\T}E+\mathds{1}^{\T}F-\gamma\mathds{1}^{\T}&<&0
    \end{array}
    \end{equation}
    hold.
  \end{enumerate}
\end{corollary}

\section{Conclusion}

A novel general necessary and sufficient condition characterizing the exponential stability and the hybrid $L_\infty\times\ell_\infty$ performance of linear positive impulsive systems has been obtained and used to obtain several verifiable stability conditions for linear positive impulsive systems under various dwell-time constraints. Those results have then been exploited to provide a solution to the state-feedback design problem with guaranteed hybrid $L_\infty\times\ell_\infty$ performance level. Results on switched systems are provided using the fact that a switched system can be reformulated as an impulsive system with augmented state-space. It has also been shown that existing results on LTI, LTV and switched systems from the literature can all be formulated as corollaries of the obtained results. Convincing numerical examples have also been presented.

An interesting remaining question concerns the design of interval observers along the same lines as \cite{Briat:15g,Briat:17ifacObs,Briat:18:IntImp} that can minimize the hybrid $L_\infty\times\ell_\infty$-gain of the transfer of the disturbances to the estimation error. It is unclear at this time whether the approach described in \cite{Briat:15g} can be adapted to the current setup.

Finally, the question of designing controllers which are independent of the timer-variable is still open. Indeed, while the approach works on the discrete-time part of the system through the application of Finsler's Lemma or, analogously, of the S-procedure, it fails when applied to the continuous-time part of the system. This is because it leads to conditions taking the form of a linear form in the variables $\dot{x},x$ and $w_c$ where the first one takes indefinite values, unlike the others. Note that when $B$ and $D$ are both nonnegative or nonpositive, then we may possibly replace the timer-dependent controller by a constant one coinciding with its (componentwise) maximal or its minimal values. However, there is no guarantee that the resulting closed-loop system be stable. This interesting problem is left for future research.

%

\end{document}